\documentclass[a4paper,reqno]{amsart}
\usepackage{amsmath,amsthm}
\usepackage{amssymb,stmaryrd,bbm}
\usepackage{hyperref}
\usepackage[capitalize]{cleveref}
\usepackage[all]{xy}
\usepackage{fullpage}
\usepackage[french,english]{babel}
\allowdisplaybreaks

\numberwithin{equation}{section}
\theoremstyle{plain}
\newtheorem{thm}[equation]{Theorem}
\newtheorem{lem}[equation]{Lemma}
\newtheorem{cor}[equation]{Corollary}
\newtheorem{prop}[equation]{Proposition}
\theoremstyle{definition}
\newtheorem{df}[equation]{Definition}
\theoremstyle{remark}
\newtheorem{rem}[equation]{Remark}
\newtheorem{eg}[equation]{Example}
\crefname{thm}{Theorem}{Theorems}
\crefname{lem}{Lemma}{Lemmas}
\crefname{cor}{Corollary}{Corollaries}
\crefname{prop}{Proposition}{Propositions}
\crefname{df}{Definition}{Definitions}
\crefname{rem}{Remark}{Remarks}
\crefname{eg}{Example}{Examples}
\crefname{equation}{}{}

\newcommand{\cA}{\mathcal{A}}
\newcommand{\bC}{\mathbb{C}}
\newcommand{\bK}{\mathbb{K}}
\newcommand{\bM}{\mathbb{M}}
\newcommand{\bN}{\mathbb{N}}
\newcommand{\bQ}{\mathbb{Q}}
\newcommand{\bR}{\mathbb{R}}
\newcommand{\bT}{\mathbb{T}}
\newcommand{\bU}{\mathbb{U}}
\newcommand{\bV}{\mathbb{V}}
\newcommand{\bZ}{\mathbb{Z}}
\newcommand{\cB}{\mathcal{B}}
\newcommand{\cC}{\mathcal{C}}
\newcommand{\cD}{\mathcal{D}}
\newcommand{\cF}{\mathcal{F}}
\newcommand{\cH}{\mathcal{H}}
\newcommand{\cK}{\mathcal{K}}
\newcommand{\cL}{\mathcal{L}}
\newcommand{\cM}{\mathcal{M}}
\newcommand{\cO}{\mathcal{O}}
\newcommand{\cT}{\mathcal{T}}
\newcommand{\cU}{\mathcal{U}}
\newcommand{\cZ}{\mathcal{Z}}

\newcommand{\fp}{\mathfrak{p}}
\newcommand{\fs}{\mathfrak{s}}
\newcommand{\fu}{\mathfrak{u}}
\newcommand{\fv}{\mathfrak{v}}
\newcommand{\fw}{\mathfrak{w}}
\newcommand{\bu}{\mathbbm{u}}
\newcommand{\bv}{\mathbbm{v}}
\newcommand{\bw}{\mathbbm{w}}
\newcommand{\e}{\varepsilon}

\providecommand{\xto}[1]{\xrightarrow{#1}}
\newcommand{\wt}{\widetilde}
\newcommand{\wh}{\widehat}
\newcommand{\bra}{\langle}
\newcommand{\ket}{\rangle}
\newcommand{\oS}{S}%{\overline{S}}
\newcommand{\tS}{\widetilde{S}}%{S}

\newcommand{\id}{\mathrm{id}}
\DeclareMathOperator{\Ker}{Ker}
\DeclareMathOperator{\Coker}{Coker}
\DeclareMathOperator{\Hom}{Hom}
\DeclareMathOperator{\Aut}{Aut}
\DeclareMathOperator{\Rep}{Rep}
\DeclareMathOperator{\Irr}{Irr}
\newcommand{\Ad}{\mathrm{Ad}} %adjoint by unitary
\DeclareMathOperator{\Cone}{Cone}
\DeclareMathOperator{\cspan}{\overline{span}}

\newcommand{\rev}{\mathrm{rev}} %reversed
\newcommand{\ass}{\mathrm{ass}}
\newcommand{\br}{\mathrm{br}}
\newcommand{\tw}{\mathrm{tw}}
\newcommand{\st}{\mathrm{st}}
\renewcommand{\det}{\mathrm{det}}

\DeclareMathOperator{\KK}{KK} %Kasparov category
\DeclareMathOperator{\K}{K} %K-group
\DeclareMathOperator{\Ho}{H} %ordinary cohomology
\newcommand{\Cor}{\mathrm{Corr}} %1-category of C*-correspondences
\newcommand{\Hilb}{\mathrm{Hilb}}

\begin{document}
\title{Actions of tensor categories on Kirchberg algebras
\\\vspace{1em}
Actions de cat{\'e}gories tensorielles sur les alg{\`e}bres de Kirchberg
}
\author{Kan Kitamura}
\address{RIKEN iTHEMS, 2-1 Hirosawa, Wako, Saitama 351-0198 Japan}
\email{kan.kitamura@riken.jp}
\maketitle
\selectlanguage{english}
\begin{abstract}
	We characterize the simplicity of Pimsner algebras for non-proper C*-correspondences. 
	With the aid of this criterion, we give a systematic strategy to produce outer 
	actions of unitary tensor categories on Kirchberg algebras. 
	In particular, every countable unitary tensor category admits an outer action on the Cuntz algebra $\mathcal{O}_2$. 
	We also study the realizability of modules over fusion rings as K-groups of Kirchberg algebras acted on by unitary tensor categories, which turns out to be generically true for every unitary fusion category. 
	Several new examples are provided, among which actions on Cuntz algebras of 3-cocycle twists of cyclic groups are constructed for all possible 3-cohomological classes, thereby answering a question asked by Izumi. 
\end{abstract}
\selectlanguage{french}
\begin{abstract}
	Nous caract{\'e}risons la simplicit{\'e} des alg{\`e}bres de Pimsner pour les C*-correspondances non propres. En utilisant ce crit{\`e}re, nous proposons une strat{\'e}gie syst{\'e}matique pour produire des actions ext{\'e}rieures de cat{\'e}gories tensorielles unitaires sur les alg{\`e}bres de Kirchberg. En particulier, chaque cat{\'e}gorie tensorielle unitaire d{\'e}nombrable admet une action ext{\'e}rieure sur l'alg{\`e}bre de Cuntz $\mathcal{O}_2$.~Nous {\'e}tudions {\'e}galement la r{\'e}alisabilit{\'e} de modules sur des anneaux de fusion en tant que K-groupes d'alg{\`e}bres de Kirchberg soumises {\`a} des actions de cat{\'e}gories tensorielles unitaires, ce qui s'av{\`e}re g{\'e}n{\'e}riquement possible pour chaque cat{\'e}gorie de fusion unitaire. Plusieurs nouveaux exemples sont obtenus, parmi lesquels des actions sur les alg{\`e}bres de Cuntz par des twists de 3-cocycles de groupes cycliques sont construites pour toutes les classes de 3-cohomologie possibles, ce qui r{\'e}sout une question pos{\'e}e par Izumi.
\end{abstract}
\selectlanguage{english}
\setcounter{tocdepth}{1}
\tableofcontents
\section{Introduction}\label{sec:intro}
In the theory of subfactors, initiated by Jones~\cite{Jones1983:index}, it was found that inclusions of operator algebras can give rise to new types of symmetries beyond groups, sometimes called quantum symmetries. 
This discovery led to unexpected connections between operator algebras and other branches of mathematics and physics, such as low-dimensional topology and quantum field theory. 
As concrete realizations of quantum symmetries, the construction and classification of subfactors have been attractive themes of operator algebras. 
Beyond Ocneanu's classification~\cite{Ocneanu:book} of actions of amenable discrete groups on AFD II$_1$ factors, 
a celebrated result of Popa~\cite{Popa1990:classification,Popa1994:classification} shows that amenable inclusions of AFD II$_1$ factors are classified by their standard invariants. 
One of the equivalent ways to formulate this standard invariant uses the notions of actions of tensor categories and Q-systems~\cite{Longo1994:duality,Muger2003:subfactorsI}. 

Recently, Gabe--Szab{\'o}~\cite{Gabe-Szabo:dynamical} made a breakthrough in the classification of amenable actions of groups on purely infinite C*-algebras. 
Their main result is the equivariant generalization of the Kirchberg--Phillips classification theorem \cite{Kirchberg:manuscript,Phillips:Kirchberg-Phillips}, stating that amenable actions of groups on Kirchberg algebras are classified by their equivariant $\KK$-theory. 
Since then, there has been increasing attention to the classification of actions of tensor categories on C*-algebras. 
For example, see \cite{Chen-HernandezPalomares-Jones2024:K-theoretic} in the approximately finite setting and \cite{Arano-Kitamura-Kubota:tensorKK,GironPacheco-Neagu2023:Elliott} for some further attempts. 

In the C*-algebraic case, new obstructions to the existence of certain quantum symmetries have been observed in \cite{Jones2021:remarks,Evington-GironPacheco2023:anomalous,Izumi2023:Gkernels}. 
These obstructions have cohomological or $\K$-theoretic origins, suggesting that C*-algebras can reflect different subtleties of quantum symmetries from those appearing in subfactor theory. 
There are several known constructions of actions of tensor categories on simple nuclear C*-algebras~\cite{Izumi:subalgebrasI,Izumi:subalgebrasII,Evans-Jones2025:quantum} besides those on AF algebras obtained from Ocneanu's compactness argument~\cite{Ocneanu:quantized}. 
However, these constructions often impose severe restrictions on the shapes of resulting C*-algebras, and not so much has been known about the existence of quantum symmetries on given C*-algebras. 

In this paper, we initiate a systematic study of their existence on C*-algebras in the purely infinite case. 
We establish a new strategy to construct actions of tensor categories on Kirchberg algebras based on the following machinery. 
\begin{thm}[\cref{thm:tensorKirch}]\label{main:1}
	For a given action of a countable unitary tensor category $\cC$ on a separable nuclear C*-algebra $A$, there is an outer $\cC$-action on a unital Kirchberg algebra $B$ with a good control of $\K$-theoretic data: more precisely, $B$ is $\cC$-equivariantly $\KK$-equivalent to $A$ in the sense of \cite{Arano-Kitamura-Kubota:tensorKK}. 
\end{thm}

Thanks to the Kirchberg--Phillips theorem~\cite{Kirchberg:manuscript,Phillips:Kirchberg-Phillips}, this machinery allows us to resort to topological or homologically algebraic methods to obtain plenty of quantum symmetries on given Kirchberg algebras. 
We first describe a few consequences for C*-algebraic analogues of questions in subfactor theory. 
\\

In his seminal work~\cite{Jones1983:index}, Jones proved that the indices of inclusions of II$_1$ factors $N\subset M$ must be contained in $\{ 4\cos^2\frac{\pi}{k+2} \mid k\in\bN \}\cup [4,\infty]$, and all of these values can be realized by some $N\subset M$. Moreover, $M$ and $N$ can be taken as AFD factors satisfying $N'\cap M=\bC1_M$ when the index is of the form $4\cos^2\frac{\pi}{k+2}$. In addition, since \cite[Problem 1]{Jones1983:index}, it is a long-standing open problem which indices in $(4,\infty)$ can be realized in this way. 

More generally, it is natural to ask which tensor categories $\cC$ admit an outer action on some AFD factor \cite[Question 1.2]{Brothier-Hartglass-Penneys2012:rigid}, \cite[6.3.1]{Popa2023:W*-representations}. 
(Here, note that any outer $\cC$-action on an AFD factor gives rise to some outer $\cC$-action on the AFD II$_1$ factor by tensoring with the AFD type III$_1$ factor and then passing to its continuous core using \cite{Izumi2003:canonical}, which is of type II$_\infty$ and thus Moria equivalent to the AFD II$_1$ factor.) 
A deep result of Popa and Shlyakhtenko~\cite{Popa1995:axiomatization,Popa-Shlyakhtenko2003:universal} shows that finite index inclusions of $L(\mathbb{F}_\infty)$ can give rise to arbitrary finitely generated tensor categories $\cC$. 
Nevertheless, when it comes to AFD factors, this question is already open even when $\cC$ is the representation category of $\operatorname{SU}_q(2)$ for $0<q<1$, whose affirmative answer would immediately imply the affirmative answer to \cite[Problem 1]{Jones1983:index} above at the index $(q+q^{-1})^2\in (4,\infty)$. 
We consider the analogues of these problems for simple nuclear C*-algebras, which turn out to have affirmative answers. 

\begin{thm}[\cref{thm:outensorO2}, \cref{prop:irrindex}]\label{main:1.5}
	The following hold. 
	\begin{enumerate}
		\item
		Every countable unitary tensor category admits an outer action on the Cuntz algebra $\cO_2$. 
		\item
		For any unital Kirchberg algebra $A$, every value in $[4,\infty)$ can be realized as the Watatani index of some unital $*$-endomorphism $\iota\colon A\to A$ with $\iota(A)'\cap A= \bC1_A$. 
	\end{enumerate}
\end{thm}

In (2), we can choose the larger and smaller C*-algebras to be $*$-isomorphic because of the particular flexibility of the $\K$-theoretic data due to the homological well-behaviour of the representation ring of $\operatorname{SU}_q(2)$. 
However, we cannot hope that (2) holds unconditionally when the index is smaller than $4$. 
This is because the quantum symmetries arising from the inclusions of such indices are closely related to the unitary fusion category of $\operatorname{SU}(2)$ at level $k\in\bN$, denoted by $\cC(\mathfrak{sl}_2,k)$, 
and its fusion ring is homologically more complicated than the representation ring of $\operatorname{SU}_q(2)$. 
This situation imposes some $\K$-theoretic restriction on $A$ for (2) in this case. 
Similarly, (1) is no longer valid for arbitrary Kirchberg algebras due to $\K$-theoretic obstructions. 
To understand these situations, we next investigate the possible $K$-theoretic data coming from actions of fusion categories. 
\\

The first constraint for a tensor category $\cC$ to act on a C*-algebra $A$ is that it must induce $\bZ[\cC]$-module structures on $\K_0(A)$ and $\K_1(A)$. 
Conversely, it is natural to ask when the existence of $\bZ[\cC]$-module structures on the $\K$-groups assures the existence of $\cC$-actions on $A$ or, more elaborately, when given $\bZ[\cC]$-module structures on them lift to some $\cC$-action on $A$. 
Affirmative answers to this question will provide sufficient conditions in algebraic terms of K-theory for the existence of $\cC$-actions on $A$, which is an analytic situation. 
A similar question has been considered in the group case~\cite{Spielberg2007:non-cyclotomic,Katsura2008:construction}, and Katsura~\cite{Benson-Kumjian-Phillips2003:symmetries} shows that this is always true for certain finite groups including all finite cyclic groups. 
Unfortunately, typically $\bZ[\cC]$-modules do not lift to $\cC$-actions on Kirchberg algebras, even when $\cC$ is a non-trivial $3$-cocycle twist of a finite group~\cite{GironPacheco:thesis}. 
Yet, using homologically algebraic methods with the aid of \cref{main:1}, we can still show that this liftability is generically true for arbitrary fusion categories. 

\begin{thm}[\cref{cor:localfusion}]\label{main:3}
	For every unitary fusion category $\cC$, there exists $d\in\bN$ such that, for all countable $\bZ[\cC]$-modules $M_0$ and $M_1$ on which $d$ acts invertibly, there is a Kirchberg algebra $A$ in the UCT class acted on by $\cC$ satisfying $M_0\cong \K_0(A)$ and $M_1\cong \K_1(A)$ as $\bZ[\cC]$-modules. 
\end{thm}
We will also give a computable choice of this constant $d\in\bN$. 
However, this choice is often far from optimal, and determining the smallest value of $d$ is much more difficult in general. 
What we can do at present is the optimization of such $d$ when the fusion category is $\cC(\mathfrak{sl}_2,k)$ or its variants for a certain series of $k$. 
\begin{thm}[\cref{thm:Temperley-Lieb-Jones}, \cref{eg:Temperley-Lieb-Jones2}]\label{main:3.5}
	Let $p$ be an odd prime. 
	Then, we can take $d=1$ in \cref{main:3} when $\cC$ is the even part of $\cC(\mathfrak{sl}_2,p-2)$ or the non-trivial twist of $\cC(\mathfrak{sl}_2,p-2)$. 
	Also, we can take $d=2$ (but cannot take $d=1$) when $\cC$ is $\cC(\mathfrak{sl}_2,p-2)$ itself. 
\end{thm}

The key observation in the proof is that $\bZ[\cC(\mathfrak{sl}_2,p-2)]$ is the group ring of $\bZ/2\bZ$ with the coefficient ring $\bZ[2\cos\frac{2\pi}{p}]$. 
We can obtain the result using certain regularity of the group $\bZ/2\bZ$ with respect to the coefficient ring, which happens to be available thanks to the number-theoretic fact that the Galois extension $\bQ(2\cos\frac{2\pi}{p})/\bQ$ is unramified at $2$. 
Consequently, \cref{main:3,main:3.5} yield many irreducible inclusions of unital Kirchberg algebras whose Watatani indices are $4\cos^2\frac{\pi}{k+2}$ for each $k\in\bN$. 
\\

Next, we consider the problem of $\K$-theoretic liftability in the case when \cref{main:3} might not be applicable. 
In such a case, this problem becomes much more delicate. 
As a first step, we focus on specific tensor categories and C*-algebras and investigate actions of $3$-cocycle twists $(\bZ/m\bZ,\omega)$ of a finite cyclic group on the Cuntz algebra $\cO_{n+1}$. 
Recently, Izumi~\cite{Izumi2023:Gkernels} gave an obstruction for $3$-cocycle twists $(\Gamma,\omega)$ of groups to act on unital Kirchberg algebras $A$ by introducing a new variation of $\K$-groups called $\K_0^\#$ therein. 
More precisely, he observed that if $(\Gamma,\omega)$ admits an action on $A$ by $*$-automorphisms, $\omega$ must lift to a cohomological class in $\Ho^3(\Gamma;\K_0^\#(A))$, which gives some constraint on $\omega$. 
After discussing this constraint in the case of $\Gamma=\bZ/m\bZ$ and $A=\cO_{n+1}$ in \cite[Theorem 3.5]{Izumi2023:Gkernels}, he asked which classes in $\Ho^3(\bZ/m\bZ;\K_0^\#(\cO_{n+1}))$ indeed come from some $(\bZ/m\bZ,\omega)$-actions on $\cO_{n+1}$. 
Analyzing $(\bZ/m\bZ,\omega)$-actions on certain continuous fields of Cuntz algebras combined with the machinery of \cref{main:1}, we answer this question affirmatively. 

\begin{thm}[\cref{thm:cyclicCuntz}, \cref{thm:cyclicCuntzbimod}]\label{main:4}
	For $m,n\geq 1$, any class in $\Ho^3(\bZ/m\bZ;\K_0^\#(\cO_{n+1}))$ lifts to some action of $(\bZ/m\bZ,\omega)$ on $\cO_{n+1}$ by $*$-automorphisms for some $3$-cocycle $\omega$, answering \cite[Problem 3.5]{Izumi2023:Gkernels}. 
	Moreover, for a $3$-cocycle $\omega$ on $\bZ/m\bZ$, the following hold. 
	\begin{enumerate}
		\item
		$(\bZ/m\bZ,\omega)$ admits an action on $\cO_{n+1}$ by $*$-automorphisms if (and only if thanks to \cite[Theorem 3.4]{Izumi2023:Gkernels}) 
		$[\omega]\in n\Ho^3(\bZ/m\bZ;\bR/\bZ)$. 
		\item
		$(\bZ/m\bZ,\omega)$ admits an action on $\cO_{n+1}$ in the tensor categorical sense (i.e., by Morita equivalences) if and only if 
		$[\omega]\in n\Ho^3(\bZ/m\bZ;\bR/\bZ) + a \Ho^3(\bZ/m\bZ;\bR/\bZ)$. 
		Here, $l\in\bN$ is the radical of $m$ (i.e., the product of all prime divisors of $m$), and $a,b\in\bN$ are coprime numbers with $\frac{a}{b}=\frac{2lm}{n}$. 
	\end{enumerate}
\end{thm}
In particular, the latter dichotomy states that when we let $p$ be an odd prime or $2$, fix its power $m$, and increase $n=1,p,p^2,\cdots$, the constraint on $[\omega]$ to admit an action on $\cO_{n+1}$ increases at first but starts to decrease since $n^2\geq pm$ or $4m$ and finally vanishes after $n\geq pm$ or $4m$, respectively. 
Also, the difference of (1) and (2) says that for some triple $(m,n,\omega)$, the trivial $\bZ[\bZ/m\bZ]$-module structure on $\K_0(\cO_{n+1})$ cannot lift to a tensor-categorical action of $(\bZ/m\bZ,\omega)$ (cf.~\cref{lem:corneranomalous}), but some $\bZ[\bZ/m\bZ]$-module structure on $\K_0(\cO_{n+1})$ does. Thus, we cannot apply \cref{main:3} to this setting (since the invertibility of $d$ only refers to the $\bZ$-module structure of $\K$-groups). 
\\

We explain the techniques to show \cref{main:1}. The desired C*-algebra $B$ in the statement is obtained from the original $\cC$-action on $A$ using the Pimsner construction \cite{Pimsner:class,Kumjian:certianCP,Benson-Kumjian-Phillips2003:symmetries,Meyer:classificationgroup}. We need to apply this construction twice to ensure the outerness of the $\cC$-action in addition to the simplicity. 
To show the simplicity of the resulting C*-algebra, we give the following criteria generalizing the particular case due to Kumjian~\cite{Kumjian:certianCP}, which are interesting in themselves. 
\begin{thm}[\cref{thm:simpleCPalg}, \cref{cor:simpleCPalg}]\label{main:2}
	For a non-zero C*-algebra $A$ and a non-degenerate faithful $(A,A)$-correspondence $(E,\phi)$, the following hold. 
	\begin{enumerate}
		\item
		The Toeplitz--Pimsner algebra $\cT_E$ is simple if and only if $\phi(A)\cap\cK(E)=0$ and there is no non-trivial closed ideal $I\subset A$ with $\bra E, \phi(I)E\ket \subset I$. 
		\item
		When $\phi(A)\not\subset\cK(E)$, then the Cuntz--Pimsner algebra $\cO_E$ is simple if and only if there is no non-trivial closed ideal $I\subset A$ with $\bra E, \phi(I)E\ket \subset I \supset \phi^{-1}(\cK(EI))$. 
	\end{enumerate}
\end{thm}
Simplicity of Pimsner algebras has been studied in many settings; see \cite{Kajiwara-Pinzari-Watatani1998:ideal,Muhly-Solel1998:simplicity,Exel-Laca1999:Cuntz-Krieger,Katsura2006:classIII} and references therein. 
Our criteria are obtained using Kishimoto's argument from \cite{Kishimoto:1981outer}, where he gave a simplicity criterion for reduced crossed products of group actions. 
In \cite{Kishimoto:1981outer}, it was crucial to assume that the $*$-automorphisms in question satisfy proper outerness, a technical strengthening of (pointwise) outerness. 
Likewise, the critical step in proving \cref{main:2} is to establish the analogue of this proper outerness. 
In \cite{Kumjian:certianCP}, Kumjian indirectly verified it for a particular form of Pimsner algebras using a result from \cite{Olesen-Pedersen:ConnesII}. 
Alternatively, we do this by a more direct approach starting from an observation that an argument in \cite{Kishimoto:1981outer} applies to our setting. 

Another key ingredient of \cref{main:1} is the outerness of the $\cC$-action. 
We ensure this by giving a criterion for the irreducibility of bimodules similar to \cref{main:2} (see \cref{thm:simpleTPcorr}). 
However, when the action is not given by $*$-automorphisms, there are further difficulties in this step, which causes the strategy above for the simplicity criteria to fail. 
Instead, we use another trick that crucially relies on the unitality and simplicity of the input C*-algebra. 
To apply this trick, we take the Toeplitz--Pimsner construction once again with the aid of Zhang's dichotomy~\cite{Zhang1992:certainI}. 
\\

This paper is organized as follows. 
In \cref{sec:simpleCPalg}, we prove \cref{main:2} on the simplicity of Pimsner algebras. The proof is purely C*-algebraic, and no definitions or results concerning tensor categories will appear in that section. 
In \cref{sec:equivTPalg}, we recall notions about actions of tensor categories and show \cref{main:1} except for the outerness. In \cref{sec:outer}, we establish the remaining part and deduce \cref{main:1.5} (1). 
Then, we focus on the problem of $\K$-theoretic liftability for modules over fusion rings. We first prove \cref{main:1.5} (2) and \cref{main:3} in \cref{sec:fusionmod}. 
Next, we investigate several examples in detail. 
In \cref{sec:TLJ}, we obtain the essential part of \cref{main:3.5} for the variants of $\cC(\mathfrak{sl}_2,k)$. 
The number-theoretic argument required there is separated into \cref{sec:numbertheory}. 
Finally, in \cref{sec:cyclicCuntz}, we study actions of $3$-cocycle twists of $\bZ/m\bZ$ on $\cO_{n+1}$ to show \cref{main:4}. 

\subsection*{Acknowledgments}
This work was supported by RIKEN Special Postdoctoral Researcher Program. 
The author would like to thank Masaki Izumi for bringing his attention to 
\cite[Problem 3.5]{Izumi2023:Gkernels}, which partly motivated this work. Also, he would like to thank Yasuyuki Kawahigashi, Reiji Tomatsu, and Stefaan Vaes for answering his questions on subfactors. He is grateful to Yuki Arano, Sergio Gir{\'o}n Pacheco, and Yosuke Kubota for fruitful discussions on anomalous actions. 
He is grateful to Yuhei Suzuki for pointing out the error in the proof of (i) $\Rightarrow$ (ii) of \cref{thm:simpleCPalg} in the former version of the manuscript. 
He would like to thank the anonymous referees for valuable suggestions that improved the exposition. 

\section{Simplicity of Pimsner algebras}\label{sec:simpleCPalg}
First, we note the general conventions throughout the paper. 
We write the multiplicative group $\bT:=\{ z\in\bC \mid |z|=1 \}$. 
For $n\in\bZ_{\geq 1}$, we let $\bM_n$ be the C*-algebra of $n\times n$-matrices, $\bM_{n^{\infty}}$ be the UHF algebra of type $n^{\infty}$ when $n\geq 2$, and $\bK:=\cK(\ell^2(\bN))$. 
We promise $A^{\otimes 0}:=\bC$ for a C*-algebra $A$. 
For an associative ring $R$, we write $\cZ R$ for the centre of $R$. 
We let $\cspan$ denote the norm closure of the $\bC$-linear span inside a Banach space. 
Abusing notation, when $(X_i)_{i\in I}$ is a family of C*-algebras or Hilbert C*-modules, we write $\bigoplus_{i\in I}X_i$ for the $c_0$- or $\ell^2$-completion of the direct sum, respectively, unless clarified otherwise. 
For an abelian group $X$, we indicate the direct sum of its countably infinitely many copies by $X^{\oplus\infty}$, where we again implicitly take its $c_0$- or $\ell^2$-completion when $X$ is a C*-algebra or a Hilbert C*-module, respectively. 

\subsection{Conventions on C*-correspondences}\label{ssec:C*-corr}
We basically follow \cite{Lance:book} for notation on C*-correspondences. 
For C*-algebras $A$ and $B$, we say an $(A,B)$-correspondence $E=(E,\phi)$, i.e., a pair of a (right) Hilbert $B$-module $E$ and a $*$-homomorphism $\phi\colon A\to \cL(E)$, is 
\begin{itemize}
	\item
	\emph{non-degenerate} if $\cspan{\phi(A)E}=E$, 
	\item
	\emph{faithful} if $\phi\colon A\to \cL(E)$ is injective, 
	\item
	\emph{irreducible} if $\phi(A)'\cap\cL(E)=\bC1_E$, where we often write $1_E$ for $\id_E\in\cL(E)$, 
	\item
	\emph{proper} if $\phi(A)\subset \cK(E)$, and \emph{non-proper} if $E$ is not proper. 
\end{itemize}
For C*-algebras $A,B$, Hilbert $A$-module $E_1$, an $(A,B)$-correspondence $E_2=(E_2,\phi)$, and an element $\xi\in E_1$, we write $T_\xi:=\xi\otimes_{\phi} (-)\in \cL(E_2,E_1\otimes_{\phi} E_2)$. We sometimes write $T_\xi^\phi:=T_\xi$ to make the dependence on $\phi$ explicit. 
When no confusion is likely to arise, we sometimes abuse notation to drop $\phi$ or $A$ from the symbol of interior tensor product $\otimes_{\phi}$ or $\otimes_A$ for readability. 
For an $(A,A)$-correspondence $E$, we promise $E^{\otimes_{A} 0}:=A$. 

We recall some fundamental facts of (augmented) Pimsner algebras from \cite[Section 4.6]{Brown-Ozawa:book}. 
For a non-degenerate faithful $(A,A)$-correspondence $E=(E,\phi)$, 
we define the full Fock space $\cF_E:=\bigoplus_{n=0}^{\infty}E^{\otimes_A n}$, which is a faithful non-degenerate $(A,A)$-correspondence with the left action denoted by $\phi^{(\infty)}\colon A\to\cL(\cF_E)$. 
We define the Toeplitz--Pimsner algebra by $\cT_E:=\mathrm{C}^*( a, T_\xi^{\phi^{(\infty)}} \mid a\in A, \xi\in E )\subset \cL(\cF_E)$. 
Note that $T_\xi^{\phi^{(\infty)}}\in\cT_E$ for all $\xi\in E^{\otimes_A n}$ with $n\geq 0$. 
We usually write $\tS_\xi:=T_\xi^{\phi^{(\infty)}}\in\cT_E$ and regard $A\subset \cL(\cF_E)$ by abusing notation. 
For the closed ideal $I_E:=\phi^{-1}(\cK(E))\subset A$, it holds $\cK(\cF_{E}I_E)\subset\cT_E$ (see \cite[Lemma 4.6.7]{Brown-Ozawa:book}), 
and we define the Cuntz--Pimsner algebra by $\cO_E:=\cT_E/\cK(\cF_EI_E)$. 
We write $\sigma\colon\cT_E\to \cO_E$ for the quotient map and $\oS_\xi:=\sigma(\tS_\xi)\in\cO_E$ for $n\geq 0$ and $\xi\in E^{\otimes_A n}$. 
We have the following. 
\begin{itemize}
	\item
	If $\phi(A)\cap \cK(E)=0$, then $\cT_E=\cO_E$ since $\cK(\cF_EI_E)=0$ by definition. 
	\item
	If $A$ is nuclear, then $\cT_E$ and $\cO_E$ are nuclear. See \cite[Theorem 4.6.25]{Brown-Ozawa:book}. 
\end{itemize}
We write $\gamma$ for the gauge action of $\bT$ on $\cO_E$, which is determined by $\gamma_z(a)=a$ and $\gamma_z(\oS_\xi)=z\oS_\xi$ for $a\in A$, $\xi\in E$, and $z\in\bT$. 

\begin{rem}\label{notn:simpleCPalg}
	For a non-zero C*-algebra $A$ and a non-degenerate faithful $(A,A)$-correspondence $E=(E,\phi)$, 
	we will use the following conventions during the proof of the simplicity criterion. 
	\begin{enumerate}
		\item
		We put $X_n:=\{ x\in\cO_E \mid \gamma_z(x)=z^nx,\, \forall z\in\bT \}$ for $n\in\bZ$ and $A_\infty:=X_0=\cO_E^\gamma$. 
		Note that $X_n^*=X_{-n}$ and that $X_n=\cspan \{ \oS_\xi A_\infty \,|\, \xi\in E^{\otimes_A n} \} \subset \cO_E$ for $n\geq 0$. 
		We regard each $X_n$ as a $(A_\infty,A_\infty)$-correspondence with the inner product $\bra x,y\ket:=x^*y$, where the left and right actions of $A_\infty$ are given by the multiplication of $\cO_E$. 
		Since $E$ is non-degenerate, $X_n$ is also non-degenerate. If $E$ is full, then so is $X_n$ for $n\geq 0$. 
		\item
		We define the C*-subalgebras $A_n:=\cspan \{ \oS_\eta \oS_\xi^* \mid \xi,\eta\in E^{\otimes_A n} \}\subset A_\infty$ for $n\geq 0$. 
		We put $A_{\leq n}:=\sum_{k=0}^{n}A_k\subset A_\infty$ and $A_{\geq n}:=\cspan\bigcup_{k=n}^{\infty}A_k\subset A_\infty$. 
		Then $(A_{\leq n})_{n\geq 0}$ is an increasing sequence of non-degenerate C*-subalgebras of $A_\infty$ such that $\bigcup_{n=0}^{\infty}A_{\leq n}\subset A_\infty$ is a dense $*$-subalgebra, and $(A_{\geq n})_{n\geq 0}$ is a decreasing sequence of closed ideals of $A_\infty$. 
		\item
		We write $e_n$ for the minimal projection of $c_0(\bZ)$ corresponding to each $n\in\bZ$. 
		By slightly abusing notation, we regard $\cO_E$ and $c_0(\bZ)$ as C*-subalgebras of $\cM(\cO_E\rtimes_\gamma\bT)$ such that $e_n x = x e_{m+n}$ for $m,n\in\bZ$, $x\in X_m$. 
		The dual action of $\bZ$ on $\cO_E\rtimes_\gamma\bT$ is generated by $\wh{\gamma}\in\Aut(\cO_E\rtimes_\gamma\bT)$ such that $\wh{\gamma}(x)=x$ and $\wh{\gamma}(e_n)=e_{n+1}$ for $n\in\bZ$, $x\in \cO_E$. 
		\item
		We write $P_0\in\cL(\cF_E)$ for the projection onto the summand $A=E^{\otimes_A 0}$. Then the map $P_0(-)P_0$ induces a conditional expectation for the inclusion $A\subset\cT_E$. 
	\end{enumerate}
\end{rem}

We record the following observations for convenience of the reader. 

\begin{lem}\label{lem:faithfulX}
	In the setting of \cref{notn:simpleCPalg}, we have the following. 
	\begin{enumerate}
		\item
		For $n\in \bZ_{\geq 0}$, there is a well-defined $*$-isomorphism 
		$\sigma_n\colon\cK(E^{\otimes_A n})\to A_n$ such that 
		$\sigma_n(\eta T_\xi^*):=\oS_\eta \oS_\xi^*$ for all $\xi,\eta\in E^{\otimes_A n}$. 
		In particular, $A\cong A_0$. 
		\item
		For $n\in \bZ_{\geq 0}$, the closed ideal $A_n\subset A_{\leq n}$ is essential. 
		Moreover, $X_1$ is a faithful $(A_\infty,A_\infty)$-correspondence. 
	\end{enumerate}
\end{lem}

\begin{proof}
	When $n=0$, (1) holds because the map $\sigma_0\colon A\to A_0$ coincides with the restriction of $\sigma\colon \cT_E\to \cO_E$ to $A$, which is injective by $A\cap \cK(\cF_E)=0$. 
	For $n\geq 1$, (1) follows from \cite[Proposition 4.6.3]{Brown-Ozawa:book} applied to the $*$-isomorphism $\sigma_0\colon A\xto{\sim} A_0$ and the linear map $E^{\otimes_A n}\ni \xi\mapsto \oS_\xi\in\cO_E$ satisfying $\oS_\xi^*\oS_\eta=\sigma_0(\bra\xi,\eta\ket)$. 
	
	As for (2), 
	we first show that for $x\in A_{\infty}$, it occurs that $x\oS_E=0$ only when $x=0$. 
	There are $a\in A$ and $b\in A_{\geq 1}$ such that $x=\sigma_0(a)+b$. 
	If $x\oS_\zeta=0$ for all $\zeta\in E$, then using an approximate unit $(e_\lambda)_{\lambda\in\Lambda}$ of $\cK(E)$, we obtain the norm convergence 
	\begin{align*}
		\sigma_1(\phi(a)e_\lambda)
		=
		(\sigma_0(a)+b)\sigma_1(e_\lambda) - b\sigma_1(e_\lambda)
		\xto{\lambda\in \Lambda}
		-b , 
	\end{align*}
	since $\sigma_0(c)\sigma_1(y)=\sigma_1(\phi(c)y)$ for $c\in A$, $y\in\cK(E)$ by the definition of $\sigma_n$. 
	In particular, $( \phi(a)e_\lambda )_{\lambda\in\Lambda}$ is a Cauchy net in $\cK(E)$, 
	while $\phi(a)e_\lambda\in\cL(E)$ should strictly converge to $\phi(a)$. 
	Therefore $\phi(a)\in\cK(E)$, or equivalently $a\in I_E$, and $x=\sigma_0(a)+b\in A_{\geq 1}$, from which it follows that $x=\lim_{\lambda\in\Lambda}x\sigma_1(e_\lambda)=0$ as claimed. 
	
	This claim readily shows the faithfulness of the left action $A_\infty\to \cL(X_1)$ as $X_1=\cspan \oS_E A_\infty$. 
	Moreover, it follows that for any $0\neq x\in A_{\leq n}$, there is $\zeta(x)\in E$ with $x\oS_{\zeta(x)}\neq 0$. 
	We put $\zeta_1:=\zeta(x)$ 
	and, inductively on $l=1,2,\cdots,n-1$, let $\zeta_{l+1}:=\zeta_l\otimes \zeta(\oS_{\zeta_l}^*x^*x\oS_{\zeta_l})\in E^{\otimes_A l+1}$, which satisfies 
	\begin{align*}
		\oS_{\zeta_{l+1}}^* x^*x \oS_{\zeta_{l+1}} 
		=
		\oS_{\zeta(\oS_{\zeta_l}^*x^*x\oS_{\zeta_l})}^*\oS_{\zeta_l}^* x^*x \oS_{\zeta_l}\oS_{\zeta(\oS_{\zeta_l}^*x^*x\oS_{\zeta_l})}
		\neq 0.
	\end{align*}
	Since $0\neq x\oS_{\zeta_{n}}\oS_{\zeta_{n}}^*x^*\in xA_{n}$, we see that $A_n\subset A_{\leq n}$ is essential. 
\end{proof}

\subsection{Simplicity criterion}\label{ssec:simpleCPalg}
In this subsection, we show the following characterization of the simplicity of Cuntz--Pimsner algebras for non-proper C*-correspondences. 

\begin{thm}\label{thm:simpleCPalg}
	For a C*-algebra $A$ and a non-proper non-degenerate faithful $(A,A)$-correspondence $E=(E,\phi)$, 
	the following are equivalent. 
	\begin{enumerate}
		\item[(i)]
		The C*-algebra $\cO_E$ is simple. 
		\item[(ii)]
		There is no closed ideal $I\subset A$ with $\bra E, \phi(I)E\ket \subset I \supset \phi^{-1}(\cK(EI))$ except $I=0$ and $A$. 
	\end{enumerate}
\end{thm}

We use the following argument by Kishimoto extracted from the proof of \cite[Lemma 1.1]{Kishimoto:1981outer}, which actually did not require the non-triviality of the strong Connes spectrum of $\alpha$. 
Recall that for a C*-algebra $A$, any pure state $\phi$ on a hereditary C*-subalgebra of $A$ uniquely extends to a pure state on $A$, which shall be still denoted by $\phi$. 

\begin{lem}[{Kishimoto~\cite{Kishimoto:1981outer}}]\label{lem:Kishimotoprf}
	Let $A$ be a C*-algebra, $\alpha\in\Aut(A)$ be a $*$-automorphism, $B\subset A$ be a non-zero hereditary C*-subalgebra, and $\phi$ be a pure state on $B$. 
	We take the GNS $*$-representation $\pi_\phi\colon A\to\cB(\cH_\phi)$ associated with the unique pure state on $A$ extending $\phi$. 
	If there is $a\in A$ such that
	\begin{align*}
		&\inf\{ \|xa\alpha(x)\| \mid 0\leq x\in B, \|x\|=1 \}>0, 
	\end{align*}
	then there is a unitary $V\in\cU(\cH_\phi)$ such that $\pi_\phi\alpha=\Ad V\pi_\phi$. 
\end{lem}

Then, we obtain a condition analogous to the conclusion of \cite[Lemma 1.1]{Kishimoto:1981outer}, sometimes referred to as Kishimoto's condition, if $\bigcap_n A_{\geq n}$ is small enough. 

\begin{lem}\label{lem:Kishimoto1.1}
	In the setting of \cref{notn:simpleCPalg}, let $B\subset A_\infty$ be a hereditary C*-subalgebra that is not contained in $\bigcap_{n=0}^{\infty}A_{\geq n}$. 
	Then for any $\xi\in X_m$ with $m\geq 1$, it holds 
	\begin{align*}
		&\inf\{ \| x\xi x \| \mid 0\leq x\in B, \|x\|=1 \}=0. 
	\end{align*}
\end{lem}

\begin{proof}
	For $n\in\bZ$, we consider the injective $*$-homomorphism onto the corner 
	\begin{align*}
		A_\infty = \cO_E^{\gamma} &\xto{\sim} e_n(\cO_E\rtimes_\gamma\bT)e_n \subset \cO_E\rtimes_\gamma\bT \\
		a &\mapsto a e_n = e_n a
	\end{align*}
	and observe that $e_nBe_n \subset \cO_E\rtimes_\gamma\bT$ is a hereditary C*-subalgebra. 
	It suffices to deduce a contradiction from the existence of $\xi\in X_m$ such that 
	\begin{align}\label{eq:lem:Kishimoto1.1}
		&\inf\{ \| e_0xe_0\xi \wh{\gamma}^m(e_0xe_0) \| \mid 0\leq x\in B, \|x\|=1 \}>0, 
	\end{align}
	because for all $x\in A_\infty$, we have 
	\begin{align*}
		&
		\| e_0xe_0\xi \wh{\gamma}^m(e_0xe_0) \| 
		= \| e_0xe_0 \xi e_mxe_m \| 
		= \| x\xi x e_m \| = \| x\xi x \| . 
	\end{align*}
	
	By assumption, there is $n\geq 1$ such that $B$ is not contained in $A_{\geq n}$. 
	We may fix such $n\geq 1$ with $n\in m\bZ$ and a pure state $\psi$ on $B/(B\cap A_{\geq n})$. Then $\psi$ uniquely extends to a pure state on $A_\infty/A_{\geq n}$ via the hereditary inclusion 
	%Here, for a hereditary C*-subalgebra $B\subset A_\infty$ and a closed ideal $I\subset A_\infty$, the C*-subalgebra $B+I\subset A_\infty$ is hereditary by $(B+I)A_\infty(B+I)=BA_\infty B+I=B+I$. 
	$B/(B\cap A_{\geq n})\subset A_\infty/A_{\geq n}$, which induces a pure state on $A_\infty$ denoted by $\psi_\infty:=\psi((-)+A_{\geq n})$. 
	Via the embedding as a corner $A_\infty\cong e_0(\cO_E\rtimes_\gamma\bT)e_0\subset \cO_E\rtimes_\gamma\bT$, we extend $\psi_\infty$ to a pure state on $\cO_E\rtimes_\gamma\bT$ denoted by $\psi_0:=\psi_\infty(e_0(-)e_0)$. 
	Note that $\psi_0|_{e_0Be_0}$ is a pure state on $B$ because it is nothing but the pullback of the pure state $\psi$ on $B/(B\cap A_{\geq n})$ along $B\twoheadrightarrow B/(B\cap A_{\geq n})$. 
	Thus, $\Ker\pi_{\psi_0}\not\supset e_0Be_0$ for the GNS $*$-representation $\pi_{\psi_0}\colon \cO_E\rtimes_\gamma\bT\to \cB(\cH_{\psi_0})$ associated with $\psi_0$. 
	Meanwhile, we have $\wh{\gamma}^{-n}(\Ker\pi_{\psi_0}) \supset e_0Be_0$ because of the following inclusion up to norm closures, 
	\begin{align*}
		&
		\bra \cH_{\psi_0}| \pi_{\psi_0}\wh{\gamma}^{n}(e_0 Be_0) \cH_{\psi_0}\ket 
		=
		\bra \cH_{\psi_0}| \pi_{\psi_0}(e_n Be_n) \cH_{\psi_0}\ket 
		\\={}&
		\psi_0( e_0 (\cO_E\rtimes_\gamma\bT) e_n B e_n (\cO_E\rtimes_\gamma\bT) e_0 )
		=
		\psi_0(e_0 X_n B X_n^* e_0)
		\\\subset{}& 
		\psi_\infty(A_{\geq n})
		=
		0. 
	\end{align*}
	
	Now assume \cref{eq:lem:Kishimoto1.1}. Then $\wh{\gamma}^{-m}(\Ker\pi_{\psi_0})=\Ker\pi_{\psi_0}$ by applying \cref{lem:Kishimotoprf} to the hereditary inclusion $e_0Be_0\subset \cO_E\rtimes_\gamma\bT$ with $\wh{\gamma}^m\in\Aut(\cO_E\rtimes_\gamma\bT)$ and the pure state $\psi_0|_{e_0Be_0}$ on $e_0Be_0$. 
	This contradicts $\Ker\pi_{\psi_0} \not\supset e_0 Be_0\subset \wh{\gamma}^{-n}(\Ker\pi_{\psi_0})$ since we chose $n\in m\bZ_{\geq 1}$. 
\end{proof}

The following lemma completes the substitution for the non-triviality of the strong Connes spectrum with the non-properness of the C*-correspondence. 

\begin{lem}\label{lem:simpleCPalg}
	In the setting of \cref{notn:simpleCPalg}, if $E$ is non-proper and satisfies (ii) of \cref{thm:simpleCPalg}, we have the following. 
	\begin{enumerate}
		\item
		Any non-zero closed ideal $J\subset A_\infty$ with $X_1^*JX_1\subset J \supset X_1JX_1^*$ is $A_\infty$ itself. 
		\item
		It holds $\bigcap_{n=0}^{\infty}A_{\geq n}=0$. 
	\end{enumerate}
\end{lem}

\begin{proof}
	We show (1). 
	There is some $n\geq 1$ with $A_{\leq n}\cap J\neq 0$ by \cite[Lemma III.4.1]{Davidson:book} applied to the dense inclusion $\bigcup_{n=0}^{\infty} A_{\leq n}\subset A_\infty$. 
	By \cref{lem:faithfulX} (2), we have $J_n:=A_n\cap J\neq 0$. 
	It follows from $A_n:=\cspan \oS_{E^{\otimes n}} \oS_{E^{\otimes n}}^*$ 
	that 
	\begin{align}\label{eq:lem:simpleCPalg}
		0\neq \oS_{E^{\otimes n}}^* J_n \oS_{E^{\otimes n}}\subset A_0\cap (X_1^*)^n J X_1^n\subset A_0\cap J . 
	\end{align}
	Since $\bra E, \phi(\sigma_0^{-1}(A_0\cap J))E\ket\subset \sigma_0^{-1}(A_0\cap J)$, and, by using \cref{lem:faithfulX} (1), 
	\begin{align*}
		&
		\phi^{-1}\bigl( \cK\bigl( E\sigma_0^{-1}(A_0\cap J) \bigr) \bigr) 
		\subset 
		\phi^{-1}\sigma_1^{-1}(\cspan S_{E\sigma_0^{-1}(A_0\cap J)}S_{E\sigma_0^{-1}(A_0\cap J)}^*) 
		\\\subset{}& %\overset{\bigstar}{\subset} 
		\sigma_0^{-1}(\sigma_0(I_E)\cap \cspan S_{E\sigma_0^{-1}(A_0\cap J)}S_{E\sigma_0^{-1}(A_0\cap J)}^*) 
		%\\\subset{}& 
		%\sigma_0^{-1}(A_0\cap \cspan X_1JX_1^*) 
		\subset 
		\sigma_0^{-1}(A_0\cap J), 
	\end{align*}
	%where we have used \cite[Lemma 4.6.15]{Brown-Ozawa:book} for $\overset{\bigstar}{\subset}$, 
	the condition (ii) implies $A_0\subset J$ or $A_0\cap J=0$, where the latter is forbidden by \cref{eq:lem:simpleCPalg}. Thus it follows $J=A_\infty$. 
	
	Next we show (2). 
	We put $J:=\bigcap_{n=0}^{\infty}A_{\geq n}$, which is a closed ideal of $A_\infty$ with $X_1^*J X_1\subset J \supset X_1JX_1^*$ and thus must be $0$ or $A_\infty$ by (1). 
	Using the projection $P_0$ from \cref{notn:simpleCPalg} (4), we observe that 
	\begin{align*}
		&
		P_0( A\cap \sigma^{-1}(J) )P_0 
		\subset 
		P_0 \sigma^{-1}(A_{\geq 1}) P_0
		\\={}& 
		P_0 \Bigl( \cK(\cF_EI_E) + \cspan\bigcup_{n=1}^{\infty} \tS_{E^{\otimes n}}\tS_{E^{\otimes n}}^* \Bigr) P_0 
		\\={}&
		P_0 \cK(\cF_EI_E) P_0 
		=
		P_0I_EP_0 , 
	\end{align*}
	and that $a\in A$ satisfies $\sigma(a)\in A_0\cap J$ only if $a\in I_E$. Thus $J\neq A_\infty$ since $E$ is non-proper, and $J=0$ follows. 
\end{proof}

With these ingredients, the harder direction (ii) $\Rightarrow$ (i) of \cref{thm:simpleCPalg} can be shown by following the strategies of \cite[Lemma 3.2, Theorem 3.1]{Kishimoto:1981outer} and \cite[Theorem 3.2]{Elliott:1980some}. 

\begin{lem}\label{lem:Kishimoto3.2}
	In the setting of \cref{notn:simpleCPalg}, suppose $\bigcap_{n=0}^{\infty}A_{\geq n}=0$. 
	Then for any $\e>0$, $0\leq a\in A_\infty$, $m\in\bZ_{\geq 0}$, and $\xi_k\in X_{n_k}$ with $n_k\in\bZ\setminus\{0\}$ for $1\leq k\leq m$, 
	there is $0\leq x\in A_\infty$ with $\|x\|=1$ such that for all $1\leq k\leq m$, 
	\begin{align*}
		&
		\|xax\|\geq \|a\|-\e
		\qquad\text{and}\qquad
		\|x\xi_k x\|\leq \e. 
	\end{align*}
\end{lem}

\begin{proof}
	By retaking $a$ and $\xi_k$ to be non-zero, 
	 multiplying positive scalars, and replacing $\xi_k$ with $\xi_k^*$ if necessary, we may assume $0<\e<1$, $\|a\|=\|\xi_k\|=1$, and $n_k\geq 1$ for all $1\leq k\leq m$. 
	For $0<\delta<1$, we fix a non-decreasing continuous functions $f_\delta, g_\delta\colon [0,\infty)\to[0,1]$ such that 
	\begin{itemize}
		\item
		$f_\delta(1-\delta)=0$, $f_\delta(1)=1$, and 
		\item
		$g_\delta(t):=\min\{(1-\delta)^{-1}t, 1\}$ for $t\geq 0$. 
	\end{itemize}
	Note that for $0\leq x\in A_\infty$ with $\|x\|=1$, we have $f_\delta(x)g_\delta(x) = f_\delta(x)$ and thus $yg_\delta(x)=y$ for any $y\in \overline{f_\delta(x)A_\infty f_\delta(x)}$. 
	
	We consider the hereditary C*-subalgebra $B_0:=\overline{f_{\e}(a)A_\infty f_{\e}(a)}$. 
	Inductively on $1\leq k\leq m$, suppose that we are given a non-zero hereditary C*-subalgebra $B_{k-1}\subset A_\infty$. 
	By \cref{lem:Kishimoto1.1}, there exists $0\leq x_k\in B_{k-1}$ with $\|x_k\|=1$ such that $\|x_k\xi_k x_k\|<\e$. 
	We take $0<\delta_k<1$ small enough to satisfy $\| g_{\delta_k}(x_k)\xi_k g_{\delta_k}(x_k) \|<\e$ and set $y_k:=f_{\delta_k}(x_k)\in B_{k-1}$. 
	Then $B_k:=\overline{y_k B_{k-1} y_k}\subset B_{k-1}$ is a hereditary C*-subalgebra of $A_\infty$ containing $y_k$. 
	
	By construction, $B_k$ forms a decreasing sequence of non-zero hereditary C*-subalgebras of $A_\infty$, and $y_{k}=y_k g_{\delta_{l}}(x_{l})=y_k g_{\e}(a)$ for $1\leq l\leq k\leq m$ by $y_k\in B_{l}\subset B_0$. 
	We see that 
	\begin{align*}
		&
		\| y_m a y_m \|
		\geq
		\| y_m g_{\e}(a) y_m \| - \e 
		=\|y_m\|^2 -\e =1-\e 
	\end{align*}
	and that 
	\begin{align*}
		\| y_{m} \xi_k y_{m} \| 
		=
		\| y_{m}g_{\delta_{k}}(x_{k}) \xi_k g_{\delta_{k}}(x_{k})y_{m} \| 
		\leq \e 
	\end{align*}
	for $1\leq k\leq m$. 
	Thus $x:=y_m$ satisfies the desired conditions. 
\end{proof}

\begin{proof}[Proof of \cref{thm:simpleCPalg}]
	We assume (ii) and deduce (i). 
	Fix a proper closed ideal $J\subsetneq \cO_E$ arbitrarily. To see $J=0$, it suffices to prove that the canonical conditional expectation $\Phi\colon \cO_E\to \cO_E^\gamma =A_\infty$ defined by $\Phi(x):=\int_\bT\gamma_z(x)dz$ for $x\in \cO_E$ (with respect to the Haar probability measure on $\bT$) factors through the quotient map $q\colon \cO_E \to \cO_E/J$, because then $\Phi(x)=0$ for any $0\leq x\in J$, which will imply $x=0$ by the faithfulness of $\Phi$. 
	
	Since $X_1^*(J\cap A_\infty)X_1 \subset J\cap A_\infty \supset X_1(J\cap A_\infty)X_1^*$, we see that $J\cap A_\infty$ is $0$ or $A_\infty$ by \cref{lem:simpleCPalg} (1). 
	Thus $J\cap A_\infty=0$ by $J\neq \cO_E$ because the smallest closed ideal of $\cO_E$ containing $A_\infty$ is $\cO_E$ itself. 
	Fix $a\in A_\infty$, $m\in\bZ_{\geq 0}$, and $\xi_k\in X_{n_k}$ with $n_k\in\bZ\setminus\{0\}$ for $1\leq k\leq m$ arbitrarily. 
	By \cref{lem:Kishimoto3.2} and \cref{lem:simpleCPalg} (2), for any $\e>0$ we have $0\leq x_\e\in A_\infty$ with $\|x_\e\|=1$ such that $\| x_\e a^*ax_\e \|\geq \| a^*a \|-\e$ and that $\|x_\e a^*\xi_k x_\e\|\leq \e$ for all $k$. 
	Therefore 
	if $a=0$, clearly $\|\Phi (a+\sum_{k=1}^{m} \xi_k)\|=0\leq \|q(a+\sum_{k=1}^{m} \xi_k)\|$, and if $a\neq 0$, 
	\begin{align*}
		&
		\| a^* \| \Bigl\| \Phi\Bigl( a+\sum_{k=1}^{m} \xi_k \Bigr) \Bigr\| 
		%= \| a^* \|\| a \|
		=
		\| a^*a \| 
		\leq 
		\| x_\e a^*a x_\e \| + \e 
		\\={}&
		\| q(x_\e a^*a x_\e) \| + \e 
		\leq
		\Bigl\| q(x_\e a^*a x_\e) + \sum_{k=1}^{m}q(x_\e a^*\xi_k x_\e ) \Bigr\|+ (1+m)\e 
		\\\leq{}&
		\|a^*\| \Bigl\| q\Bigl(a+\sum_{k=1}^{m} \xi_k\Bigr) \Bigr\| + (1+m)\e . 
	\end{align*}
	Since $\e>0$ is arbitrary, it follows that $\| \Phi (a+\sum_{k=1}^{m} \xi_k) \| \leq \| q( a+\sum_{k=1}^{m} \xi_k ) \|$ 
	in both cases. Therefore $\Phi$ factors through $q$. 
	
	Conversely, it is a consequence of \cref{prop:CPideal} below that any non-trivial closed ideal $I\subset A$ with $\bra E,\phi(I)E\ket\subset I\supset \phi^{-1}(\cK(EI))$ generates a non-trivial closed ideal of $\cO_E$. Thus (i) implies (ii). 
\end{proof}

\begin{prop}\label{prop:CPideal}
	In the setting of \cref{notn:simpleCPalg}, let $I\subset A$ be a closed ideal such that $\bra E, \phi(I)E\ket\subset I \supset \phi^{-1}(\cK(EI))$. Then, $A_0\cap \cspan\cO_E \sigma_0(I)\cO_E = \sigma_0(I)$. 
\end{prop}

In the proof, we use the following lemma. 

\begin{lem}\label{lem:CPideal}
	Let $C$ be a C*-algebra and $A,D\subset C$ be C*-subalgebras satisfying $ADA\subset D$. 
	For $\e>0$, consider the continuous function $h_\e\colon [0,\infty) \ni t \mapsto \max\{ 0, t-\e \} \in [0,\infty)$. 
	For any positive $a\in A$ and any positive $d\in D$ with $\|a-d\|\leq\e$, we have $h_{2\e}(a)\in D$. 
\end{lem}

\begin{proof}
	Let $k_\e\colon [0,\infty) \to [0,\infty)$ be the continuous function such that $k_\e(t)=0$ for $0\leq t\leq 2\e$ and $k_\e(t)=\sqrt{\frac{t-2\e}{t-\e}}$ for $t\geq 2\e$. 
	It follows from $a-\e1_C \leq d$ that 
	\begin{align*}
		&
		0\leq h_{2\e}(a) = k_\e(a) (a-\e1_C) k_\e(a) \leq k_\e(a) d k_\e(a) \in ADA\subset D. 
	\end{align*}
	Since $D$ is a closed ideal of the C*-subalgebra generated by $A\cup D\subset C$, its heredity shows $h_{2\e}(a)\in D$. 
\end{proof}

\begin{proof}[Proof of \cref{prop:CPideal}]
	For all $n\in\bZ_{\geq 0}$, using an approximate unit of $\cK(E^{\otimes_A n})$, we check that 
	\begin{align}\label{eq:thm:simpleCPalg1}
		I \tS_{E^{\otimes n}} 
		= 
		\tS_{I E^{\otimes n}} 
		\subset 
		\cspan \tS_{\cK(E^{\otimes n}) I E^{\otimes n}} 
		\subset 
		\cspan \tS_{E^{\otimes n} I} 
		= 
		\tS_{E^{\otimes n}} I 
	\end{align}
	by $\bra E, \phi(I)E\ket$, and that 
	\begin{align}\label{eq:thm:simpleCPalg2}
		&\begin{aligned}
			&
			\{ x\in \cK(E^{\otimes_A n}) \mid x\otimes_\phi 1_E \in \cK(E^{\otimes_A n+1}I) \}
			\\\subset{}&
			\cspan 
			\cK(E^{\otimes_A n})
			\{ x\in \cK(E^{\otimes_A n}) \mid x\otimes_\phi 1_E \in \cK(E^{\otimes_A n+1}I) \}
			\cK(E^{\otimes_A n})
			\\\subset{}&
			\cspan 
			T_{E^{\otimes n}} 
			\{ a\in A \mid a\otimes_\phi 1_E \in \cK(EI) \}
			T_{E^{\otimes n}}^* 
			\subset 
			\cK(E^{\otimes_A n}I) , 
	\end{aligned}\end{align}
	where the last inclusion follows from $\phi^{-1}(\cK(EI))\subset I$. 
	For $n\in\bZ_{\geq 0}$, we let $P_n\in\cL(\cF_E)$ be the projection onto the direct summand $E^{\otimes_A n}$ and put 
	\begin{align*}
		D_n:= 
		\sum_{k=0}^{n} \cK(P_k\cF_EI_E) + \cspan \tS_{E^{\otimes k}}I\tS_{E^{\otimes k}}^* \subset \cT_E , 
	\end{align*}
	which is a $*$-subalgebra of $\cT_E$ as $\tS_{E^{\otimes k}}I\tS_{E^{\otimes k}}^* \cK(P_l\cF_EI_E)$ is contained in $\cK(P_l\cF_EI_E)$ if $0\leq k\leq l$ and equals $0$ if $0\leq l<k$. 
	Moreover, $D_n$ is a C*-subalgebra because $\sum_{k=0}^{n} \cK(P_k\cF_EI_E) \subset \cspan D_n$ is a norm closed ideal and $\sum_{k=0}^{n}\cspan \tS_{E^{\otimes k}}I\tS_{E^{\otimes k}}^*\subset \cT_E$ is a C*-subalgebra. 
	
	Abusing notation, we write $\gamma$ for the continuous $\bT$-action on the Hilbert $A$-module $\cF_E$ such that $\gamma_z(\xi)=z^n\xi$ for $n\in\bZ_{\geq 0}$, $\xi\in E^{\otimes_A n}$, and $z\in\bT$, which induces a strictly continuous $\bT$-action on the C*-algebra $\cL(\cF_E)$ still denoted by $\gamma$ such that $\gamma_z(x)=z^{n-m} x$ for $m,n\in\bZ_{\geq 0}$ and $x\in \cL(P_m\cF_E,P_n\cF_E)$. Its restriction yields the continuous $\bT$-action on $\cT_E$ such that $A\subset \cT_E^\gamma$ and $\gamma_z(\tS_\xi)=z^n\tS_\xi$ for $\xi\in E^{\otimes_A n}$. 
	Using \cref{eq:thm:simpleCPalg1}, we observe that 
	\begin{align*}
		&
		\sigma^{-1}(\cspan\cO_E \sigma_0(I) \cO_E) 
		=
		\cK(\cF_EI_E) + \cspan\cT_E I\cT_E 
		\\={}&
		\cspan \bigcup_{m,n=0}^{\infty} \cK(P_m\cF_EI_E,P_n\cF_EI_E) \cup \tS_{E^{\otimes n}}I\tS_{E^{\otimes m}}^*, 
	\end{align*}
	which is closed under the $\bT$-action. 
	Using the conditional expectation $\cT_E\ni x\mapsto \int_\bT\gamma_z(x)dz\in \cT_E^\gamma$, we have 
	\begin{align*}
		&
		\cT_E^\gamma \cap \sigma^{-1}(\cspan\cO_E \sigma_0(I) \cO_E) 
		= 
		\cspan \bigcup_{n=0}^{\infty} 
		\cK(P_n \cF_E I_E) 
		\cup 
		\tS_{E^{\otimes n}}I\tS_{E^{\otimes n}}^* 
		= 
		\cspan \bigcup_{n=0}^{\infty} D_n . 
	\end{align*}
	Clearly, we have $\sigma_0(I)\subset \cspan\cO_E\sigma_0(I)\cO_E$. To see that $A_0\cap \cspan\cO_E\sigma_0(I)\cO_E \subset \sigma_0(I)$, it is enough to show that $A\cap \sigma^{-1}(\cspan\cO_E \sigma_0(I) \cO_E) \subset I$. 
	Take any positive element $a\in A\cap \sigma^{-1}(\cspan\cO_E \sigma_0(I) \cO_E) \subset \cspan\bigcup_{n=0}^{\infty}D_n$ and fix any $\e>0$. 
	Since $h_{2\e}(a)\xto{\e\to +0}a$, and $A\cap \sigma^{-1}(\cspan\cO_E \sigma_0(I) \cO_E)$ is linearly spanned by its positive part, it suffices to prove $h_{2\e}(a)\in I$. 
	
	There are $m\in \bZ_{\geq 0}$ and $d\in D_m$ such that 
	$d\geq 0$ and $\|a - d\|<\e$. It follows from \cref{lem:CPideal} applied to $(C,A,D,a,d)=(\cT_E,A,D_m,a,d)$ that $h_{2\e}(a)\in D_m$, which implies that there are $x_k\in \cK(P_k\cF_EI_E)$ and $y_k\in \cspan\tS_{E^{\otimes k}}I\tS_{E^{\otimes k}}^*$ for $0\leq k\leq m$ such that 
	$h_{2\e}(a)=\sum_{k=0}^{m}x_k+y_k$. 
	Thus, by looking at each corner $\cK(P_k\cF_E)\cong \cK(E^{\otimes_A k})$ using $P_k$, we have 
	\begin{align*}
		&
		h_{2\e}(a)\otimes_\phi 1_{E^{\otimes k}} = P_k h_{2\e}(a) P_k = x_k+\sum_{l=0}^{k}P_k y_l P_k 
	\end{align*}
	for $0\leq k\leq m+1$, where we let $x_{m+1}:=0$ and $y_{m+1}:=0$. 
	We obtain 
	\begin{align*}
		&
		x_{k-1} \otimes_\phi 1_E 
		= 
		h_{2\e}(a)\otimes_\phi 1_{E^{\otimes_A k}} - \sum_{l=0}^{k-1}P_{k-1} y_l P_{k-1} \otimes_\phi 1_E  
		= 
		x_k + P_k y_k P_k \in \cK(E^{\otimes_A k})
	\end{align*}
	for $1\leq k\leq m+1$ because $P_{k-1}y_lP_{k-1}\otimes_\phi 1_E=P_ky_lP_k$ for $0\leq l<k$. %In particular, $x_{m}\otimes_\phi 1_E = 0$. 
	Inductively on $k=m+1,m,m-1,\cdots,1$, we check that $x_{k-1}\otimes_\phi 1_E = x_k + P_k y_k P_k\in \cK(E^{\otimes_A k}I)$ and that $x_{k-1} \in \cK(E^{\otimes_A k-1}I)$ by \cref{eq:thm:simpleCPalg2}. 
	Hence, $h_{2\e}(a) = x_0 + y_0 \in I$ as desired, which completes the proof. 
\end{proof}

As a consequence of \cref{thm:simpleCPalg}, we obtain the simplicity criterion for Toeplitz--Pimsner algebras. 

\begin{cor}\label{cor:simpleCPalg}
	For a non-zero C*-algebra $A$ and a non-degenerate faithful $(A,A)$-correspondence $E=(E,\phi)$, 
	the C*-algebra $\cT_E$ is simple if and only if the following hold. 
	\begin{itemize}
		\item[(i)]
		$\phi(A)\cap \cK(E)=0$, and 
		\item[(ii)]
		there is no closed ideal $I\subset A$ with $\bra E, \phi(I)E\ket \subset I$ except $I=0$ and $A$. 
	\end{itemize}
\end{cor}

Note that the condition (ii) automatically implies the fullness of $E$ because $\bra E,\phi(A)E\ket$ must generate $A$ as a closed ideal in that case. 

\begin{proof}
	It follows from the definition of the quotient map $\sigma\colon \cT_E\to\cO_E$ that $\cT_E$ is simple if and only if $\cO_E$ is simple and $\cF_EI_E=0$. 
	Since $\cF_E$ contains $A$, we see that $\cF_EI_E=0$ occurs if and only if $I_E=0$, or equivalently $\phi(A)\cap \cK(E)=0$ as $\phi$ is injective. 
	Now \cref{thm:simpleCPalg} implies the desired equivalence since $E$ must be non-proper and satisfy $\phi^{-1}(\cK(EI))=0$ for any closed ideal $I\subset A$ if $\phi(A)\cap \cK(E)=0$. 
\end{proof}

\section{Bimodules and equivariance}\label{sec:equivTPalg}

From this section, we investigate actions of unitary tensor categories on C*-algebras. %In this paper, we only consider small categories. 

\subsection{Terminologies on tensor categories and their actions}\label{ssec:tenorcat}
We recall and fix several terminologies on tensor categories and their actions on C*-algebras. 
A \emph{C*-category} is a category $\cA$ equipped with a structure of a complex Banach space on $\Hom_{\cA}(X,Y)$ and an anti-linear map $\Hom_{\cA}(X,Y)\to \Hom_{\cA}(Y,X)$ for all objects $X,Y\in\cA$ satisfying 
$\|gf\|\leq\|g\|\|f\|$, $f^{**}=f$, $\|f^*f\|=\|f\|^2$, and $f^*f\in \{ a^*a \mid a\in\Hom_{\cA}(X,X) \}$ for all morphisms $f\in \Hom_{\cA}(X,Y)$ and $g\in \Hom_{\cA}(Y,Z)$. 
Note that $\Hom_{\cA}(X,X)$ is a C*-algebra by assumption. 
We also assume that a C*-category $\cA$ is additive and idempotent-complete in the way compatible with the $*$-linear structure. More precisely, for all objects $X,Y\in\cA$, there is an object $X\oplus Y$ with isometries $r\in \Hom_{\cA}(X, X\oplus Y)$ and $s\in \Hom_{\cA}(Y, X\oplus Y)$ (i.e., $r^*r=\id_X$ and $s^*s=\id_Y$) satisfying $rr^*+ss^*=\id_{X\oplus Y}$, and there are an object $Z\in\cA$ and an isometry $v\in\Hom_{\cA}(Z,X)$ satisfying $vv^*=p$ for any projection $p\in\Hom_{\cA}(X,X)$ (i.e., $p=p^*=p^2$). 
For C*-categories $\cA$ and $\cC$, a \emph{$*$-functor} $F\colon \cC\to\cA$ is a functor such that $F(cf+g)=cF(f)+F(g)$ and $F(f^*)=F(f)^*$ for all $\pi,\varpi\in\cC$, $f,g\in\Hom_\cC(\pi,\varpi)$, and $c\in\bC$. 

A \emph{C*-tensor category} is a C*-category $\cA$ equipped with a bifunctor $-\otimes -\colon\cA\times \cA\to \cA$ preserving the $*$-linear structure (more precisely, bilinear and $(f\otimes g)^*=f^*\otimes g^*$ for all morphisms $f$ and $g$ in $\cA$), a natural unitary isomorphism $\ass(X,Y,Z)\colon (X\otimes Y)\otimes Z\to X\otimes(Y\otimes Z)$ for $X,Y,Z\in\cA$, an object $\mathbbm{1}_{\cA}$ with natural unitary isomorphisms $\mathrm{l}_X\colon \mathbbm{1}_\cA\otimes X\to X$ and $\mathrm{r}_X\colon X\otimes\mathbbm{1}_\cA\to X$ for $X\in\cA$ such that $\mathrm{l}_{\mathbbm{1}_\cA}=\mathrm{r}_{\mathbbm{1}_\cA}$ and that the following diagrams are commutative for all $W,X,Y,Z\in\cA$: 
\begin{align*}
	&
	\begin{aligned}
		\xymatrix@C=7em@R=2em{
			((W\otimes X)\otimes Y)\otimes Z \ar[dd]_-{\ass(W\otimes X, Y,Z)} \ar[r]^-{\ass(W,X,Y)\otimes\id_Z}& 
			(W\otimes (X\otimes Y))\otimes Z 
			\ar[d]^-{\ass(W,X\otimes Y,Z)} 
			\\
			& W\otimes ((X\otimes Y)\otimes Z ) \ar[d]^-{\id_W\otimes\ass(X,Y,Z)}
			\\
			(W\otimes X)\otimes (Y\otimes Z) 
			\ar[r]^-{\ass(W, X, Y\otimes Z)} &
			W\otimes (X\otimes (Y\otimes Z)) 
		}
	\end{aligned}, 
\end{align*}
and 
\begin{align*}
	&
	\begin{aligned}
		\xymatrix@C=2em{
			(X\otimes \mathbbm{1}_\cA) \otimes Y 
			\ar[rr]^-{\ass(X,\mathbbm{1}_\cA,Y)} \ar[dr]_-{\mathrm{r}_X\otimes\id_Y}&& 
			X\otimes (\mathbbm{1}_\cA \otimes Y) 
			\ar[dl]^-{\id_X\otimes\mathrm{l}_Y} 
			\\
			& X\otimes Y &
		}
	\end{aligned}. 
\end{align*}

For a C*-tensor category $\cA$, we call $\ass(X,Y,Z)$ the associator and $\mathbbm{1}_\cA$ the unit object. 
For simplicity, we shall not write $\mathrm{r}_X$ and $\mathrm{l}_X$ by abuse of notation. %(which is justified by replacing $\cA$ with its strictification, for example). 

For objects $X$ and $\overline{X}$ in a C*-tensor category $\cA$, we say that $\overline{X}$ is conjugate to $X$ if there are $R\in \Hom_{\cA}(\mathbbm{1}_\cA,\overline{X}\otimes X)$ and $\overline{R}\in \Hom_{\cA}(\mathbbm{1}_\cA,X\otimes\overline{X})$ such that $(\id_{\overline{X}}\otimes\overline{R})^*\ass(\overline{X},X,\overline{X})(R\otimes\id_{\overline{X}})=\id_{\overline{X}}$ and $(\id_{X}\otimes R)^*\ass(X,\overline{X},X)(\overline{R}\otimes\id_X)=\id_X$. 
The pair $(R,\overline{R})$ satisfying these equalities are called a solution of the conjugate equation for $X$ and $\overline{X}$. 
A C*-tensor category $\cA$ is said to be \emph{rigid} if every objects $X\in\cA$ admits its conjugate object. 
An object $X\in\cA$ is called \emph{irreducible} if $\Hom_\cA(X)\cong\bC$. 
Assume that $\cA$ is a rigid C*-tensor category such that $\mathbbm{1}_\cA$ is irreducible. 
Then, $\cA$ is \emph{semisimple} in the sense that any object in $\cA$ is isomorphic to a direct sum of finitely many irreducible objects. 
For each $X\in\cA$ with its conjugate $\overline{X}$, there exists a solution $(R_X,\overline{R}_X)$ of the conjugate equation for $X$ and $\overline{X}$ with $\|R_X\|=\|\overline{R}_X\|$ such that $\|R_X\|\|\overline{R}_X\|\leq \|R\|\|\overline{R}\|$ for any other solution $(R,\overline{R})$ of the conjugate equation for $X$ and $\overline{X}$. Such $(R_X,\overline{R}_X)$ is called a \emph{standard solution}. We put $\dim_\cA X:=\|R_X\|\|\overline{R}_X\|$ and call it the \emph{intrinsic dimension} of $X$. 
Note that $R_X^* R_X = (\dim_\cA X) \id_{\mathbbm{1}_\cA} = \overline{R}_X^*\overline{R}_X$. 

For C*-tensor categories $\cA$ and $\cC$, a functor $F\colon\cC\to \cA$ accompanied by certain data that make $F$ compatible with the C*-tensor structures on $\cC$ and $\cA$, which we will recall soon later in a specific case, is called a \emph{unitary tensor functor}. 
If moreover $F$ admits a quasi-inverse $*$-functor, it is called a \emph{unitary monoidal equivalence}. 
For details on these notions and for further facts on rigid C*-tensor categories, we refer to \cite{Etingof-Gelaki-Nikshych-Ostrik:book,Neshveyev-Tuset:book}. 

In this paper, we call an (essentially small) rigid C*-tensor category $\cC$ with $\Hom_\cC(\mathbbm{1}_\cC)\cong\bC$ a \emph{unitary tensor category}. We fix a set $\Irr\cC$ of representatives of isomorphism classes of irreducible objects in $\cC$ such that $\mathbbm{1}_\cC\in\Irr\cC$. 
An example of a unitary tensor category is the category $\Rep G$ of finite-dimensional unitary representations of a compact group or, more generally, a compact quantum group $G$. 
In \cref{sec:TLJ}, we will focus on another example denoted by $\cC(\mathfrak{sl}_2,k)$ and its variants (see the beginning of that section for some of the properties). 
In \cref{sec:cyclicCuntz}, we focus on unitary tensor categories coming from $3$-cocycles on discrete groups in the following manner. 

\begin{eg}\label{eg_Hilb_Gamma}
	%For a (small) set $I$, we let $\Hilb_I$ be the C*-category such that its objects are finite-dimensional Hilbert spaces $\cH$ with $I$-gradings, i.e., orthogonal direct sum decompositions of the form $\cH=\bigoplus_{i\in i}\cH_i$, and that its morphisms are $\bC$-linear maps preserving the $I$-gradings. We write $\bC_i$ for the one-dimensional Hilbert space graded by a single element $i\in I$. 
	%Let $\Gamma$ be a discrete group. 
	For a discrete group $\Gamma$, we let $\Hilb_\Gamma$ be the C*-category such that its objects are finite-dimensional Hilbert spaces $\cH$ with $\Gamma$-gradings, i.e., orthogonal direct sum decomposition of the form $\cH=\bigoplus_{g\in \Gamma}\cH_g$, and that its morphisms are $\bC$-linear maps preserving the $\Gamma$-gradings. We write $\bC_g$ for the one-dimensional Hilbert space graded by a single element $g\in \Gamma$. 
	Then $\Hilb_\Gamma$ has a structure of a C*-tensor category such that $\bC_g\otimes\bC_h:=\bC_{gh}$ with $\mathbbm{1}_{\Hilb_{\Gamma}}=\bC_{e}$ and 
	that $\ass(\bC_f,\bC_g,\bC_h)=\id_{\bC_{fgh}}\colon (\bC_f\otimes\bC_g)\otimes\bC_h = \bC_{fgh} \to \bC_{fgh} = \bC_f\otimes(\bC_g\otimes\bC_h)$ for all $f,g,h\in\Gamma$, 
	where $\mathrm{l}_\cH$ and $\mathrm{r}_\cH$ are multiplications by scalars. 
	We can see that $\Hilb_{\Gamma}$ is a unitary tensor category. 
	
	More generally, let $\omega\in \mathrm{Z}^3(\Gamma;\bT)$ be a $\bT$-valued $3$-cocycle, i.e., $\omega\colon\Gamma^3\to \bT$ such that $\omega(f,g,h)\omega(f,gh,k)\omega(g,h,k)=\omega(fg,h,k)\omega(f,g,hk)$ for all $f,g,h,k\in\Gamma$. 
	Then, $\Hilb_{\Gamma}$ has another structure of a C*-tensor category 
	such that $\bC_g\otimes\bC_h:=\bC_{gh}$ and that $\ass(\bC_f,\bC_g,\bC_h)=\omega(f,g,h)\id_{\bC_{fgh}}\colon (\bC_f\otimes\bC_g)\otimes\bC_h = \bC_{fgh} \to \bC_{fgh} = \bC_f\otimes(\bC_g\otimes\bC_h)$ for all $f,g,h\in\Gamma$ with the same unit object and natural unitaries $\mathrm{l}_\cH$, $\mathrm{r}_\cH$ as $\Hilb_{\Gamma}$ (note that we have $\omega(e,g,e)=1$ by letting $f=h=k=e$ in the cocycle relation above). 
	The resulting C*-tensor category is denoted by $\Hilb_{\Gamma,\omega}$. 
	We can see that $\Hilb_{\Gamma,\omega}$ is a unitary tensor category, which depends only on the cohomology class $[\omega]\in\Ho^3(\Gamma;\bT)$ up to unitary monoidal equivalence. 
	%(Note that $\Hilb_{\Gamma,\omega}$ need not be strict, but $\mathrm{l}_\cH$ and $\mathrm{r}_\cH$ of $\Hilb_{\Gamma,\omega}$ are still defined in a trivial way.) 
\end{eg}

Let $\cC$ and $\cD$ be unitary tensor categories. 
We write $\cC\boxtimes\cD$ for their Deligne tensor product (see \cite[Subsection 4.6]{Etingof-Gelaki-Nikshych-Ostrik:book}). 
It is a C*-category such that its objects are isomorphic to direct sums of finitely many objects denoted by $\pi\boxtimes\varpi$ for $\pi\in\cC$ and $\varpi\in\cD$, that $\Hom_{\cC\boxtimes\cD}(\pi_1\boxtimes\varpi_1,\pi_2\boxtimes\varpi_2) = \Hom_\cC(\pi_1,\pi_2)\otimes\Hom_\cD(\varpi_1,\varpi_2)$ for all $\pi_1,\pi_2\in\cC$, $\varpi_1,\varpi_2\in\cD$, and that the bifunctor $-\boxtimes-\colon \cC\times\cD\to \cC\boxtimes\cD$, $(\pi,\varpi)\mapsto \pi\boxtimes\varpi$, $(f,g)\mapsto f\boxtimes g$ is compatible with the $*$-linear structures in a similar sense to $-\otimes -\colon \cA\times\cA\to\cA$ above. 
It has a canonical structure of a C*-tensor category such that the tensor product of $\pi_1\boxtimes\varpi_1$ and $\pi_2\boxtimes\varpi_2$ is $(\pi_1\otimes\pi_2) \boxtimes (\varpi_1\otimes\varpi_2)$ for all $\pi_1,\pi_2\in\cC$, $\varpi_1,\varpi_2$ with $\mathbbm{1}_{\cC\boxtimes\cD}=\mathbbm{1}_\cC\boxtimes \mathbbm{1}_\cD$, where the natural unitary isomorphisms $\ass$, $\mathrm{l}$, $\mathrm{r}$ of $\cC\boxtimes\cD$ are induced by those of $\cC$ and $\cD$. 
Then, $\cC\boxtimes\cD$ is a unitary tensor category. We may choose $\Irr(\cC\boxtimes\cD) = \{ \pi \boxtimes\varpi \mid \pi\in\Irr\cC, \varpi\in\Irr\cD \}$ and $\overline{\pi\boxtimes\varpi}=\overline{\pi}\boxtimes\overline{\varpi}$ for all $\pi\in\cC$ and $\varpi\in\cD$. 

We often assume that a unitary tensor category $\cC$ is \emph{countable} in the sense that $\cC$ has only countably many objects. 
Note that if $\Irr\cC$ is countable, we have a countable strict unitary tensor category that is unitarily monoidally equivalent to $\cC$ by applying strictification after skeletalization to $\cC$. 
A countable unitary tensor category $\cC$ is said to be a \emph{unitary fusion category} if $\#\Irr\cC<\infty$. 

For a C*-algebra $A$, we write $\Cor(A)$ for the C*-tensor category of non-degenerate $(A,A)$-correspondences as objects, bilinear adjointable operators as morphisms with the operator norms and the $*$-operations, and interior tensor products as the monoidal structure. 
A \emph{$\cC$-action} on $A$ is a unitary tensor functor $(\alpha,\fu)\colon \cC\to \Cor(A)$. 
Or equivalently, up to some simplification as in \cite[Remark 2.14]{Arano-Kitamura-Kubota:tensorKK}, it is a pair of a $*$-functor $\alpha\colon\cC\to\Cor(A)$ and a family $\fu=(\fu_{\pi,\varpi})_{\pi,\varpi\in\cC}$ of unitaries $\fu_{\pi,\varpi}\in\cU\cL(\alpha(\pi)\otimes_A\alpha(\varpi),\alpha(\pi\otimes\varpi))$ satisfying left $A$-linearity and naturality such that $\alpha(\mathbbm{1}_\cC)= (A,\id_A)$ and the diagram 
\begin{align*}
	\xymatrix@C=3em{
		\alpha(\tau\otimes\pi)\otimes_A\alpha(\varpi) \ar[d]_-{\fu_{\tau\otimes\pi,\varpi}} & 
		\alpha(\tau)\otimes_A\alpha(\pi)\otimes_A\alpha(\varpi) 
		\ar[l]_-{\fu_{\tau,\pi}\otimes1} \ar[r]^-{1\otimes\fu_{\pi,\varpi}} & 
		\alpha(\tau)\otimes_A\alpha(\pi\otimes\varpi)\ar[d]^-{\fu_{\tau,\pi\otimes\varpi}}\\
		\alpha((\tau\otimes\pi)\otimes\varpi) 
		\ar[rr]^-{\alpha(\ass(\tau,\pi,\varpi))} && 
		\alpha(\tau\otimes(\pi\otimes\varpi))
	}
\end{align*}
is commutative for all $\tau,\pi,\varpi\in\cC$. %, where $\ass(\tau,\pi,\varpi)\colon (\tau\otimes\pi)\otimes\varpi\to \tau\otimes(\pi\otimes\varpi)$ is the unitary associator of $\cC$. 
We also call such a triple $(A,\alpha,\fu)$ a \emph{$\cC$-C*-algebra}. 
For $\cC$-C*-algebras $(A,\alpha,\fu)$ and $(B,\beta,\fv)$, a $\cC$-$(A,B)$-correspondence is an $(A,B)$-correspondence $(E,\phi)$ equipped with a family $\bv=(\bv_\pi)_{\pi\in\cC}$ of (possibly non-adjointable) $(A,B)$-bilinear isometries $\bv_\pi\colon \alpha(\pi)\otimes_A E\to E\otimes_B\beta(\pi)$ satisfying the naturality such that the diagram 
\begin{align*}
	\xymatrix@C=3em{
		\alpha(\pi)\otimes_A\alpha(\varpi)\otimes_AE 
		\ar[r]^-{1\otimes\bv_\varpi} \ar[d]_-{\fu_{\pi,\varpi}\otimes1}& 
		\alpha(\pi)\otimes_AE\otimes_B\beta(\varpi)
		\ar[r]^-{\bv_\pi\otimes1}&
		E\otimes_B\beta(\pi)\otimes_B\beta(\varpi) 
		\ar[d]^-{1\otimes\fv_{\pi,\varpi}}\\
		\alpha(\pi\otimes\varpi)\otimes_AE
		\ar[rr]^-{\bv_{\pi\otimes\varpi}}&&
		E\otimes_B\beta(\pi\otimes\varpi)
	}
\end{align*}
is commutative for all $\pi,\varpi\in\cC$. 
For a $\cC$-$(A,B)$-correspondence of the form $(B,\phi,\bv)$, we say that the pair $(\phi,\bv)\colon A\to\cM(B)$ is a \emph{$\cC$-$*$-homomorphism}. 
For basic facts about $\cC$-actions, we refer to \cite{Arano-Kitamura-Kubota:tensorKK}, whose conventions we mostly follow. 
One of the few exceptions is that we shall write $\alpha_\pi$ for the left action $\alpha_\pi\colon A\to\cL(\alpha(\pi))$. 
Note that whenever $\pi\neq 0$, the $(A,A)$-correspondence $\alpha(\pi)$ is full, non-degenerate, faithful, and proper. 
Also, $\bv_\pi$ are unitary for a $\cC$-C*-correspondence $(E,\phi,\bv)$ with $(E,\phi)$ non-degenerate. 
For C*-algebras $A$, $B$, and an imprimitivity $(A,B)$-bimodule $M$, any $\cC$-action on $A$ induces a $\cC$-action on $B$ via the unitary monoidal equivalence $\Cor(A)\simeq\Cor(B)$ given by tensoring $M$, and $M$ becomes a $\cC$-$(A,B)$-correspondence canonically. 
See \cite[Lemma 2.26, Lemma 2.31, Remark 2.27]{Arano-Kitamura-Kubota:tensorKK}, respectively. 

\subsection{Actions on Pimsner algebras}\label{ssec:equivTPalg}
Let $A$ be a C*-algebra, $E=(E,\phi)$, $M=(M,\alpha)$ be non-degenerate faithful $(A,A)$-correspondences, and $\bv\in\cU\cL(M\otimes_A E, E\otimes_A M)$ be an $(A,A)$-bilinear unitary. 
We define the $(A,A)$-bilinear unitary $\bv^{(n)}\colon M\otimes_A E^{\otimes_A n}\to E^{\otimes_A n}\otimes_A M$ inductively on $n\in\bZ_{\geq 0}$ by $\bv^{(0)}:=1_M$ and $\bv^{(n+1)}:=(1_E\otimes_A\bv^{(n)})(\bv\otimes_A1_{E^{\otimes n}})$. 
In particular, $\bv^{(1)}=\bv$. 
We define the $(A,A)$-bilinear unitary $\bv^{(\infty)}:=\bigoplus_{n=0}^{\infty}\bv^{(n)}\colon M\otimes_A\cF_E\to \cF_E\otimes_A M$. 
The following lemma gives a correspondence over $\cT_E$ coming from $M$. 

\begin{lem}\label{lem:bimodFock}
	Let $A$ be a C*-algebra, $E=(E,\phi)$, $M=(M,\alpha)$ be non-degenerate faithful $(A,A)$-correspondences, and $\bv\in\cU\cL(M\otimes_A E, E\otimes_A M)$ be an $(A,A)$-bilinear unitary. 
	Suppose that $M$ is full and proper. 
	Then, we have a well-defined $(\cT_E,\cT_E)$-correspondence $M^\cT=(M^\cT,\alpha^\cT)$ such that 
	\begin{itemize}
		\item
		$M^\cT:=\cspan\{ \bv^{(\infty)} T_\xi x \mid \xi\in M, x\in\cT_E \} 
		\subset \cL(\cF_E,\cF_E\otimes_AM)$ as a right $\cT_E$-submodule,
		\item
		$\bra x,y\ket_{M^\cT}:=x^*y \in \cL(\cF_E)$ for $x,y\in M^\cT$, and 
		\item
		$\alpha^\cT(x) \xi:= (x\otimes_A1_M)\xi \in \cL(\cF_E,\cF_E\otimes_AM)$ for $\xi\in M^\cT$, $x\in\cT_E$. 
	\end{itemize}
	Moreover, $M^\cT$ is proper, full, non-degenerate, and faithful. 
\end{lem}

\begin{proof}
	It is easy to see that $\xi^*\eta\in\cT_E$ for any $\xi,\eta\in M^\cT\subset \cL(\cF_E,\cF_E\otimes_AM)$, and that the Hilbert $\cT_E$-module $M^\cT$ is unitary equivalent to the interior tensor product $M\otimes_A\cT_E$ by 
	\begin{align*}
		&
		M\otimes_A\cT_E \ni a\otimes x\mapsto \bv^{(\infty)} T_a x \in M^\cT. 
	\end{align*}
	In particular, $M^\cT$ is a full Hilbert $\cT_E$-module since $M$ and $M\otimes_A\cT_E$ are full Hilbert $A$- and $\cT_E$-modules, respectively. 
	
	We are going to show that the left $\cT_E$-action on $M^\cT$ is well-defined and non-degenerate. 
	First, $A\subset \cT_E$ acts on $M^\cT$ non-degenerately since it acts on $M$ non-degenerately and $\bv^{(\infty)}$ is left $A$-linear. 
	We have the following equality as subsets of $\cL(\cF_E,\cF_E\otimes_A M)$ up to closed linear spans, 
	\begin{align*}
		&
		T_{E}\bv^{(\infty)} T_{M} 
		=
		(1_{E}\otimes_A\bv^{(\infty)}) T_{E\otimes M} 
		=
		(1_{E}\otimes_A\bv^{(\infty)}) T_{\bv (M\otimes E)} 
		=
		\bv^{(\infty)} T_{M} T_{E} . 
	\end{align*}
	By using the projection $P_0\in\cL(\cF_E)$ from \cref{notn:simpleCPalg} (4), we see the following inclusion as subsets of $\cL(\cF_E,\cF_E\otimes_A M)$ up to closed linear spans, 
	\begin{align*}
		&
		T_{E}^* \bv^{(\infty)} T_{M} 
		=
		T_{E}^* \bv^{(\infty)} T_{M} (1-P_0) 
		=
		\bv^{(\infty)} T_E^* (\bv\otimes1_\cF) T_{M} (1-P_0) 
		\\\overset{\bigstar}{\subset}{}& 
		\bv^{(\infty)} T_{M} T_{M}^* T_E^* (\bv\otimes1_\cF) T_{M} (1-P_0) 
		=
		\bv^{(\infty)} T_{M} T_{\bv^*(E\otimes M)}^* T_{M} (1-P_0) 
		\\={}&
		\bv^{(\infty)} T_{M} T_{E}^* T_{M}^* T_{M} (1-P_0) 
		=
		\bv^{(\infty)} T_{M} T_E^* (1-P_0) 
		=
		\bv^{(\infty)} T_{M} T_E^* , 
	\end{align*}
	where for $\overset{\bigstar}{\subset}$ we have used the non-degenerate inclusion $\alpha(A)\subset\cK(M)$. 
	It follows $\cT_E M^\cT \subset M^\cT$. 
	Finally, the faithfulness of $M^\cT$ follows from the fact that $X\otimes_A1_M\neq 0$ for $0\neq X\in\cL(E)$ by the faithfulness of $M$, 
	and the properness of $M^\cT$ holds by combining the properness of $M$, the $(A,A)$-bilinearity of $\bv_\pi^{(\infty)}$, and the non-degeneracy of $A\subset\cT_E$. 
\end{proof}

We want to apply this construction to $(M,\bv)=(\alpha(\pi),\bv_\pi)$ for a $\cC$-action $(\alpha,\fu)$ on $A$ and a $\cC$-$(A,A)$-correspondence $(E,\phi,\bv)$. 

\begin{thm}\label{thm:equivTPalg}
	Let $\cC$ be a unitary tensor category, $(A,\alpha,\fu)$ be a $\cC$-C*-algebra, and $(E,\phi,\bv)$ be a non-degenerate faithful $\cC$-$(A,A)$-correspondence. 
	Then $\cT_E$ has a well-defined $\cC$-action such that $\cF_E$ is a $\cC$-$(\cT_E,A)$-correspondence with a canonical family of unitaries denoted by $\bv^\cT$. 
	The canonical inclusion gives a non-degenerate injective $\cC$-$*$-homomorphism $(\iota,\bu)\colon A\to \cT_E$. 
\end{thm}

\begin{proof}
	Note that $\cF_E$ is a well-defined $\cC$-$(A,A)$-correspondence with $\bv_\pi^{(\infty)}\colon \alpha(\pi)\otimes_A \cF_E\to \cF_E\otimes_A \alpha(\pi)$ as the $\ell^2$-direct sum of $\cC$-$(A,A)$-correspondences $E^{\otimes_A n}$. 
	We write $(\alpha^\cF,\fu^\cF)$ for the $\cC$-action on $\cK(\cF_E)$ defined by $\alpha^\cF(\pi):=\cK(\cF_E,\cF_E\otimes_A\alpha(\pi))$ and 
	$\fu^\cF_{\pi,\varpi}(x\otimes_{\cK(\cF_E)} y) := (1_{\cF_E}\otimes_A\fu_{\pi,\varpi})(x\otimes_A1_{\alpha(\varpi)})y$ 
	for $\pi,\varpi\in\cC$, $x\in \cK(\cF_E,\cF_E\otimes_A\alpha(\pi))$, and $y\in\cK(\cF_E,\cF_E\otimes_A\alpha(\varpi))$.
	
	For $\pi\in\cC\setminus\{0\}$, we let $\alpha^\cT(\pi):=\alpha(\pi)^\cT$ be the non-degenerate $(\cT_E,\cT_E)$-correspondence defined by \cref{lem:bimodFock}. 
	We put $\alpha^\cT(0):=0$. 
	For $\pi,\varpi\in\cC$, we observe that there is an obvious isometric embedding of the Banach $(\cT_E,\cT_E)$-bimodule preserving inner products, 
	\begin{align*}
		\fs_{\pi,\varpi}\colon \alpha^\cT(\pi)\otimes_{\cT_E}\alpha^\cT(\varpi)
		\to{}&
		\cL(\cF_E,\cF_E\otimes_A \alpha(\pi)\otimes_A \alpha(\varpi)) , 
		\\
		( x \otimes_{\cT_E} y )
		\mapsto{}&
		( x \otimes_A 1_{\alpha(\varpi)}) y 
	\end{align*}
	whose image is $\cspan (\alpha^\cT(\pi)\otimes1_{\alpha(\varpi)})\alpha^\cT(\varpi)$. 
	Using the $(A,A)$-bilinearity of $\bv^{(\infty)}_\varpi$, we check, up to closed linear spans, 
	\begin{align*}
		&
		\fu^\cF_{\pi,\varpi} (\alpha^\cT(\pi)\otimes_A 1_{\alpha(\varpi)})\alpha^\cT(\varpi) 
		\\={}&
		\fu^\cF_{\pi,\varpi} \bigl( (\bv^{(\infty)}_\pi T_{\alpha(\pi)} \cT_E) \otimes_A 1_{\alpha(\varpi)} \bigr) \alpha^\cT(\varpi)
		\\={}&
		\bv^{(\infty)}_{\pi\otimes\varpi} (\fu_{\pi,\varpi}\otimes1_{\cF_E}) (1_{\alpha(\pi)}\otimes\bv^{(\infty) *}_\varpi) (T_{\alpha(\pi)}\otimes_A1_{\cF_E\otimes\alpha(\varpi)}) (\cT_E\otimes_A 1_{\alpha(\varpi)}) \alpha^\cT(\varpi)
		\\={}&
		\bv^{(\infty)}_{\pi\otimes\varpi} (\fu_{\pi,\varpi}\otimes1_{\cF_E}) (T_{\alpha(\pi)}\otimes_A1_{\alpha(\varpi)\otimes\cF_E}) \bv^{(\infty) *}_\varpi (\cT_E\otimes_A1_{\alpha(\varpi)}) \alpha^\cT(\varpi)
		\\={}&
		\bv^{(\infty)}_{\pi\otimes\varpi} (\fu_{\pi,\varpi}\otimes1_{\cF_E}) (T_{\alpha(\pi)}\otimes_A T_{\alpha(\varpi)}\otimes_A \cT_E) ,
	\end{align*}
	where for the last equality we have used the non-degeneracy of $\alpha^\cT(\varpi)$. 
	It follows that $\fu^\cF_{\pi,\varpi} \fs_{\pi,\varpi} (\alpha^\cT(\pi)\otimes_{\cT_E}\alpha^\cT(\varpi)) = \alpha^\cT(\pi\otimes\varpi)$. 
	We can check that the family of $(\cT_E,\cT_E)$-bilinear unitaries $\fu^\cT:=\fu^\cF\fs$ satisfies the pentagonal axiom. We obtain a $\cC$-action $(\alpha^\cT,\fu^\cT)$ on $\cT_E$, which, by construction, makes $\cF_E$ a $\cC$-$(\cT_E,A)$-correspondence with the canonical family of unitaries 
	\begin{align*}
		\bv^{\cT}_\pi \colon \alpha^\cT(\pi)\otimes_{\cT_E}\cF_E \ni 
		( \bv^{(\infty)}_\pi T_a x )\otimes_{\cT_E} \xi \mapsto{}& \bv^{(\infty)}_\pi ( a \otimes_{A} x\xi )
		\in \cF_E\otimes_A\alpha(\pi) 
	\end{align*}
	for $\xi\in\cF_E$, $x\in \cT_E$, $a\in\alpha(\pi)$. 
	Finally, it is routine to check that the inclusion $\iota\colon A\subset\cT_E$ is $\cC$-equivariant with $\bu_\pi\colon \alpha(\pi)\otimes_A\cT_E\ni a\otimes x \mapsto \bv^{(\infty)}_\pi(T_a x)\in\alpha^\cT(\pi) = A\otimes_A\alpha^{\cT}(\pi)$. 
\end{proof}

\begin{prop}\label{prop:KKCequiv}
	In the setting of \cref{thm:equivTPalg}, if $A$ and $E$ are separable and $\cC$ is countable, then $\cT_E$ is $\KK^\cC$-equivalent to $A$ by $(\iota,\bu)$. 
\end{prop}

We refer to \cite{Arano-Kitamura-Kubota:tensorKK} also for the definition and properties of $\cC$-equivariant $\KK$-theory. 

\begin{proof}
	We consider the $\cC$-$(\cT_E,A)$-correspondences $(\cF_E^+,\bv^{\cT+}) := (\cF_E,\id_{\cT_E},\bv^{\cT})$, $(\cF_E^-,\bv^{\cT-}) := \cF_E^+ \otimes_A (E,\phi,\bv)$ and the triple 
	$(\cF^{\pm}_E,\bv^{\cT\pm},F)$, 
	where we wrote $F:=\left( \begin{smallmatrix}0&P^*\\P&0\end{smallmatrix} \right)$ using the projection onto the direct summand $P\colon\cF_E\to \cF^-_E=\bigoplus_{n=1}^{\infty}E^{\otimes_A n}$, that is, $P:=1_{\cF_E}-P_0$ in the convention of \cref{notn:simpleCPalg} (4). 
	Then using $P_0\cT_E,\cT_EP_0\subset\cK(\cF_E)$, we can see that $(\cF^{\pm}_E,\bv^{\cT\pm},\left( \begin{smallmatrix}0&P^*\\P&0\end{smallmatrix} \right))$ is a $\cC$-Kasparov $(\cT_E,A)$-bimodule. 
	We will check that its $\KK^\cC$-inverse can be given by the inclusion $(\iota,\bu)\colon A\subset \cT_E$. 
	
	Since $F$ itself is an $F$-connection for $A\otimes_\iota \cF^{\pm}_E\cong \cF^{\pm}_E$, we see 
	\begin{align*}
		&
		(\cT_E,\iota,\bu,0)\mathbin{\wh{\otimes}}_{\cT_E}(\cF^{\pm}_E,\bv^{\cT\pm},F) \simeq_\cC (\cF^{\pm}_E, \bv^{(\infty)\oplus2}, F) 
		\\={}& 
		(\cF^{-\oplus2}_E, \bv^{(\infty)\oplus2}, \left( \begin{smallmatrix} 0&1\\1&0 \end{smallmatrix} \right)) \oplus (\id_A,0) \simeq_\cC (\id_A,0) , 
	\end{align*}
	where $\wh{\otimes}$ indicates the $\cC$-equivariant Kasparov product and $\simeq_\cC$ means $\cC$-equivariant homotopy of $\cC$-Kasparov bimodules. 
	For $t\in[0,1]$, the universality of $\cT_E$ induces the well-defined $*$-homomorphism 
	$\Phi_t\colon \cT_E \to \cL(\cF_E\otimes_\iota\cT_E)\cong \cL(\cT_E\oplus \cF^-_E\otimes_\iota\cT_E)$ such that $\Phi_t(a):=a\otimes1$ for $a\in A$ and that 
	\begin{align*}
		&
		\Phi_t(\tS_\xi) 
		:= 
		\left(\begin{array}{cc}
			\sqrt{t} \tS_\xi & 0 \\ \sqrt{1-t} (T_\xi P_0) \otimes_\iota 1_{\cT_E} & (T_\xi P) \otimes_\iota 1_{\cT_E}
		\end{array}\right)
		\in \cL\bigl(\cT_E\oplus (\cF_E^-\otimes_\iota\cT_E)\bigr)
	\end{align*}
	for $\xi\in E$. 
	For $\pi\in\cC$, we define the $(A,\cT_E)$-bilinear unitary 
	\begin{align*}
		\bV_{t,\pi}\colon 
		\alpha^\cT(\pi) \otimes_{\Phi_t} (\cF_E\otimes_\iota\cT_E) 
		\to{}&
		\cF_E\otimes_A \alpha^\cT(\pi), 
		\\
		( \bv^{(\infty)}_\pi(a\otimes_A1_{\cT_E}) )\otimes_{\Phi_t} (\xi\otimes_\iota x)\mapsto{}&
		(\bv^{(\infty)}_\pi)^{(2)} (a\otimes_A\xi\otimes_\iota x) 
	\end{align*}
	for $a\in\alpha(\pi)$, $\xi\in\cF_E$, $x\in\cT_E$, where we recall $(\bv^{(\infty)}_\pi)^{(2)}\in\cL(\alpha(\pi)\otimes_A\cF_E^{\otimes_A2},\cF_E^{\otimes_A2}\otimes_A\alpha(\pi))$. 
	
	To see that the unitary $\bV_t$ is left $\cT_E$-linear, we only need to check that $\bV_t$ commutes with the left actions of $\tS_\xi$ for all $\xi\in E$. 
	For any $\e>0$, there are finitely many $a_1,\cdots,a_n,b_1,\cdots,b_n\in\alpha(\pi)$ and $\xi_1,\cdots,\xi_n\in E$ such that $\|\bv_\pi^*T_{\xi} - \sum_{i=1}^{n} T_{a_i} T_{\xi_{i}} T_{b_i}^* \|<\e$ inside $\cK(\alpha(\pi),\alpha(\pi)\otimes_A E)$. 
	Then $\|T_\xi\bv_\pi^{(\infty)} - \bv_\pi^{(\infty)}\sum_{i=1}^{n} T_{a_i} T_{\xi_{i}} T_{b_i}^* \|<\e$ inside $\cL(\alpha(\pi)\otimes_A\cF_E)$. 
	For contractive elements $c\in\alpha(\pi)$, $x\in\cT_E$, $\eta_0\in A$, and $\eta_+\in\cF_E^-$, we put $\eta:=\eta_0+\eta_+\in\cF_E$ and see 
	\begin{align*}
		&
		\Bigl\| 
		\bigl( (\Phi_t\otimes_{\cT_E}1_{\alpha^\cT(\pi)}) ( \tS_\xi ) \bigr) 
		\bV_{t,\pi} \bigl( (\bv_\pi^{(\infty)} c\otimes_A1_{\cT_E}) \otimes_{\Phi_t} (\eta\otimes_A x) \bigr) 
		\\&-
		\bV_{t,\pi} \Bigl( 
		\bigl( (T_\xi\otimes_A1_{\alpha(\pi)}) \bv_\pi^{(\infty)} ( c\otimes_A1_{\cT_E}) \bigr) \otimes_{\Phi_t} (\eta\otimes_A x) 
		\Bigr) \Bigr\| 
		\\\leq{}& 
		\Bigl\| 
		\bigl( (\Phi_t\otimes_{\cT_E}1_{\alpha^\cT(\pi)}) ( \tS_\xi ) \bigr) 
		(\bv_\pi^{(\infty)})^{(2)} (c\otimes_A \eta\otimes_A x) 
		\\&-
		\bV_{t,\pi} \Bigl( \sum_{i=1}^{n} 
		\bigl( \bv_\pi^{(\infty)} (T_{a_i} \otimes_A \tS_{\xi_i} \bra b_i, c\ket) \bigr) \otimes_{\Phi_t} (\eta\otimes_A x) 
		\Bigr) \Bigr\| +2\e 
		\\\leq{}& 
		\Bigl\| 
		\xi\otimes_A \bigl( (\bv_\pi^{(\infty)})^{(2)} (c\otimes_A (\eta_+ +\sqrt{1-t}\eta_0)\otimes_A x) \bigr) 
		+ \sqrt{t} 
		(T_\xi\otimes_A1_{\alpha(\pi)})\bv_\pi^{(\infty)} (c \otimes_A (\eta_0 x) ) 
		\\&-
		\sum_{i=1}^{n} 
		(\bv_\pi^{(\infty)})^{(2)} \bigl( a_i  \otimes_A \Phi_t(\tS_{\xi_i})( \bra b_i, c\ket \eta\otimes_A x ) \bigr) 
		\Bigr\| +2\e 
		\\\leq{}&5\e. 
	\end{align*}
	Since $\e$ is arbitrary, the $(\cT_E,\cT_E)$-bilinearity of $\bV_{t,\pi}$ follows. 
	
	With the aid of $(\Phi_t(x)(P_0\otimes1))_{t\in[0,1]},((P_0\otimes1)\Phi_t(x))_{t\in[0,1]}\in C[0,1]\otimes\cK(\cF_E\otimes_\iota\cT_E)$ for all $x\in\cT_E$, we can check that $\big( (\cF_E^{\pm}\otimes_\iota\cT_E, \Phi_t \oplus P\Phi_{1}P^*, \bV_t\oplus\bV_1, F\otimes1_{\cT_E}) \big)_{t\in[0,1]}$ is a $\cC$-Kasparov $(\cT_E,C[0,1]\otimes\cT_E)$-bimodule. 
	This shows the second $\cC$-equivariant homotopy in the following, 
	\begin{align*}
		&
		(\cF^{\pm}_E,\bv^{\cT\pm},F)\mathbin{\widehat{\otimes}}_A (\cT_E,\iota,\bu,0) 
		\simeq_\cC
		(\cF_E^{\pm}\otimes_\iota\cT_E, \Phi_0 \oplus P\Phi_{1}P^*, \bV_0\oplus\bV_1, F\otimes1_{\cT_E}) 
		\\\simeq_\cC{}&
		(\cF_E^{\pm}\otimes_\iota\cT_E, \Phi_1 \oplus P\Phi_1P^*, \bV_1\oplus\bV_1, F\otimes1_{\cT_E}) 
		\simeq_\cC \id_{\cT_E}. 
	\qedhere\end{align*}
\end{proof}

\subsection{Regular equivariant correspondences}\label{ssec:regularcorr}

To apply the Pimsner construction to $\cC$-equivariant settings, we enhance the construction of the regular unitary half-braiding \cite[Theorem 3.4]{Neshveyev-Yamashita:Drinfeld}, which originally produces an object in the Drinfeld centre of the ind-completion of a unitary tensor category, into the situation with coefficients of C*-correspondences. 
It should be noted that the same construction is also utilized in \cite[Section 4.1]{Henriques-Penneys2017:bicommutant}. 

\begin{prop}\label{prop:regularcorr}
	For a $\cC$-C*-algebra and a non-degenerate $(A,A)$-correspondence $(E,\phi)$, 
	there is a canonical $\cC$-equivariant structure on the $(A,A)$-correspondence $\ell^2(\cC;E):=\bigoplus_{\pi\in\Irr\cC}\alpha(\pi)\otimes_A E\otimes_A \alpha(\overline{\pi})$. 
\end{prop}

\begin{proof}
	Via unitary monoidal equivalence, we may assume that $\cC$ is strict for simplicity. 
	We fix the standard solution $R_\pi\colon\mathbbm{1}_\cC\to\overline{\pi}\otimes\pi$, $\overline{R}_\pi\colon\mathbbm{1}_\cC\to\pi\otimes\overline{\pi}$ of the conjugate equation for each $\pi\in\Irr\cC$. 
	
	For $\tau\in\cC$ and $\pi,\varpi\in\Irr\cC$, 
	we regard $\Hom_\cC(\varpi,\tau\otimes\pi)$ as a Hilbert space with the inner product $\bra f,g\ket:=f^*g\in \Hom_\cC(\varpi,\varpi)\cong \bC$ for $f,g\colon\varpi\to\tau\otimes\pi$ 
	and, for its orthonormal basis $(s_i)_i$, let 
	\begin{align*}
		\overline{s_i}:=\frac{\|R_\pi\|}{\|R_\varpi\|}(\id_{\overline{\varpi}\otimes \tau}\otimes \overline{R}_\pi^*) (\id_{\overline{\varpi}}\otimes s_i\otimes \id_{\overline{\pi}}) (R_\varpi\otimes\id_{\overline{\pi}}) , 
	\end{align*}
	which form an orthonormal basis $(\overline{s_i})_i$ of $\Hom_\cC(\overline{\pi},\overline{\varpi}\otimes\tau)$. 
	We define 
	\begin{align*}
		\bU_{\tau,\pi,\varpi} 
		&:= 
		\sum_{i} (\alpha(s_i^*)\fu_{\tau,\pi}) \otimes_A 1_E \otimes_A (\fu_{\overline{\varpi},\tau}^*\alpha( \overline{s_i} ))
		\\&
		\colon \alpha(\tau) \otimes_A \alpha(\pi) \otimes_A E \otimes_A \alpha(\overline{\pi}) \to \alpha(\varpi) \otimes_A E \otimes_A \alpha(\overline{\varpi}) \otimes_A \alpha(\tau), 
	\end{align*}
	which does not depend on the choice of $(s_i)_i$, and 
	\begin{align*}
		&
		\bU_\tau:=\bigoplus_{\pi,\varpi\in\Irr\cC} \bU_{\tau,\pi,\varpi} \colon \alpha(\tau)\otimes_A \ell^2(\cC;E) \to \ell^2(\cC;E) \otimes_A \alpha(\tau). 
	\end{align*}
	Then $\bU_\tau$ is a well-defined $(A,A)$-bilinear unitary 
	since the fact that $(\overline{s_i})_i$ is an orthonormal basis shows that each $\bU_{\tau,\pi,\varpi}$ is a unitary between the direct summands of the domain and codomain of $\bU_\tau$ that are isomorphic to $(\alpha(\varpi)\otimes_AE\otimes_A\alpha(\overline{\pi}))^{\oplus\dim_\bC\Hom_\cC(\varpi,\tau\otimes\pi)}$. 
	It is routine to check the naturality of $(\bU_\tau)_{\tau\in\cC}$ and the cocycle relation 
	$(\bU_\sigma \otimes_A 1_{\alpha(\tau)})(1_{\alpha(\sigma)} \otimes_A \bU_\tau) = \fu_{\sigma,\tau}^*\bU_{\sigma\otimes \tau}\fu_{\sigma,\tau}$
	for $\sigma,\tau\in\cC$ using the independence of the choice of $(s_i)_i$ as in the proof of \cite[Theorem 3.4]{Neshveyev-Yamashita:Drinfeld}. 
\end{proof}

\begin{prop}\label{prop:tensorKirch}
	Let $\cC$ be a countable unitary tensor category. 
	For any non-zero separable $\cC$-C*-algebra $(A,\alpha,\fu)$, there is a separable simple purely infinite C*-algebra $B$ with a $\cC$-action and a non-degenerate injective $\cC$-$*$-homomorphism $(\iota,\bu)\colon A\to B$ giving a $\KK^\cC$-equivalence. 
	If moreover $A$ is nuclear, we can arrange $B$ to be nuclear. 
\end{prop}

\begin{proof}
	Let $\phi\colon A\to \cB(\cH)$ be a non-degenerate faithful $*$-representation of $A$ on a separable Hilbert space $\cH$. By taking $\cH^{\oplus\infty}$ if necessary, we may assume that $\phi(A)\cap\cK(\cH)=0$. Then we apply \cref{prop:regularcorr} to the $(A,A)$-correspondence $(\cH\otimes_\bC A,\phi\otimes1)$ to get a $\cC$-$(A,A)$-correspondence $E:=\ell^2(\cC;\cH\otimes_\bC A)$, which is non-degenerate, faithful, and satisfying (i) of \cref{cor:simpleCPalg}. 
	It follows from \cref{thm:equivTPalg,prop:KKCequiv} that the non-degenerate inclusion $A\to \cT_{E}$ is a $\KK^\cC$-equivalence. 
	
	We check (ii) in \cref{cor:simpleCPalg} as, for any non-zero ideal $I\subset A$, we have a non-zero vector $\xi\in \phi(I)\cH$ and see that 
	$\bra E, IE\ket \supset \bra \cH\otimes_\bC A,\xi\otimes_\bC A\ket=A$. Thus $\cT_{E}$ is simple. 
	
	Now we put $B:=\cT_{E}\otimes \cO_\infty$. The non-degenerate inclusion $\id\otimes 1_{\cO_\infty}\colon \cT_{E}\to B$ is a $\KK^\cC$-equivalence by \cite[Theorem 3.19]{Arano-Kitamura-Kubota:tensorKK} with the aid of the $\KK$-equivalence $\bC\subset\cO_{\infty}$. 
	Thus the composition $A\subset \cT_{E}\subset B$ is the desired non-degenerate injective $\cC$-$*$-homomorphism. 
	Here, $B$ is simple and purely infinite by \cite[Lemma 2.3]{Izumi:finiteI}. 
	Finally, $B:=\cT_{E}\otimes\cO_\infty$ is nuclear if so is $A$. 
\end{proof}

\section{Outer actions on Kirchberg algebras}\label{sec:outer}

We want to modify the $\cC$-action on $\cT_E$ to be outer in the following sense. 

\begin{df}\label{def:irroutfull}
	Consider an action $(\alpha,\fu)$ of a unitary tensor category $\cC$ on a C*-algebra $A$. 
	\begin{enumerate}
		\item
		We say that $(\alpha,\fu)$ is \emph{irreducible} if the $(A,A)$-correspondence $\alpha(\pi)$ is irreducible for all $\pi\in\Irr\cC$. 
		\item
		We say that $(\alpha,\fu)$ is \emph{(objectwise) injective} if there is no $(A,A)$-bilinear unitary $\alpha(\pi)\to \alpha(\varpi)$ for all $\pi,\varpi\in\Irr\cC$ with $\pi\neq\varpi$. 
		\item
		We say that $(\alpha,\fu)$ is \emph{(pointwise) outer} if it is full as a unitary tensor functor $\cC\to\Cor(A)$. 
	\end{enumerate}
\end{df}

\begin{rem}\label{rem:irroutfull}
	Consider an action $(\alpha,\fu)$ of a unitary tensor category $\cC$ on a non-zero C*-algebra $A$. 
	\begin{enumerate}
		\item
		Then the functor $(\alpha,\fu)\colon\cC\to\Cor(A)$ is always faithful thanks to the semisimplicity and rigidity of $\cC$. 
		\item
		If $(\alpha,\fu)$ is irreducible, then any non-zero element $x\in \Hom_{\Cor(A)}(\alpha(\pi),\alpha(\varpi))$ is a unitary up to a scalar multiplication for all $\pi,\varpi\in\Irr\cC$, where we recall that $\Hom_{\Cor(A)}(\alpha(\pi),\alpha(\varpi))$ is the set of $(A,A)$-bilinear elements in $\cL(\alpha(\pi),\alpha(\varpi))$. 
		Indeed, the irreducibility of $\alpha(\pi)$ and $\alpha(\varpi)$ shows that both $x^*x$ and $xx^*$ are scalars. Thus $\|x\|^{-1}x$ is a unitary and the claim follows. 
		\item
		Hence, a $\cC$-action $(\alpha,\fu)$ on the C*-algebra $A$ is outer if and only if $\Hom_{\Cor(A)}(\alpha(\pi),\alpha(\varpi)) = \delta_{\pi,\varpi}\bC1_{\alpha(\pi)}$ for all $\pi,\varpi\in\Irr\cC$, which is equivalent to saying that $(\alpha,\fu)$ is both irreducible and injective. 
	\end{enumerate}
\end{rem}

\subsection{Modification trick}\label{ssec:outer}
\begin{rem}\label{notn:outer}
	In this subsection, we work with the following situation. 
	Consider a C*-algebra $A$ and a full non-degenerate faithful $(A,A)$-correspondence $E=(E,\phi)$. 
	We assume the following. 
	\begin{enumerate}
		\item
		$E$ is of the form $\bigoplus_{\lambda\in\Lambda}\cH_{\lambda}\otimes_\bC N_{\lambda}$, %with $\phi = \bigoplus_{\lambda\in\Lambda}\phi_\lambda\otimes 1_{N_\lambda}$, 
		where 
		$\cH_{\lambda}$ is a Hilbert space, $\phi_{\lambda}\colon A\to\cB(\cH_{\lambda})$ is a non-degenerate faithful $*$-representation of $A$, 
		and $N_{\lambda}$ is a full Hilbert $A$-module for each $\lambda$ in a countable non-empty set $\Lambda$. 
		\item
		$\phi_{\lambda}(A)\cap \cK(\cH_{\lambda})=0$ for all $\lambda\in\Lambda$. 
		In particular, $\phi(A)\cap \cK(E)=0$. 
		\item
		$A$ is unital and admits no finite-dimensional quotient C*-algebra. 
	\end{enumerate}
	Since $\cT_E=\cO_E$ by (2), we again use the convention from \cref{notn:simpleCPalg}. 
\end{rem}

We are going to give a criterion (\cref{thm:simpleTPcorr}) of the irreducibility of bimodules in a similar spirit to \cref{thm:simpleCPalg}. 
However, the proof is based on different techniques because several arguments in \cref{thm:simpleCPalg} are no longer valid for bimodules. %such as the one in \cref{lem:simpleCPalg} (1) showing $A_{\leq n}\cap J\neq 0$ for some $n$. 
Alternatively, we will use the following trick that crucially relies on the condition \cref{notn:outer} (3) on $A$. 

\begin{lem}\label{lem:simpleTPcorr}
	In the setting of \cref{notn:outer}, for proper full non-degenerate faithful $(A,A)$-correspondences $M_i=(M_i,\alpha_i)$ with $i=1,2$, we have 
	\begin{align*}
		&
		\{ y\in \cL(M_1\otimes_A A_\infty, M_2\otimes_A A_\infty) \mid  
		y(\alpha_1(a)\otimes1) = (\alpha_2(a)\otimes1)y, \, \forall a\in A
		\} 
		\\={}&
		\{ x\otimes_A1_{A_\infty} \mid x\in \cK(M_1, M_2), 
		x\alpha_1(a)=\alpha_2(a)x, \, \forall a\in A \} . 
	\end{align*}
\end{lem}

Before the proof, note that $A_{\leq n}\cap A_{\geq n+1}=0$ for $n\in\bZ_{\geq 0}$. 
Thus the algebraic direct sum $\bigoplus_{n\geq 0}^{\mathrm{alg}}A_{n}$ is a norm-dense $(A,A)$-subbimodule of $A_\infty$. 

\begin{proof}
	By letting $(M,\alpha):=(M_1,\alpha_1)\oplus(M_2,\alpha_2)$ and looking at the corner, it suffices to show the statement when $M_1=M_2=M$. 
	
	Since the $*$-homomorphism $\alpha\colon A\to \cK(M)$ is unital by assumption, we see $\cK(M)$ is unital and thus 
	$\cK(M)=\cL(M)$. 
	Then the unital inclusion $A\subset A_\infty$ shows $1_M\otimes 1_{A_\infty}\in \cK(M\otimes_AA_{\infty})$ and thus 
	\begin{align*}
		&
		\cL(M\otimes_AA_{\infty}) 
		= 
		\cK(M\otimes_AA_{\infty}) 
		=
		\cspan\bigcup_{n=0}^{\infty}\cK(M\otimes_A A_n). 
	\end{align*}
	Here, since 
	\begin{align*}
		\cK(M\otimes_A A_n)
		\subset
		\{
		x\in \cK(M\otimes_A A_\infty) 
		\mid 
		T_\xi^*x^*xT_\xi \in A_n, \, \forall \xi\in M
		\}, 
	\end{align*}
	the following norm-dense $\cK(M)$-subbimodule is indeed a direct sum 
	\begin{align*}
		\bigoplus_{n\geq 0}^{\mathrm{alg}}\cK(M\otimes_A E^{\otimes_A n})
		\cong
		\bigoplus_{n\geq 0}^{\mathrm{alg}}\cK(M\otimes_A A_n)
		\subset 
		\cL(M\otimes_AA_{\infty}) , 
	\end{align*}
	where $\cK(M\otimes_A E^{\otimes_A n})\cong \cK(M\otimes_A A_n)$ because, by using $\sigma_n$ from \cref{lem:faithfulX}, $(E^{\otimes_A n},\sigma_n^{-1})$ is an imprimitivity $(A_n,A)$-bimodule satisfying $M\otimes_A A_n\otimes_{\sigma_n^{-1}} E^{\otimes_A n}\cong M\otimes_A E^{\otimes_A n}$, which induces the canonical $*$-isomorphism $\cK(M\otimes_A A_n)\ni x \mapsto x\otimes 1_{E^{\otimes n}}\in \cK(M\otimes_A A_n\otimes_{\sigma_n^{-1}} E^{\otimes_A n})\cong \cK(M\otimes_A E^{\otimes_A n})$. 
	By the same argument, we see $\bigoplus_{n=0}^{m}\cK(M\otimes_A A_n)= \cL(M\otimes_A A_{\leq m})$ for all $m\in\bZ_{\geq 0}$. 
	For the $(A,A)$-bilinear projection $P_{\leq m}\in\cL(\cF_E)$ onto $\bigoplus_{n=0}^{m}E^{\otimes_A n}$, we have the conditional expectation 
	\begin{align*}
		\Phi_m:=(1_M\otimes P_{\leq m})(-)(1_M\otimes P_{\leq m}) \colon 
		\cK(M\otimes_A A_\infty) \to \cK(M\otimes_A A_{\leq m}) , 
	\end{align*}
	such that $\| x-\Phi_m(x) \|\xto{m\to\infty} 0$ for all $x\in\cK(M\otimes_A A_\infty)$. Since $\Phi_m$ restricts to a conditional expectation 
	\begin{align*}
		(\alpha(A)\otimes1_{A_\infty})'\cap\cK(M\otimes_A A_\infty) \to (\alpha(A)\otimes1_{A_{\leq m}})'\cap\cK(M\otimes_A A_{\leq m}), 
	\end{align*}
	we see the norm-dense inclusion of $\bC$-vector spaces 
	\begin{align*}
		&
		\bigoplus_{n\geq 0}^{\mathrm{alg}} \bigl(
		(\alpha(A)\otimes1_{E^{\otimes n}})'\cap\cK(M\otimes_A E^{\otimes_A n}) \bigr)
		\subset 
		(\alpha(A)\otimes1_{A_\infty})'\cap \cL(M\otimes_AA_{\infty}) . 
	\end{align*}
	Thus it suffices to show that the summand of the left hand side is $0$ for all $n\geq 1$. 
	
	For $n\geq 0$, consider the Hilbert space $\wt{\cH}_{n} := M\otimes_AE^{\otimes_A n}\otimes_A \bigoplus_{\lambda\in\Lambda}\cH_{\lambda}$ with the unital $*$-representation of $A$ denoted by $\wt{\phi}_{n}:=\alpha\otimes1$. 
	Using an $(A,A)$-bilinear adjointable isometry 
	\begin{align*}
		M\otimes_A E^{\otimes_A n+1} 
		\cong 
		M\otimes_A E^{\otimes_A n} \otimes_A \bigoplus_{\lambda\in\Lambda}\cH_{\lambda} \otimes_\bC N_{\lambda} 
		\leq 
		\wt{\cH}_n \otimes_\bC \bigoplus_{\lambda\in\Lambda}N_{\lambda} , 
	\end{align*}
	we get an injective $*$-homomorphism 
	\begin{align}\label{eq:lem:simpleTPcorr2}
		(\alpha(A)\otimes1)'\cap\cK(M\otimes_A E^{\otimes_A n+1}) 
		\hookrightarrow 
		(\wt{\phi}_{n}(A)\otimes1)'\cap \Bigl( \cK(\wt{\cH}_{n}) \otimes \cK\Bigl(\bigoplus_{\lambda\in\Lambda}N_{\lambda}\Bigr) \Bigr) . 
	\end{align}
	If $\wt{\phi}_n(A)'\cap \cK(\wt{\cH}_n)$ contains a non-zero finite rank projection $p$, then we have a unital $*$-homomorphism $A\to \cB(p\wt{\cH}_n)$, whose range must be finite-dimensional, but this is forbidden by \cref{notn:outer} (3). 
	Thus $\wt{\phi}_n(A)'\cap \cK(\wt{\cH}_n)=0$, and a slicing argument shows that the right hand side of \cref{eq:lem:simpleTPcorr2} is $0$, which completes the proof. 
\end{proof}

\begin{prop}\label{thm:simpleTPcorr}
	In the setting of \cref{notn:outer}, for a proper full non-degenerate faithful $(A,A)$-correspondence $M=(M,\alpha)$ and an $(A,A)$-bilinear unitary $\bv\in\cL(M\otimes_A E,E\otimes_A M)$, the following are equivalent. 
	\begin{enumerate}
		\item[(i)]
		The $(\cT_E,\cT_E)$-correspondence $(M^\cT,\alpha^\cT)$ is irreducible. 
		\item[(ii)]
		There is no $(A,A)$-bilinear adjointable operator $x\in\cL(M)$ such that $\bv (x\otimes_A1_E) = (1_E\otimes_A x)\bv$ except scalars $\bC 1_M$. 
	\end{enumerate}
\end{prop}

\begin{rem}\label{notn:gauge}
	We consider the situation of \cref{thm:simpleTPcorr}. 
	\begin{enumerate}
		\item
		We have the $\bT$-actions on the Hilbert $A$-modules $\cF_E\otimes_AM$ and $M\otimes_A\cF_E$, for which we still write $\gamma$ by abusing notation, such that $\gamma_z(\xi\otimes\eta)=z^n(\xi\otimes\eta)$ and $\gamma_z(\eta\otimes\xi)=z^n(\eta\otimes\xi)$ for all $n\in\bZ_{\geq 0}$, $\xi\in E^{\otimes_A n}$, $\eta\in M$, and $z\in\bT$. By definition, $\bv^{(\infty)}$ preserves $\gamma$. Also, $\gamma$ induces the strictly continuous $\bT$-actions on $\cL(\cF_E\otimes_AM)$ and $\cL(M\otimes_A\cF_E)$, the C*-algebras of adjointable $A$-linear endomorphisms, again denoted by $\gamma$. 
		Also, we have the $\bT$-action on the Hilbert $\cT_E$-module $M\otimes_A\cT_E$ such that each $z\in \bT$ acts on $\eta\otimes x$ as $\eta\otimes\gamma_z(x)$ for all $\eta\in M$ and $x\in\cT_E$, which induces the strictly continuous $\bT$-actions on $\cL(M\otimes_A\cT_E)$, the C*-algebra of adjointable $\cT_E$-linear endomorphisms. 
		\item
		By construction, $\cT_E\ni x\mapsto x\otimes_A1_{M} \in \cL(\cF_E\otimes_AM)$ is $\bT$-equivariant with respect to $\gamma$, where we identify $x\in\cL(\cF_E)$ when we take $x\otimes_A1_{M}$. 
		%and $\cL(M\otimes_A\cF_E)$ contains $\cL(M\otimes_A\cT_E)$ as a $\bT$-equivariant C*-subalgebra. 
		In a similar spirit to $\cT_E\subset\cL(\cF_E)$, we regard $\cL(M\otimes_A\cT_E)$ as a $\bT$-equivariant C*-subalgebra of $\cL(M\otimes_A\cF_E)$ using the $*$-homomorphism $\cL(M\otimes_A\cT_E)\ni x\mapsto x\otimes_{\cT_E}1_{\cF_E}\in \cL(M\otimes_A\cT_E\otimes_{\cT_E}\cF_E)\cong\cL(M\otimes_A\cF_E)$, which is injective and $\bT$-equivariant. 
		Then for $m,n\in\bZ$, we have the well-defined contractive $\bC$-linear map 
		\begin{align}\label{eq:thm:simpleTPcorr1}\begin{aligned}
				\{ x\in \cL(M\otimes_A \cT_E) \mid \gamma_z(x)=z^mx, \forall z\in\bT \} 
				&\to 
				\cL(M\otimes_A X_{n},M\otimes_A X_{m+n}) , 
				\\
				x&\mapsto x|_{M\otimes X_n}
		\end{aligned}\end{align}
		which is injective for $n\leq 0$ because then $\cspan(M\otimes_AX_n)\cT_E=M\otimes_A\cT_E$ by the fullness of $E$. 
		\item
		Finally, we note that $\cL(M\otimes_A\cT_E)=\cK(M\otimes_A\cT_E)$ because of $1_M\in \cK(M)$ and the unital inclusion $A\subset \cT_E$, and that the algebraic direct sum of the left hand side of \cref{eq:thm:simpleTPcorr1} over all $m\in\bZ$ forms a dense $*$-subalgebra of $\cK(M\otimes_A\cT_E)$. 
	\end{enumerate}
\end{rem}

\begin{proof}
	By \cref{notn:gauge}, the gauge action $\gamma$ on $\cL(M\otimes_A\cT_E)$ is inherited by the C*-subalgebra 
	\[ D := \bigl(\bv^{(\infty)*}(\cT_E\otimes_A1_{M})\bv^{(\infty)}\bigr)' \cap \cL(M\otimes_A\cT_E) \cong \alpha^\cT(\cT_E)' \cap \cL(M^\cT) . \]
	Then, we have the injective $*$-homomorphism 
	\begin{align*}
		\Theta\colon 
		D_0 := \{ x\in \alpha(A)'\cap\cL(M) \mid \bv(x\otimes_A1_E) = (1_E\otimes_Ax)\bv \} 
		&\to D^{\gamma} , \\
		x&\mapsto x\otimes_A 1_{\cT_E}
	\end{align*}
	which is well-defined because for all $x\in D_0$, $\xi\in E$, and $\eta\in M\otimes_A E^{\otimes_A n}\leq M\otimes_A\cF_E$ with $n\in\bZ_{\geq 0}$, by identifying $\cL(M\otimes_A\cT_E)\subset\cL(M\otimes_A\cF_E)$ we have 
	\begin{align*}
		&
		\bv^{(\infty)*} (T_\xi\otimes 1_M) \bv^{(\infty)} (x\otimes 1_{\cT_E}) \eta 
		= \bv^{(n+1)*} (T_\xi\otimes 1_M) \bv^{(n)} (x\otimes 1_{\cT_E}) \eta 
		\\={}&
		\bv^{(n+1)*} (T_\xi\otimes 1_M) (1_{E^{\otimes n}}\otimes x) \bv^{(n)} \eta 
		= 
		\bv^{(n+1)*} (1_{E^{\otimes n+1}}\otimes x) ( \xi\otimes \bv^{(n)} \eta )
		\\={}&
		(x\otimes 1_{E^{\otimes n+1}}) \bv^{(n+1)*} T_\xi \bv^{(n)} \eta  
		= 
		(x\otimes 1_{\cT_E}) \bv^{(\infty)*} (T_\xi\otimes 1_M) \bv^{(\infty)} \eta . 
	\end{align*}
	
	We show the surjectivity of $\Theta$. 
	Using \cref{lem:simpleTPcorr} and \cref{eq:thm:simpleTPcorr1} we have the injective $*$-homomorphism 
	\begin{align*}
		&
		D^{\gamma}\hookrightarrow (\alpha(A)\otimes1_{A_\infty})'\cap \cL(M\otimes_AA_\infty) = (\alpha(A)'\cap \cK(M))\otimes_A1_{A_\infty}. 
	\end{align*}
	Therefore any element in $D^{\gamma}$ must be of the form $x\otimes1_{\cT_E}$ for uniquely determined $x\in \alpha(A)'\cap\cK(M)$. 
	For $\xi\in E\leq\cF_E$ and $\eta\in M\cong M\otimes_AA\leq M\otimes_A\cF_E$ we check, 
	\begin{align*}
		&
		\bv^*(1_E\otimes x)(\xi\otimes\eta) 
		= \bv^*(T_\xi\otimes1_M)(x\otimes1_{\cF_E})\eta 
		\\={}& 
		\bv^{(\infty)*}(T_\xi\otimes1_M)\bv^{(\infty)}(x\otimes1_{\cF_E})\eta 
		=
		(x\otimes1_{\cF_E})\bv^{(\infty)*}(T_\xi\otimes1_M)\bv^{(\infty)}\eta 
		\\={}&
		(x\otimes1_{\cF_E})\bv^{*}(\xi\otimes\eta) , 
	\end{align*}
	which shows $x\in D_0$. 
	Hence $\Theta\colon D_0\to D^{\gamma}$ is surjective and thus $*$-isomorphic. 
	Note that we have not used either (i) or (ii) so far. 
	
	Clearly, $D_0\cong\bC$ if $D\cong\bC$, which shows that (i) implies (ii). Conversely, we assume $D_0\cong \bC$. 
	Suppose that for some $0\neq m\in \bZ$, there is non-zero $w\in D$ with $\gamma_z(w)=z^{-m}w$ for all $z\in \bT$. 
	By considering $w^*$ if necessary, we may assume $m\geq 0$. 
	As in \cref{rem:irroutfull} (2), we see from $w^*w, ww^*\in D^{\gamma}=\bC1_D$ that $\|w\|^{-1}w\in \cU(D)$, which can be identified as an $(A_\infty,A_\infty)$-bilinear unitary in $\cU\cL(M\otimes_A A_\infty,M\otimes_A X_{-m})$ via \cref{eq:thm:simpleTPcorr1}. 
	But this contradicts the fact that $M\otimes X_{-m}$ is not full when $m\geq 1$. 
	Therefore we obtain $D=D^{\gamma}\cong\bC$, which shows that (ii) implies (i). 
\end{proof}

\begin{thm}\label{thm:tensorKirch}
	Let $\cC$ be a countable unitary tensor category and $(A,\alpha,\fu)$ be a non-zero separable $\cC$-C*-algebra. Then, there is a separable simple purely infinite C*-algebra $B$ with an outer $\cC$-action and a non-degenerate injective $\cC$-$*$-homomorphism $(\iota,\bu)\colon A\to B$ giving a $\KK^\cC$-equivalence. 
	If moreover $A$ is nuclear, we can arrange $B$ to be nuclear. 
\end{thm}

\begin{proof}
	By \cref{prop:tensorKirch}, we may further assume that $A$ is simple and purely infinite. 
	
	First, suppose that $A$ is unital. 
	Then, for a unital faithful $*$-representation $A\to\cB(\cH)$ on a Hilbert space $\cH$, the $\cC$-$(A,A)$-correspondence $E:=\ell^2(\cC;\cH^{\oplus\infty}\otimes_\bC A)$ satisfies the situation of \cref{notn:outer} by setting $((\cH_{\lambda},N_{\lambda}) \mid \lambda\in\Lambda) = ((\alpha(\pi)\otimes_A\cH^{\oplus\infty},\alpha(\overline{\pi}))\mid \pi\in\Irr\cC)$. 
	Then, as in the proof of \cref{prop:tensorKirch}, we see (i) and (ii) in \cref{cor:simpleCPalg} and thus the simplicity of $\cT_E$. 
	The unital inclusion $A\subset \cT_E$ is a $\KK^\cC$-equivalence by \cref{prop:KKCequiv}. 
	
	Using $\bU$ from the proof of \cref{prop:regularcorr}, we observe for $\pi,\varpi\in\Irr\cC$, 
	\begin{align}\label{eq:thm:simpleTPcorr2}
		&
		\{ x\in \Hom_{\Cor(A)}(\alpha(\pi),\alpha(\varpi))
		\mid 
		\bU_\varpi(x\otimes_A1_E)=(1_E\otimes_Ax)\bU_\pi \} 
		= \delta_{\pi,\varpi}\bC1 
	\end{align}
	as follows. 
	Indeed, when $\pi=\varpi$ and if \cref{eq:thm:simpleTPcorr2} is false, the left hand side, which is a C*-algebra, must contain non-zero elements $x,y$ with $x^*y=0$. 
	Then, as subbimodules of $E\otimes_A \alpha(\pi)$ we have 
	\begin{align*}
		&
		\cspan x\alpha(\pi) = \bra E,E\ket\otimes_A\cspan x\alpha(\pi)
		= 
		\cspan T_E^*(1_E\otimes x)(E\otimes_A \alpha(\pi))
		\\
		={}& 
		\cspan T_E^*(1_E\otimes x)\bU_\pi (\alpha(\pi)\otimes_A E)
		= 
		\cspan T_E^*\bU_\pi(x\alpha(\pi)\otimes_A E)
		\\
		\supset{}&
		\cspan
		T_{ \alpha(\pi) \otimes_A \cH^{\oplus\infty} \otimes_\bC \alpha(\overline{\pi}) }^*
		\bU_\pi \bigl(
		x\alpha(\pi) \otimes_A \alpha(\mathbbm{1}_\cC) \otimes_A \cH^{\oplus\infty} \otimes_\bC \alpha(\mathbbm{1}_\cC) 
		\bigr) 
		\\
		\overset{\bigstar}{=}{}& 
		\cspan
		T_{ \alpha(\pi) \otimes_A \cH^{\oplus\infty} \otimes_\bC \alpha(\overline{\pi}) }^*
		x\alpha(\pi) \otimes_A \cH^{\oplus\infty} \otimes_\bC \fu^*_{\overline{\pi},\pi}\alpha(R_\pi)\alpha(\mathbbm{1}_\cC) 
		\\
		={}& 
		\cspan
		T_{ \alpha(\overline{\pi}) }^*
		\fu^*_{\overline{\pi},\pi}\alpha(R_\pi)\alpha(\mathbbm{1}_\cC) 
		=\alpha(\pi) \supset y\alpha(\pi), 
	\end{align*}
	where for the last equality we have used the bijectivity of the map $\alpha(\overline{\pi})\ni \xi\mapsto \alpha(R_\pi)^*\fu_{\overline{\pi},\pi}T_\xi\in\cK(\alpha(\pi),A)$ for a standard solution $(R_\pi,\overline{R}_\pi)$ with the inverse $\cK(\alpha(\pi),A)\ni w\mapsto (w\otimes1)\fu_{\pi,\overline{\pi}}^*\alpha(\overline{R}_\pi)\in \alpha(\overline{\pi})$, 
	and for $\overset{\bigstar}{=}$ we have used $\bU_\pi=1\otimes\|R_\pi\|^{-1}\fu_{\overline{\pi},\pi}^*\alpha(R_\pi)$ on the summand $\alpha(\pi)\otimes_A\alpha(\mathbbm{1}_\cC)\otimes_A\cH^{\oplus\infty}\otimes_\bC\alpha(\mathbbm{1}_\cC)$ by invoking the construction of $\bU_\pi$. 
	This contradicts $\bra x\alpha(\pi), y\alpha(\pi)\ket=0$ and $y\neq0$. Thus \cref{eq:thm:simpleTPcorr2} holds when $\pi= \varpi$. 
	When $\pi\neq\varpi$ and if \cref{eq:thm:simpleTPcorr2} is false, there is a non-zero element $x$ in the left hand side. 
	By \cref{eq:thm:simpleTPcorr2} in the previous case, we see that $x^*x$ and $xx^*$ are non-zero scalars, which implies the bijectivity of $x$. 
	Then as direct summands of $E\otimes_A \alpha(\varpi)$ we have 
	\begin{align*}
		&
		(1_E\otimes x)\bU_\pi 
		\bigl(
		\alpha(\pi) \otimes_A \alpha(\overline{\pi}) \otimes_A \cH^{\oplus\infty} \otimes_\bC \alpha(\pi) 
		\bigr)
		\\
		\supset{}& 
		(1_E\otimes x) 
		\bigl(
		\alpha(\mathbbm{1}_\cC) \otimes_A \cH^{\oplus\infty} \otimes_\bC \alpha(\mathbbm{1}_\cC) \otimes_A \alpha(\pi) 
		\bigr)
		\\
		={}&
		\alpha(\mathbbm{1}_\cC) \otimes_A \cH^{\oplus\infty} \otimes_\bC \alpha(\mathbbm{1}_\cC) \otimes_A \alpha(\varpi) , 
	\end{align*}
	which is orthogonal to the summand 
	\begin{align*}
		&
		\bU_\varpi (x\otimes1_E) 
		\bigl(
		\alpha(\pi) \otimes_A \alpha(\overline{\pi}) \otimes_A \cH^{\oplus\infty} \otimes_\bC \alpha(\pi) 
		\bigr)
		\\={}&
		\bU_\varpi \bigl( \alpha(\varpi) \otimes_A \alpha(\overline{\pi}) \otimes_A \cH^{\oplus\infty} \otimes_\bC \alpha(\pi) \bigr) 
		\\={}&
		\bigoplus_{\varpi\otimes\overline{\pi}\geq \tau\in\Irr\cC} \alpha(\tau) \otimes_A \cH^{\oplus\infty} \otimes_\bC \alpha(\overline{\tau})\otimes_A \alpha(\varpi) , 
	\end{align*}
	since $\mathbbm{1}_\cC \not\leq \varpi\otimes\overline{\pi}$. 
	Therefore we get a contradiction to obtain \cref{eq:thm:simpleTPcorr2} for all $\pi,\varpi\in\Irr\cC$. 
	
	Thanks to \cref{thm:simpleTPcorr}, it follows that the $\cC$-action on $\cT_E$ is irreducible. 
	We show the injectivity. 
	We suppose the existence of $\pi,\varpi\in\Irr\cC$ with $\pi\neq \varpi$ such that $\Hom_{\Cor(\cT_E)}(\alpha^\cT(\pi), \alpha^\cT(\varpi))\neq 0$ and deduce a contradiction. 
	
	We write $\gamma$ for the gauge actions from \cref{notn:gauge} with $M:=\alpha(\pi\oplus\varpi)$. 
	Since $\Hom_{\Cor(\cT_E)}(\alpha^\cT(\pi),\alpha^\cT(\varpi))$ inherits the $\bT$-action $\gamma$ on $\cL(\alpha^\cT(\pi\oplus\varpi),\alpha^\cT(\pi\oplus\varpi))$ by taking the commutant $(\alpha^\cT_{\pi\oplus\varpi}(\cT_E)\otimes1)'$ and looking at the corner, 
	the spectral decomposition as in \cref{notn:gauge} (3) shows that there exists $u\in\Hom_{\Cor(\cT_E)}(\alpha(\pi)\otimes_A\cT_E, \alpha(\varpi)\otimes_A\cT_E)\cong \Hom_{\Cor(\cT_E)}(\alpha^\cT(\pi),\alpha^\cT(\varpi))$ and $m\in \bZ$ such that $0\neq u =z^m\gamma_z(u)$ for all $z\in\bT$. 
	By considering $\fu^*$ and exchanging $\pi$ and $\varpi$, we may assume $m\geq 0$. 
	By \cref{rem:irroutfull} (2), we may assume that $u$ is a unitary by considering $\|u\|^{-1}u$. 
	Using \cref{eq:thm:simpleTPcorr1}, we see that $u|_{\alpha(\pi)\otimes_AA_\infty}\in \cL(\alpha(\pi)\otimes_AA_\infty, \alpha(\varpi)\otimes_AX_{-m})$ is a unitary, 
	but this is impossible if $m\geq 1$ because the domain of $u|_{\alpha(\pi)\otimes_AA_\infty}$ is full while its codomain is not. 
	Thus $u$ is invariant under $\gamma$. 
	
	Thanks to the $*$-isomorphism $\Theta\colon D_0\cong D^{\gamma}$ in the proof of \cref{thm:simpleTPcorr} for $(M,\bv):=(\alpha(\pi)\oplus\alpha(\varpi),\bU)$, we see $u|_{\alpha(\pi)\otimes_AA_\infty}=v\otimes_A1_{A_\infty}$ for a uniquely determined $(A,A)$-bilinear unitary $v\in\cL(\alpha(\pi),\alpha(\varpi))$ 
	satisfying $\bU_\varpi(v\otimes 1_E)=(1_E\otimes v)\bU_\pi$. 
	Since \cref{eq:thm:simpleTPcorr2} forces $v=0$, we get a desired contradiction. 
	Now it follows that $(\alpha^\cT,\fu^\cT)$ is outer when $A$ is unital. 
	
	Finally, if $A$ is not unital, then $A$ is stable by Zhang's dichotomy~\cite{Zhang1992:certainI}, and thus $A\cong \bK\otimes pAp$ for some non-zero projection $p\in A$. We apply the argument above to the unital separable simple purely infinite C*-algebra $pAp$ with the $\cC$-action induced by the imprimitivity bimodule $Ap$ and obtain $(\iota,\bu)\colon pAp\to\cT_E$. 
	Then $A\cong \bK\otimes pAp\xto{\id_\bK\otimes(\iota,\bu)} \bK\otimes\cT_E$ is a non-degenerate injective $*$-homomorphism giving a $\KK^\cC$-equivalence, where the $\cC$-action on $\bK\otimes\cT_E$ is outer by \cref{lem:irrbimodtensor} below. 
\end{proof}

\begin{lem}\label{lem:irrbimodtensor}
	Let $A$, $B$ be C*-algebras, and $E=(E,\phi)$, $E_i=(E_i,\phi_i)$ for $i=1,2$ be non-degenerate $(A,B)$-correspondences. 
	Then for any non-zero simple C*-algebra $D$, we have the following. 
	\begin{enumerate}
		\item
		The $(A\otimes D, B\otimes D)$-correspondence $E\boxtimes D$ is irreducible if and only if $E$ is irreducible. 
		Here, $\boxtimes$ indicates the exterior tensor product. 
		\item
		There is no non-zero $(A\otimes D,B\otimes D)$-bilinear adjointable operator in $\cL(E_1\boxtimes D,E_2\boxtimes D)$ if and only if there is no non-zero $(A,B)$-bilinear adjointable operator in $\cL(E_1,E_2)$. 
	\end{enumerate}
\end{lem}

\begin{proof}
	Since $(\phi(A)'\cap\cL(E))\otimes1_D\subset (\phi(A)\otimes D)'\cap\cL(E\boxtimes D)$, we see that if $E\boxtimes D$ is irreducible, so is $E$. 
	Conversely, the irreducibility of $E$ implies the irreducibility of $E\boxtimes D$ by a slicing argument. 
	Indeed, any $x\in (\phi(A)\otimes D)'\cap\cL(E\boxtimes D)$ satisfies $(\id\otimes f)(x)\in \phi(A)'\cap \cL(E)=\bC$ for all states $f\colon D\to \bC$ and thus $x\in 1_E\otimes (D'\cap \cM(D))=\bC$. 
	%(Here, $D'\cap \cM(D)$ must be trivial because otherwise it contains non-zero positive elements $x,y$ with $xy=0$ giving rise to non-zero closed ideals $\cspan DxD$ and $\cspan DyD$ of $D$ orthogonal to each other, which cannot exist by the simplicity of $D$.) 
	This shows (1). 
	
	As for (2), again the slicing argument shows that the absence of non-zero adjointable bilinear operators in $\cL(E_1,E_2)$ implies the absence of such operators in $\cL(E_1\boxtimes D,E_2\boxtimes D)$, while the converse is easy. 
\end{proof}

\subsection{First application}\label{ssec:outensorO2}
It is observed that a $3$-cocycle twist of a countable group always admits an action on $\cO_2$ \cite[Theorem 2.4]{Izumi2023:Gkernels}. Also, it is certainly known to the experts that every fusion category has an action on $\cO_2$ as a consequence of Ocneanu's compactness argument. 
From \cref{thm:tensorKirch}, we obtain the generalization of these results for arbitrary countable unitary tensor categories. 

\begin{cor}\label{thm:outensorO2}
	Every countable unitary tensor category $\cC$ admits an outer action on $\cO_2$, or equivalently, a fully faithful unitary tensor functor $\cC\to\Cor(\cO_2)$. 
\end{cor}

\begin{proof}
	Since $\cC$ acts on $c_0(\Irr\cC)$ (see \cref{eg:LRtranslation} (3) below, for example), there is an outer $\cC$-action on some Kirchberg algebra in the UCT class. 
	By tensoring $\cO_2$, we get an outer $\cC$-action on a Kirchberg algebra $A$ Morita equivalent to $\cO_2$ with the aid of \cref{lem:irrbimodtensor}. 
	This yields a fully faithful unitary tensor functor $\cC\to\Cor(A)\simeq\Cor(\cO_2)$. 
\end{proof}

Note that for an action $(\alpha,\fu)$ of a unitary tensor category $\cC$ on a unital simple C*-algebra $A$, associated with each non-zero object $\pi\in\cC$ we have a unital inclusion $A\to \cK(\alpha(\pi))$ via the left action $\alpha_\pi$. 
Via $(\alpha,\fu)$, the standard solution of the conjugate equation for $\pi$ induces a solution of the conjugate equation for $\alpha(\pi)$ in $\Cor(A)$ and then yields a conditional expectation $\cK(\alpha(\pi))\to A$ of finite Watatani index $(\dim_\cC\pi)^2$ (see \cite[Lemma 1.26]{Kajiwara-Watatani2000:Jones} and \cite[Theorems 4.4, 4.13]{Kajiwara-Pinzari-Watatani2004:Jones}).

\begin{rem}\label{rem:prf:main1}
	Let $\cC$ be a countable unitary tensor category $\cC$ and $(\alpha,\fu)$ be an outer $\cC$-action on $\cO_2$. 
	\begin{enumerate}
		\item
		If $\cC$ is finitely generated in the sense that there exists $\pi\in\cC$ such that every $\varpi\in\Irr\cC$ satisfies $\varpi\leq \pi^{\otimes n}$ for some $n\in\bZ_{\geq 0}$, 
		there is a unital inclusion $\iota\colon \cO_2\hookrightarrow \cO_2$ with a conditional expectation $E$ of finite Watatani index~\cite{Watatani:index} such that the idempotent-complete full tensor subcategory of $\Cor(\cO_2)$ generated by the $(\cO_2,\cO_2)$-correspondence given by $\cO_2$ with the inner product coming from $E$ is unitarily monoidally equivalent to $\cC$. 
		See \cite[Remark 8.2]{Popa-Vaes2015:representation} for a similar statement on $L(\mathbb{F}_\infty)$ and also \cite{Hartglass-HernandezPalomares2020:realizations} for its variant for exact monotracial simple C*-algebras (which might depend on $\cC$). 
		Indeed, for $\pi\in\cC$ as above, we let $\iota:=\alpha_{\pi\oplus\mathbbm{1}}\colon \cO_2\hookrightarrow\cK(\alpha(\pi\oplus\mathbbm{1}_\cC))$, which has a conditional expectation of finite Watatani index, and see that $\cK(\alpha(\pi\oplus\mathbbm{1}_\cC))$ is Morita equivalent and thus $*$-isomorphic to $\cO_2$. 
		In particular, this is the case when $\cC$ has property (T) thanks to \cite[Proposition 5.4]{Popa-Vaes2015:representation}. 
		See \cite{Arano:cpxssDrinfeld,Arano-Vaes:C*-tensor,Vaes-Valvekens:property} for examples of $\cC$ with property (T) beyond groups. 
		\item
		It is not hard to see that $(\cO_2\otimes\cO_2,\alpha\otimes\id_{\cO_2},\fu\otimes1_{\cO_2})$ is $\KK^\cC$-equivalent to $0$ by using \cite[Theorem 3.19]{Arano-Kitamura-Kubota:tensorKK}. Thus we see \cref{main:1} in the remaining case of $A=0$ with the aid of \cref{lem:irrbimodtensor}. 
	\end{enumerate}
\end{rem}

Note that a discrete quantum group always acts trivially on arbitrary C*-algebras, though the corresponding action of the tensor category is typically not injective or irreducible. 
However, we can produce outer actions from such trivial actions using \cref{thm:tensorKirch}. 

\begin{cor}\label{prop:outensorOinfty}
	For any compact quantum group $G$ with countable $\Irr(\Rep G)$, there is an outer $\Rep G$-action on any Kirchberg algebra. 
\end{cor}

\begin{proof}
	Thanks to the $\cO_\infty$-absorption theorem~\cite{Kirchberg-Phillips:embedding} combined with \cref{lem:irrbimodtensor}, it suffices to show the statement for $\cO_\infty$. 
	By \cref{thm:tensorKirch} applied to the $\Rep G$-action on $\bC$ given by the unitary fibre functor $\Rep G\to \Hilb=\Cor(\bC)$, 
	we obtain an outer $\Rep G$-action on a unital Kirchberg algebra $A$ with the unital embedding $\bC\to A$ being a $\KK$-equivalence. Thus $A\cong\cO_\infty$. 
\end{proof}

In the proof of \cref{thm:outensorO2}, we used the following $\cC$-action, which will be the fundamental source of further constructions of $\cC$-actions also in \cref{sec:fusionmod} and \cref{sec:TLJ}. 
\begin{rem}\label{eg:LRtranslation}
	Consider a countable unitary tensor category $\cC$ and a countable set $I$. 
	Let $\cC^{\rev}$ denote the unitary tensor category $\cC$ whose monoidal structure is reversed. 
	We write $\Hilb_I$ for the semisimple C*-category of finite-dimensional Hilbert $c_0(I)$-modules as objects and adjointable operators as morphisms. For each $i\in I$, we write $e_i\in c_0(I)$ for the minimal projection corresponding to $i$ and $\bC_i$ for the one-dimensional Hilbert $c_0(I)$-module on which $e_i$ acts by $\id$. 
	\begin{enumerate}
		\item
		A structure of a \emph{unitary right $\cC$-module} on $\Hilb_I$ is a bifunctor $\olessthan\colon \Hilb_I\times \cC\to \Hilb_I$ preserving the $*$-linear structures accompanied by a natural family of unitaries $\ass(\tau,\pi,\varpi)\colon (\tau\olessthan\pi)\olessthan\varpi \xto{\sim} \tau\olessthan(\pi\otimes\varpi)$ for $\tau\in\Hilb_I$, $\pi,\varpi\in\cC$ and a natural family of unitaries $\tau\olessthan\mathbbm{1}_\cC\xto{\sim}\tau$ for $\tau\in\Hilb_I$ satisfying the commutativity of the pentagon and triangle diagrams analogous to \cite[(7.2), (7.4)]{Etingof-Gelaki-Nikshych-Ostrik:book}. 
		
		For a given structure of a unitary right $\cC$-module on $\Hilb_I$, we obtain a $\cC$-action $(\alpha,\fu)$ on $c_0(I)$. See \cite[Propositions 2.3, 2.4]{Antoun-Voigt2020:bicolimits}, \cite[Proposition 2.5, Remark 2.2]{Arano-Kitamura-Kubota:tensorKK}, for example. 
		Here, $\alpha$ is given by, for $i,j\in I$ and $\pi,\varpi\in\cC$, 
		\[ e_{i}\alpha(\pi)e_{j} := \Hom_{\Hilb_I}(\bC_j,\bC_i\olessthan\pi), \]
		which has the inner product defined by $\bra \xi,\eta\ket:=\xi^*\eta\in \Hom_{\Hilb_I}(\bC_j,\bC_j)\cong \bC e_j$. 
		\item
		Moreover, for another countable set $J$ and a unitary right $\cC$-module structure $\olessthan'\colon \Hilb_J\times \cC\to \Hilb_J$, 
		a \emph{unitarily $\cC$-linear} functor $F\colon \Hilb_I\to \Hilb_J$ in the sense of a $*$-functor accompanied by a natural family of unitaries $F(\tau\olessthan\pi)\xto{\sim}F(\tau)\olessthan'\pi$ for $\tau\in\Hilb_I$, $\pi\in\cC$ satisfying the commutativity of appropriate diagrams analogous to \cite[(7.6), (7.7)]{Etingof-Gelaki-Nikshych-Ostrik:book}, induces a $\cC$-$(c_0(I),c_0(J))$-correspondence $(E,\bv)$. 
		See \cite[Proposition 2.16, Remark 2.2]{Arano-Kitamura-Kubota:tensorKK}, for example. 
		Here, $E$ is given by, for $i\in I$ and $j\in J$, 
		\[ e_{i}Ee_{j} := F(\bC_i)e_j. \]
		This $\cC$-$(c_0(I),c_0(J))$-correspondence $(E,\bv)$ is always proper thanks to $\dim_\bC F(\bC_i)<\infty$ by the definition of $\Hilb_J$. 
		\item
		In particular, $\cC=\Hilb_{\Irr\cC}$ has a unitary right $\cC^{\rev}\boxtimes\cC$-module structure such that $\tau\olessthan(\pi\boxtimes\varpi)=\pi\otimes\tau\otimes\varpi$ for $\pi\in\Irr\cC^{\rev}$, $\varpi\in\Irr\cC$, and $\tau\in\cC$, which induces a $\cC^{\rev}\boxtimes\cC$-action on $c_0(\Irr\cC)$. 
		This restricts to a $\cC$-action and a $\cC^{\rev}$-action on $c_0(\Irr\cC)$ denoted by $(\rho,\mathfrak{r})$ and $(\lambda,\mathfrak{l})$, respectively. 
		When we regard $c_0(\Irr\cC)$ as a $\cC$-C*-algebra with $(\rho,\mathfrak{r})$, then $\lambda(\pi)$ has a canonical structure of a $\cC$-$(c_0(\Irr\cC),c_0(\Irr\cC))$-correspondence. 
	\end{enumerate}
\end{rem}

\section{K-theoretic realization of modules over fusion rings}\label{sec:fusionmod}
For a semisimple C*-category $\cC$, we write $\bZ[\cC]$ for its Grothendieck group. 
It is a free $\bZ$-module with the basis $( [\pi] \mid \pi\in \Irr\cC )$. 
When $\cC$ is a unitary tensor category, its monoidal structure makes $\bZ[\cC]$ a unital ring by $[\pi] [\varpi] := \sum_{\tau\in\Irr\cC}\dim_\bC\Hom_\cC(\tau,\pi\otimes\varpi) [\tau]$ for $\pi,\varpi\in\Irr\cC$, which is called the fusion ring of $\cC$. 
For a unital commutative ring $R$, we let $R[\cC]:=R\otimes_\bZ\bZ[\cC]$. 
When we consider modules over $R[\cC]$, we will not take into account any compatibility with the involution of $R[\cC]$ induced by the conjugation of $\cC$. 

\begin{rem}\label{rem:fusionmodule}
	Let $\cC$ be a countable unitary tensor category and $(A,\alpha,\fu)$, $(B,\beta,\fv)$ be separable $\cC$-C*-algebras. 
	\begin{enumerate}
		\item
		It is easy to see that for $\pi\in\cC$ the proper $(A,A)$-correspondence $\alpha(\pi)$ induces an element in $\KK(A,A)$, and that this assignment equips $\K_i(A)$ for $i=0,1$ with right $\bZ[\cC]$-module structures by Kasparov products. 
		\item
		For separable $\cC$-C*-algebras $A$ and $B$, any $\mathbf{x}\in\KK^\cC(A,B)$ induces homomorphisms of right $\bZ[\cC]$-modules $\K_i(A)\to \K_i(B)$ for $i=0,1$, which are isomorphic if $\mathbf{x}$ is a $\KK$-equivalence. 
		Indeed, this claim can be easily checked when $\mathbf{x}$ is given by some proper $\cC$-$(A,B)$-correspondence $(E,\phi,\bv)$ with the aid of $[E,\phi,\bv]=[\overline{\phi(A)E},\phi,\bv]\in\KK^\cC(A,B)$, and the general case of $\mathbf{x}$ can be reduced to this case 
		because it holds by \cite[Proof of Theorem 4.24]{Arano-Kitamura-Kubota:tensorKK} that $[\mathbf{x}]=[\mathsf{K}_A]^{-1}\mathbin{\widehat{\otimes}}[\psi,\bu]\mathbin{\widehat{\otimes}}[\mathsf{K}_B]$ for some $\cC$-$*$-homomorphism $(\psi,\bu)$ and some separable proper $\cC$-C*-correspondences $\mathsf{K}_A$, $\mathsf{K}_B$ that give $\KK^\cC$-equivalences. 
		%Or alternatively, we can directly check that for any $\cC$-Kasparov $(A,B)$-module $(E,\bv,F)$ with a non-degenerate left $B$-action (cf.~\cite[Lemma 3.16]{Arano-Kitamura-Kubota:tensorKK}), $\bv_\pi^{*}(F\mathbin{\widehat{\otimes}}_B 1_{\beta(\pi)})\bv_\pi \in \cL(\alpha(\pi)\mathbin{\widehat{\otimes}}_A E)$ is a $F$-connection for any $\pi\in\cC$, which implies $[E\mathbin{\widehat{\otimes}}_A\beta(\pi), F\mathbin{\widehat{\otimes}}_B1_{\beta(\pi)}]\in\KK(A,B)$ equals the Kasparov product of $\alpha(\pi)$ and $(E,F)$. 
	\end{enumerate}
\end{rem}

\begin{rem}\label{rem:LRtranslation}
	We note several compatibilities of modules over fusion rings, which are routine to check. 
	\begin{enumerate}
		\item
		In \cref{eg:LRtranslation} (1), 
		the right $\bZ[\cC]$-module structures on $\bZ[\Hilb_I]\cong \bZ^{\oplus I}$ induced by $\olessthan$ and on $\K_0(c_0(I))\cong \bZ^{\oplus I}$ induced by $\alpha$ as in \cref{rem:fusionmodule} (1) coincide. 
		\item
		In \cref{eg:LRtranslation} (2), 
		the right $\bZ[\cC]$-linear maps $\bZ[\Hilb_I]\to\bZ[\Hilb_J]$ induced by $F$ and $\K_0(c_0(I))\to \K_0(c_0(J))$ induced by $E$ as in \cref{rem:fusionmodule} (2) coincide. 
		\item
		In particular, for the $\cC$-C*-algebra $(c_0(\Irr\cC),\rho,\mathfrak{r})$, we have $\K_0(c_0(\Irr\cC))\cong \bZ[\cC]$ as a right $\bZ[\cC]$-module, and 
		the map $\K_0(c_0(\Irr\cC))\to\K_0(c_0(\Irr\cC))$ induced by $\lambda_\pi$ equals the left multiplication by $[\pi]$ on $\bZ[\cC]$ for all $\pi\in\cC$. 
	\end{enumerate}
\end{rem}

In general, $\bZ[\cC]$-module structures on the $\K$-groups need not lift to a $\cC$-action on a given C*-algebra (see also \cref{eg:liftconstgrp} (2) below). 
In this section, we explore sufficient conditions for such liftability in the purely infinite case. 
In particular, we investigate the following constant $d_\cC$. 
\begin{df}\label{def:liftconst}
	For a countable unitary tensor category $\cC$, we define $d_\cC\in \bZ_{\geq 1}$ to be the smallest natural number $d\in\bZ_{\geq 1}$ with the following property if exists: 
	\begin{itemize}
		\item
		For any countable right $\bZ[\cC]$-modules $M_0$ and $M_1$ on which $d$ acts invertibly, 
		there is a Kirchberg algebra $A$ in the UCT class and a $\cC$-action on $A$ such that $\K_i(A)\cong M_i$ as $\bZ[\cC]$-modules for $i=0,1$. 
	\end{itemize}
\end{df}
Thanks to the Kirchberg--Phillips theorem and \cref{thm:tensorKirch}, if $d_\cC$ exists, for any Kirchberg algebra $A$ in the UCT class with $\bM_{d_\cC^{\infty}}\otimes A\cong A$, 
arbitrary structures of right $\bZ[\cC]$-modules on $\K_0(A)$ and $\K_1(A)$ can be realized by some outer $\cC$-action on $A$. 
Hereafter, we promise $\bM_{1^{\infty}}:=\bC$. 

\subsection{Geometric resolution}\label{ssec:geomrsl}
First, we show the liftability under a homological condition using $\cC$-equivariant geometric resolution. 
Using the following iteration of mapping cones, we build a $\cC$-C*-algebra from a sequence of $\cC$-C*-correspondences implementing a given free resolution of finite length. 

\begin{prop}\label{prop:inductiveMC}
	Let $\cC$ be a countable unitary tensor category and $n\in\bZ_{\geq 1}$. Consider separable $\cC$-C*-algebras $(A_k,\alpha_k,\fu_k)$ with $\K_1(A_k)=0$ for $k=0,1,\cdots,n$ and separable proper $\cC$-$(A_{k-1},A_k)$-correspondences $(E_k,\phi_k,\bv_k)$ for $k=1,\cdots,n$. 
	If the sequence 
	\begin{align*}
		&
		0\to \K_0(A_0)\xto{(E_1,\phi_1)} \K_0(A_1) \to \cdots \to \K_0(A_{n-1})\xto{(E_n,\phi_n)}\K_0(A_n)
	\end{align*}
	is exact, then there is a separable $\cC$-C*-algebra $(D,\delta,\fw)$ with $\K_0(D)\cong \Coker(\K_0(A_{n-1})\xto{(E_n,\phi_n)} \K_0(A_n))$ and $\K_1(D)=0$ as $\bZ[\cC]$-modules. 
	If moreover $A_k$ is nuclear and in the UCT class for all $k$, we can arrange $D$ to be nuclear and in the UCT class. 
\end{prop}

\begin{proof}
	We define the Hilbert $A_0$-module $\wt{E}_{0}:=A_0$ and the Hilbert $A_{k}$-module $\wt{E}_{k}:=(\wt{E}_{k-1}\otimes_{\phi_{k}}E_{k})\oplus A_{k}$ inductively on $k=1,\cdots,n$, all of which are separable and full. 
	We define $B_k:=\cK(\wt{E}_{k})$ for $k=0,1,\cdots,n$ with the $\cC$-action induced by the imprimitivity $(B_k,A_k)$-bimodule $\wt{E}_{k}$ and 
	see that the $\cC$-$(B_{k-1},B_{k})$-correspondence 
	\[ \bigl( \wt{E}_{k-1}\otimes_{A_{k-1}} (E_k,\phi_k,\bv_k)\otimes_{A_{k}} \cK(\wt{E}_{k},A_{k}) \bigr) \oplus (\cK(\wt{E}_{k},A_{k}),0,0) \]
	is of the form $(B_k,\psi_k,\bw_k)$ for some $\cC$-$*$-homomorphism $(\psi_k,\bw_k)\colon B_{k-1}\to B_{k}$ for $k=1,\cdots,n$ and that 
	the following diagram of $\bZ[\cC]$-linear maps is commutative, 
	\begin{align*}\begin{aligned}
			\xymatrix{
				\K_0(B_{k-1}) \ar[d]_-{\wt{E}_{k-1}} \ar[r]^-{\psi_k} & \K_0(B_{k}) \ar[d]^-{\wt{E}_{k}} \\
				\K_0(A_{k-1}) \ar[r]^-{(E_k,\phi_k)} & \K_0(A_{k})
			}
		\end{aligned}.\end{align*}
	Since $\wt{E}_{k}$ is an imprimitivity $(B_k,A_k)$-bimodule, we see that the vertical maps above are isomorphic, that $\K_1(B_k)=0$, and that $B_k$ is nuclear and in the UCT class if so is $A_k$. 
	Thus by replacing $A_k$ and $(E_k,\phi_k,\bv_k)$ by $B_k$ and $(B_k,\psi_k,\bw_k)$, we may assume that all $(E_k,\phi_k,\bv_k)$ are $\cC$-$*$-homomorphisms. 
	
	We define the C*-algebras $D_{k}\subset \bigoplus_{j=0}^{k} C_0(0,1]^{\otimes j}\otimes A_j$ for $k=0,1,\cdots,n$ by 
	\begin{align*}
		D_k :={}& 
		\Biggl\{ (a_j)_{j=0}^{k} \in \bigoplus_{j=0}^{k} C_0(0,1]^{\otimes j}\otimes A_j
		\,\Bigg|\, 
		\begin{array}{l}
			\phi_{j} \phi_{j-1}\cdots \phi_{j-i+1}(a_{j-i}(t_{i+1},\cdots,t_j)) 
			\\
			= a_{j}(t_1,t_2,\cdots,t_{i-1},1,t_{i+1},\cdots,t_j), 
			\\
			\forall 1\leq i\leq j\leq k, \forall t_1,t_2,\cdots,t_k\in (0,1]
		\end{array}
		\Biggr\}. 
	\end{align*}
	Note that $D_0=A_0$. 
	If all $A_k$ are nuclear and in the UCT class, then so are $D_k$. 
	Inductively on $k=0,1,\cdots,n$, we can equip $D_k$ with a canonical $\cC$-action with the aid of \cite[Subsection A.2]{Arano-Kitamura-Kubota:tensorKK}, 
	because $D_k$ is the fibred sum of $D_{k-1}$ and $C_0(0,1]^{\otimes k}\otimes A_k$ with respect to the $\cC$-$*$-homomorphism 
	\begin{align*}
		\bigoplus_{j=0}^{k-1} C_0(0,1]^{\otimes j}\otimes A_j 
		\supset D_{k-1} &\to 
		C_0((0,1]^{k}\setminus (0,1)^{k})\otimes A_k ,
		\\
		(a_j)_{j=0}^{k-1} &\mapsto 
		[(t_j)_{j=1}^{k}\mapsto 
		\phi_{k} \phi_{k-1}\cdots \phi_{k-i_t+1}(a_{k-i_t}(t_{i_t+1},\cdots,t_k)) ]
	\end{align*}
	where $i_t:= \max\{ i \mid 0\leq i\leq k, t_i=1 \}$ for $t=(t_j)_{j=1}^{k}\in (0,1]^{k}\setminus (0,1)^{k}$, 
	and the restriction 
	\begin{align*}
		C_0(0,1]^{\otimes k}\otimes A_k\twoheadrightarrow 
		C_0((0,1]^{k}\setminus (0,1)^{k})\otimes A_k , 
	\end{align*}
	which is $\cC$-equivariant and surjective. 
	
	We calculate the $\K$-groups of $D_n$. 
	When we let $D_{-1}:=0$, we have the following $\cC$-equivariant short exact sequence for $k=0,1,\cdots,n$, 
	\begin{align*}
		&
		0\to C_0(0,1)^{\otimes k}\otimes A_k \xto{\iota_k} D_k \xto{q_k} D_{k-1}\to 0, 
	\end{align*}
	where $\iota_k$ is the canonical inclusion and $q_k(a_j)_{j=0}^{k}:=(a_j)_{j=0}^{k-1}$ is the projection, 
	both of which are canonically $\cC$-equivariant by construction. 
	Also, for $k=1,\cdots,n$, we write $\Cone(q_k)\subset D_k\oplus (C_0[0,1)\otimes D_{k-1})$ for the mapping cone of $q_k$ with the inclusion denoted by $\mathrm{in}_k\colon C_0(0,1)\otimes D_{k-1}\to \Cone(q_k)$, which is canonically $\cC$-equivariant. Then we have the following diagram that commutes up to (non-equivariant) homotopy, 
	\begin{align*}\begin{aligned}
		\xymatrix@C=3em{
			C_0(0,1)^{\otimes k}\otimes A_{k-1} \ar[r]^-{\id\otimes\iota_{k-1}}\ar[d]_-{\id^{\otimes k}\otimes \phi_{k}} & 
			C_0(0,1)\otimes D_{k-1} \ar[d]^-{\mathrm{in}_k}
			\\
			C_0(0,1)^{\otimes k}\otimes A_k \ar[r]^-{\jmath_k}&
			\Cone(q_k) 
		}
	\end{aligned},\end{align*}
	where $\jmath_k\colon C_0(0,1)^{\otimes k}\otimes A_{k}\to (C_0[0,1)\otimes 0)\oplus\Ker q_k \subset \Cone(q_k)$ is the inclusion, which is canonically $\cC$-equivariant and induces isomorphisms on $\K_0$ and $\K_1$-groups as $\Coker\jmath_k\cong\Cone(\id_{D_{k-1}})$. 
	%Indeed, the homotopy is given by, for $a\in C_0(0,1)^{\otimes k}\otimes A_{k-1}$ and $s\in [0,1]$, 
	%\begin{align*}
	%	h_s(a):={}&
	%	\bigl[ [0,1]\ni t\mapsto \bigl( \chi_{[1,2]}(t+s)\phi_k(a(t+s-1)), \chi_{[0,1]}(t+s)a(t+s) \bigr) \bigr]
	%	\\\in{}& 
	%	\{ (a_{k},a_{k-1})\in C[0,1]\otimes ( C_0(0,1)^{\otimes k-1}\otimes (A_{k}\oplus A_{k-1}) ) \mid \phi_k(a_{k-1}(0)) = a_{k}(1) , a_{k}(0)=0 , a_{k-1}(1)=0 \} 
	%	\\={}& \Cone\bigl( \{ (a_j)_{j=0}^{k}\in D_k \mid a_0=\cdots=a_{k-2}=0 \} \xto{q_k} \{ (a_j)_{j=0}^{k-1}\in D_{k-1} \mid a_0=\cdots=a_{k-2}=0 \} \bigr) 
	%	\\\subset{}& \Cone(q_k) . 
	%\end{align*}
	This induces the following commutative diagram of $\bZ[\cC]$-linear maps, 
	\begin{align*}\begin{aligned}
			\xymatrix{
				\K_{2k} (A_{k-1}) \ar[r]^-{\iota_{k-1}} \ar[d]_-{\phi_{k}} & 
				\K_{k+1}(D_{k-1}) \ar[dl]_-{\partial_{k-1}} \ar[d]^-{\mathrm{in}_k}
				\\
				\K_{2k} (A_k) \ar[r]^-{\jmath_k}_-{\sim}&
				\K_k(\Cone(q_k))
			}
		\end{aligned},\end{align*}
	where $\partial_{k-1}$ fits into the following six term exact sequence of $\bZ[\cC]$-modules by the construction of the boundary map (see \cite[Theorem 7.5.14]{Murphy:book}, for example), 
	\begin{align*}\begin{aligned}
		\xymatrix{
			\K_{2k}(A_k)\ar[r]^-{\iota_k}& 
			\K_k(D_{k})\ar[r]& 
			\K_k(D_{k-1})\ar[d] \\
			\K_{k+1}(D_{k-1})\ar[u]^-{\partial_{k-1}}&
			\K_{k+1}(D_k)\ar[l]^-{q_k}& 
			0\ar[l]
		}
	\end{aligned}.\end{align*}
	
	We know $\K_1(D_0)=0$ and $\K_0(D_0)=\K_0(A_0)$ by $D_0=A_0$ and $\iota_0=\id_{A_0}$ induces bijections on $\K$-groups. Using the Bott periodicity, inductively on $1\leq k\leq n$ we see the following; 
	\begin{itemize}
		\item
		If $\K_k(D_{k-1})=0$ and $\K_{2k-2}(A_{k-1}) \xto{\iota_{k-1}} \K_{k-1}(D_{k-1})$ is surjective, then we see 
		that the sequence 
		$\K_{2k}(A_{k-1}) \xto{\phi_{k}} \K_{2k}(A_{k}) \xto{\iota_{k}} \K_{k}(D_{k})\to 0$ is exact and that $\iota_k$ induces an isomorphism $\Coker\partial_{k-1}\xto{\sim}\K_k(D_k)$. 
		\item
		If the sequence $\K_{2k-2}(A_{k-2}) \xto{\phi_{k-1}} \K_{2k-2}(A_{k-1}) \xto{\iota_{k-1}} \K_{k-1}(D_{k-1})\to 0$ is exact, then we see 
		the injectivity of $\partial_{k-1}$, which implies $\K_{k+1}(D_k)=0$. 
		Here, we let $A_{-1}=0$ and $\phi_{-1}=0$. 
	\end{itemize}
	Thus $D:=C_0(0,1)^{\otimes n}\otimes D_n$ has the desired properties. 
\end{proof}

We will often use the following absorption argument. 
\begin{rem}\label{rem:absorption}
	Let $R$ be a unital ring and $M$ be a countably generated projective right $R$-module. Then, we have $M\oplus R^{\oplus\infty}\cong R^{\oplus\infty}$ as $R$-modules. 
	Indeed, there is an $R$-linear surjection $p\colon R^{\oplus\infty}\to M$, and we see $R^{\oplus\infty}\cong M\oplus\Ker p$ by the projectivity of $M$, which gives an $R$-linear isomorphism 
	\begin{align*}
		&
		R^{\oplus\infty} \cong (R^{\oplus\infty})^{\oplus\infty} \cong (M\oplus \Ker p)^{\oplus\infty} 
		\cong M\oplus (M\oplus \Ker p)^{\oplus\infty} 
		\cong M \oplus R^{\oplus\infty} . 
	\end{align*}
\end{rem}

We recall that for a unital ring $R$, the \emph{projective dimension} of a right $R$-module is the minimal length, possibly $\infty$, of its projective resolutions. 
The \emph{right global dimension} of $R$ is the supremum of projective dimensions over all right $R$-modules. 
We recall the following lemma from homological algebra. 
\begin{lem}\label{lem:freersl}
	Let $R$ be a countable unital commutative ring and $M$ be a countable right $R$-module of projective dimension no greater than $n\in\bZ_{\geq 1}$. 
	Then, $M$ has a free resolution of length $n$ of the form 
	\begin{align*}
		0\to R^{\oplus\infty} \to \cdots \to R^{\oplus\infty} \to R^{\oplus\infty} \to M\to 0. 
	\end{align*}
\end{lem}

\begin{proof}
	We put $M_{-1}:=M$ and, inductively on $k\geq 0$, take an $R$-linear surjection $p_k\colon R^{\oplus\infty}\to M_{k-1}$, whose kernel $M_{k}:=\Ker p_k$ is countable again. 
	We get the following exact sequence of $R$-modules 
	\begin{align*}
		0\to M_n\to 
		R^{\oplus\infty}\xto{p_{n-1}} \cdots \xto{p_2} R^{\oplus\infty}\xto{p_1} R^{\oplus\infty}\xto{p_0} M\to 0, 
	\end{align*}
	where $M_n$ is projective by a variant of Schanuel's lemma \cite[7.1.2]{McConnell-Robson:book} (the proof is the same as the commutative case). 
	With the aid of \cref{rem:absorption}, by taking the direct sum with $\id\colon R^{\oplus\infty}\to R^{\oplus\infty}$ at $M_n\to R^{\oplus\infty}$ in the sequence above we obtain the desired free resolution. 
\end{proof}

Note that for a unital ring $R$ that is torsion-free as an additive group, we can always consider the localization $R[d^{-1}]$ for $d\in\bZ_{\geq 1}\subset \cZ R$. 
For an abelian group $M$, a right $R$-module structure on $M$ extends to a right $R[d^{-1}]$-module structure on $M$ if and only if the multiplication by $d$ is an invertible map $M\to M$, and if so such an $R[d^{-1}]$-module structure is unique. 

\begin{thm}\label{thm:homologyKirch}
	Let $\cC$ be a countable unitary tensor category and $d\in\bZ_{\geq 1}$. 
	Then for any countable right $\bZ[d^{-1}][\cC]$-modules $M_0$ and $M_1$ of finite projective dimensions, 
	there is an outer $\cC$-action on a Kirchberg algebra $A$ in the UCT class such that $M_i\cong \K_i(A)$ as right $\bZ[\cC]$-modules for $i=0,1$. 
\end{thm}

For $i,j\in\bZ_{\geq 1}$, we let $\bC_{i,j}$ be the $(c_0(\bZ_{\geq 1}),c_0(\bZ_{\geq 1}))$-correspondence with $\dim_\bC\bC_{i,j}=1$ such that $e_i \bC_{i,j} e_j=\bC_{i,j}$ for the minimal projections $e_i,e_j\in c_0(\bZ_{\geq 1})$. Also, for $0\neq k\in\bZ[d^{-1}]$, we fix a projection $p(k)\in \bM_{d^\infty}\otimes\cO_{\infty}$ such that $[p]=k\in\K_0(\bM_{d^\infty}\otimes\cO_{\infty})\cong \bZ[d^{-1}]$ and set $p(0):=0\in \bM_{d^\infty}\otimes\cO_{\infty}$. 

\begin{proof}
	It suffices to show that for any countable right $\bZ[d^{-1}][\cC]$-module $M$ of finite projective dimension, there is a $\cC$-action on some separable nuclear C*-algebra $D$ in the UCT class such that $\K_0(D)\cong M$ and $\K_1(D)\cong 0$. 
	Indeed, when we write $D_i$ for the C*-algebra $D$ given by this claim with $M=M_i$ for each $i=0,1$, 
	we can apply \cref{thm:tensorKirch} to $D_0\oplus (C_0(\bR)\otimes D_1)$ and obtain the desired $\cC$-C*-algebra. 
	
	By \cref{lem:freersl}, we have a free resolution of a given countable right $\bZ[d^{-1}][\cC]$-module $M$ of the form 
	\begin{align*}
		0\to \bZ[d^{-1}][\cC]^{\oplus\infty}\xto{F_n} \bZ[d^{-1}][\cC]^{\oplus\infty} \xto{F_{n-1}} \cdots \xto{F_{1}} \bZ[d^{-1}][\cC]^{\oplus\infty} \xto{F_0} M\to 0. 
	\end{align*}
	For $k=0,\cdots,n-1$, we regard each $(i,j)$-th component of $F_{n-k}$ as the left multiplication by $F_{n-k}(i,j)\in\bZ[d^{-1}][\cC]$ so that $F_{n-k}(a_j)_j = (\sum_j F_{n-k}(i,j)a_j)_i$ for all $(a_j)_j\in \bZ[d^{-1}][\cC]^{\oplus\infty}$ and express $F_{n-k}(i,j)=\sum_{\pi\in\Irr\cC}F_{n-k}(i,j)_\pi [\pi]$ using $F_{n-k}(i,j)_\pi\in\bZ[d^{-1}]$. 
	We write $A_k:=c_0(\Irr\cC)\otimes c_0(\bZ_{\geq 1})\otimes(\bM_{d^\infty}\otimes\cO_\infty)^{\otimes k+1}$ with the $\cC$-action $(\rho\otimes\id,\mathfrak{r}\otimes1)$, where $(\rho,\mathfrak{r})$ is from \cref{eg:LRtranslation} (3). 
	Using the proper $\cC$-C*-correspondences $\lambda(\pi)$ from \cref{eg:LRtranslation} (3), we consider the $\cC$-$(A_k,A_{k+1})$-correspondence 
	defined by the direct sum $E_{k}:=\bigoplus_{i,j=1}^{\infty}\bigoplus_{\pi\in\Irr\cC}E_{k}(i,j,\pi)$ of 
	\begin{align*}
		&
		E_{k}(i,j,\pi):=
		\lambda(\pi) \boxtimes \bC_{j,i}\boxtimes ( p(F_{n-k}(i,j)_{\pi}) \bM_{d^\infty}\otimes\cO_\infty) \boxtimes (\bM_{d^\infty}\otimes\cO_\infty)^{\otimes k+1} 
	\end{align*}
	where the tensorial components $c_0(\Irr\cC)$, $c_0(\bZ_{\geq 1})$, and $(\bM_{d^\infty}\otimes\cO_{\infty})^{\otimes k+1}$ of $A_k$ act on each $E_{k}(i,j,\pi)$ from the left by $\lambda_\pi\otimes1$, $1_{\lambda(\pi)}\otimes e_j(-)e_j\otimes 1$, and $1\otimes \id_{\bM_{d^\infty}\otimes \cO_{\infty}}^{\otimes k+1}$, respectively. 
	
	Here, we claim that $E_{k}$ is proper. 
	To see this, it suffices to show that $e_\varpi\otimes e_j\otimes 1\in A_k$ acts on $E_k$ as an element in $\cK(E_k)$ for all $j\in\bZ_{\geq 1}$ and $\varpi\in\Irr\cC$. 
	There are at most finitely many pairs $(i,\pi)\in\bZ_{\geq 1}\times\Irr\cC$ such that $F_{n-k}(i,j)_\pi\neq 0$ as $(F_{n-k}(i,j))_i=(\sum_{\pi\in\Irr\cC}F_{n-k}(i,j)_\pi[\pi])_i$ is contained in the algebraic direct sum $\bZ[d^{-1}][\cC]^{\oplus\infty} = \bigoplus_{\pi\in\Irr\cC}\bZ[d^{-1}][\pi]^{\oplus\infty}$. It follows from $p(0)=0$ that $E_k(i,j,\pi)\neq 0$ for at most finitely many $(i,\pi)$. 
	Since $e_\varpi\otimes e_j\otimes 1$ acts on each component $E_k(i,j,\pi)$ properly by the properness of $\lambda(\pi)$, $\bC_{j,i}$, and $( p(F_{n-k}(i,j)_{\pi}) \bM_{d^\infty}\otimes\cO_\infty) \boxtimes (\bM_{d^\infty}\otimes\cO_\infty)^{\otimes k+1}$, we see the claim. 
	
	We have $\bZ[\cC]$-linear isomorphisms $\K_1(A_l)=0$ and $\K_0(A_l)\cong \bZ[d^{-1}][\cC]^{\oplus\infty}$ for $l=k,k+1$. Under these isomorphisms, the $\bZ[\cC]$-linear map $\K_0(A_k)\to \K_0(A_{k+1})$ induced by the proper $\cC$-$(A_{k},A_{k+1})$-correspondence $E_k$ coincides with $F_{n-k}\colon \bZ[d^{-1}][\cC]^{\oplus\infty}\to \bZ[d^{-1}][\cC]^{\oplus\infty}$ by construction. In particular, we have a $\bZ[\cC]$-linear isomorphism $\Coker(\K_0(A_{n-1})\to \K_0(A_n))\cong M$. Here, note that any $\bZ[\cC]$-linear map between right $\bZ[d^{-1}][\cC]$-modules is automatically $\bZ[d^{-1}][\cC]$-linear. Thus, \cref{prop:inductiveMC} applied to the exact sequence 
	\begin{align*}
		&
		0\to \K_0(A_0)\xto{E_0} \K_0(A_1) \to \cdots \to \K_0(A_{n-1})\xto{E_{n-1}}\K_0(A_n) 
	\end{align*}
	yields the desired $\cC$-C*-algebra $A$ with the aid of \cref{thm:tensorKirch}. 
\end{proof}

\begin{eg}\label{eg:repCQG}
	When $\bZ[\cC]\cong\bZ[X_1,\cdots,X_n]$ for some $n\in\bZ_{\geq 0}$, its right global dimension is finite. 
	Then the liftability constant $d_\cC$ exists and equals $1$ by \cref{thm:homologyKirch} with $d=1$. 
	For example, this happens for $\cC=\Rep G$ with the following compact quantum groups $G$. 
	\begin{itemize}
		\item
		When $G$ is the $q$-deformation $K_q$ of a connected simply connected compact group $K$ for $0<q\leq 1$ 
		or the free orthogonal quantum group $O^+_N$ for $N\geq 2$. 
		See \cite{Bichon-deRijdt-Vaes:ergodic}. 
		\item
		When $G$ is the $q$-deformation $\operatorname{SO}_q(3)$ of $\operatorname{SO}(3)$ for $0<q\leq 1$ or the quantum permutation group $S^+_N$ for $N\geq 4$. 
		See \cite{deRijdt-vanderVennet:moneq} and also \cite{Soltan:SOq3M2,Banica-Bichon2009:quantumaut4}. 
	\end{itemize}
\end{eg}

\begin{prop}\label{prop:irrindex}
	Let $A$ be a unital Kirchberg algebra and $d\in [4,\infty)$. Then, there is a unital $*$-homomorphism $\iota\colon A\to A$ with $\iota(A)'\cap A=\bC1_A$ whose Watatani index is $d$. 
\end{prop}

Note that by \cite[Corollary 1.4.3]{Watatani:index}, for a unital inclusion of C*-algebras $B\subset A$ with $B'\cap A=\bC1_A$, if there is a conditional expectation $E\colon A\to B$ of finite Watatani index, then $E$ is the unique conditional expectation. In this case, we define the Watatani index of the inclusion $B\subset A$ by the Watatani index of $E$, which must be a scalar. 

\begin{proof}
	Thanks to the $\cO_\infty$-absorption theorem~\cite{Kirchberg-Phillips:embedding}, %combined with \cref{lem:irrbimodtensor}, 
	it suffices to show the statement for $\cO_\infty$ and take the tensor product with $\id_A$. 
	We take $q\in (0,1]$ such that $q+q^{-1}=\sqrt{d}$. 
	
	Since we have the ring isomorphism $\bZ[X]\cong\bZ[\Rep \operatorname{SU}_q(2)]$ sending $X$ to the fundamental representation $\pi_1\in\Irr \operatorname{SU}_q(2)$ with $\dim_{\Rep \operatorname{SU}_q(2)}\pi_1=q+q^{-1}$, there is an outer $\Rep \operatorname{SU}_q(2)$-action $(\alpha,\fu)$ on $\cO_\infty$ such that $[\pi_1]\in \bZ[\Rep \operatorname{SU}_q(2)]$ acts on $\K_0(\cO_\infty)$ by $\id$ as in \cref{eg:repCQG} (1). 
	Thanks to the Kirchberg--Phillips theorem, $\cK(\alpha(\pi_1))\cong \cO_\infty$ as the places of their units in the $\K_0$-groups coincide. 
	The composition $\iota\colon \cO_\infty\to \cK(\alpha(\pi_1))\cong \cO_\infty$ satisfies $\iota(\cO_\infty)'\cap\cO_\infty=\bC1$ by the irreducibility of $\alpha(\pi_1)$. 
	The Watatani index of $\iota$ is $(\dim_{\Rep \operatorname{SU}_q(2)}\pi_1)^2=d$ as desired. 
\end{proof}

\subsection{Generic regularity}\label{ssec:genreg}

As geometrically expected, an affine variety over $\bC$ contains some open subvariety that is smooth, or equivalently its coordinate ring has finite global dimension after some localization. 
In a similar spirit, one might hope that fusion rings also satisfy generic finiteness of global dimension in good situations. 
The observation below confirms this naive expectation for fusion categories. 

For a fusion category $\cC$, we have the unital ring homomorphism 
\begin{align*}
	\bZ[\cC] \ni [\tau] &\mapsto (\dim\Hom_\cC(\pi,\tau\otimes\varpi))_{\pi,\varpi\in \Irr\cC} \in \bM_{\Irr\cC},  
\end{align*}
whose range lies in $\Irr\cC\times\Irr\cC$-matrices with integer coefficients. 
Since this homomorphism is injective by extracting the column at $\mathbbm{1}_\cC\in\Irr\cC$, 
we often identify $\bZ[\cC]$ as a unital subring of $\bM_{\Irr\cC}$ and regard $[\tau]$ for each $\tau\in\cC$ as a matrix with non-negative integer coefficients. 
Then it holds $[\overline{\tau}]=[\tau]^*$. 

\begin{prop}\label{thm:localfusion}
	For a unitary fusion category $\cC$, let $Z_\cC:=\sum_{\pi\in\Irr\cC}[\pi][\overline{\pi}]\in\bZ[\cC]\subset\bM_{\Irr\cC}$ and $d:=|\det Z_\cC|\in\bZ_{\geq 0}$. 
	Then $d\geq 1$ and for any Noetherian unital commutative ring $R$ of global dimension $0$ or $1$ such that $d$ is invertible in $R$, the following hold. 
	\begin{enumerate}
		\item
		A finitely generated right $R[\cC]$-module $M$ is projective if it is projective as an $R$-module. 
		\item
		The right global dimension of $R[\cC]$ is $0$ or $1$. 
	\end{enumerate}
\end{prop}

%When $M$ is not finitely generated, (1) is not valid even when $(R,d)=(\bZ,1)$ as Baer's result says that the torsion-free $\bZ$-module $\prod_{k=1}^{\infty}\bZ$ is not projective. 
Note that the unitary structure and the choice of coefficient field $\bC$ for $\cC$ will not be essential in the proof below as long as $\cC$ is rigid and semisimple. 
The proof is inspired by Maschke's theorem, combined with the construction similar to \cref{prop:regularcorr}. 

\begin{proof}
	We see that 
	$[Z_\cC]=\sum_{\pi\in\Irr\cC}[\pi] [\pi]^*\geq [\mathbbm{1}_\cC][\mathbbm{1}_\cC]^*=1$ inside $\bM_{\Irr\cC}$. 
	Thus $d\geq 1$. 
	Note that for $F:=\det(X1_{\bM_{\Irr\cC}} - Z_\cC)\in\bZ[X]$ and the polynomial $f\in\bZ[X]$ such that $F=F(0)-Xf$, it holds $d=|F(0)|$ and $F(0) - [Z_\cC]f([Z_\cC]) = F([Z_\cC])=0$ in $\bZ[\cC]$. So does for their images in $R[\cC]$. 
	
	To show (2), it suffices to show that any finitely generated right $R[\cC]$-module has projective dimension at most $1$ by \cite[7.1.8]{McConnell-Robson:book} (the proof is the same as the commutative case). 
	(2) follows once we show (1) because any finitely generated right $R[\cC]$-module is a quotient of $R[\cC]^{\oplus m}$ for some $m\in\bZ_{\geq 0}$, whose kernel must be a projective $R$-module. 
	Here, we used the fact that an $R$-submodule of a free $R$-module is projective, which follows from the assumption on the global dimension of $R$ and Schanuel's lemma. 
	
	For a finitely generated right $R[\cC]$-module $M$ that is projective as an $R$-module, we have an $R[\cC]$-linear surjection $q\colon R[\cC]^{\oplus n}\to M$ for some $n\in\bZ_{\geq 0}$ with an $R$-linear section $\iota$. 
	Then the $R$-linear map 
	\begin{align*}
		\jmath\colon M \ni x &\mapsto 
		\sum_{\pi\in\Irr\cC}\iota (x[\pi])[\overline{\pi}] \frac{f([Z_\cC])}{F(0)} 
		\in R[\cC]^{\oplus n} 
	\end{align*}
	is $R[\cC]$-linear because for all $\tau\in\Irr\cC$ and $x\in M$, 
	\begin{align*}
		&
		\sum_{\pi\in\Irr\cC}\iota(x[\tau][\pi])[\overline{\pi}] 
		=
		\sum_{\pi,\varpi\in\Irr\cC}\dim\Hom(\varpi,\tau\otimes\pi) \iota(x[\varpi])[\overline{\pi}] 
		\\={}&
		\sum_{\pi,\varpi\in\Irr\cC}\dim\Hom(\overline{\pi},\overline{\varpi}\otimes\tau) \iota (x[\varpi]) [\overline{\pi}] 
		=
		\sum_{\varpi\in\Irr\cC} \iota (x[\varpi])[\overline{\varpi}][\tau] , 
	\end{align*}
	and $[Z_\cC]\in \cZ R[\cC]$ by a similar reasoning with $\iota$ replaced by $\id_{R[\cC]}$, which implies $f([Z_\cC])\in\cZ R[\cC]$. 
	Using $q\iota=\id_M$, we have for $x\in M$, 
	\begin{align*}
		q\jmath(x) 
		= 
		\sum_{\pi\in\Irr\cC}q \iota (x[\pi])[\overline{\pi}] \frac{f([Z_\cC])}{F(0)} 
		= 
		\sum_{\pi\in\Irr\cC}x[\pi] [\overline{\pi}] \frac{f([Z_\cC])}{F(0)} 
		= 
		x\frac{[Z_\cC]f([Z_\cC])}{F(0)} 
		= 
		x. 
	\end{align*}
	Therefore $M$ is a direct summand of $R[\cC]^{\oplus n}$ also as an $R[\cC]$-module. This shows (1). 
\end{proof}

\begin{cor}\label{cor:localfusion}
	For any unitary fusion category $\cC$, the constant $d_\cC\in\bZ_{\geq 1}$ exists and satisfies $d_\cC\leq |\det Z_\cC|$. 
\end{cor}

\begin{proof}
	It follows from \cref{thm:homologyKirch,thm:localfusion} that $d=|\det Z_\cC|$ satisfies the property in \cref{def:liftconst}. 
\end{proof}

%For a discrete group $\Gamma$ and a $\bT$-valued $3$-cocycle $\omega$ on $\Gamma$, we have a unitary tensor category $\Hilb_{\Gamma,\omega}$ of finite-dimensional $\Gamma$-graded Hilbert spaces whose associators are given by $\omega(f,g,h)\colon (\bC_f\otimes\bC_g)\otimes\bC_h \to \bC_f\otimes(\bC_g\otimes\bC_h)$ for $f,g,h\in\Gamma$, where $\bC_g$ denotes the one-dimensional Hilbert space with the single grade $g$. The resulting unitary tensor category only depends on $[\omega]\in\Ho^3(\Gamma;\bT)$ up to unitary monoidal equivalence. 

\begin{eg}\label{eg:liftconstgrp}
	Let $\cC=\Hilb_{\Gamma,\omega}$ for a finite group $\Gamma$ and $[\omega]\in\Ho^3(\Gamma;\bT)$. 
	\begin{enumerate}
		\item
		We know $d_\cC\leq r$ for the radical $r$ of $\#\Gamma\in\bZ_{\geq 1}$ by the proof of \cref{cor:localfusion} as $|\det Z_\cC|=\#\Gamma^{\#\Gamma}$ and $\bZ[r^{-1}]=\bZ[(\#\Gamma)^{-\#\Gamma}]$. 
		Also, it directly follows from \cite[Theorem 3.5]{Katsura2008:construction} for stable Kirchberg algebras in the UCT class that $d_\cC=1$ when every Sylow subgroup of $\Gamma$ is cyclic and $[\omega]$ is trivial. 
		\item
		On the other hand, \cite[Theorem 3.4]{Izumi2023:Gkernels} is a source of lower bounds for $d_\cC$ depending on $[\omega]$. 
		Indeed, partly using its technique, it is shown in \cite[Proposition 5.3.4]{GironPacheco:thesis} that $\cC$ never acts on $\cO_\infty\otimes\bK$ unless $[\omega]\in\Ho^3(\Gamma;\bT)$ is trivial. In this case, $M_0=\bZ$ and $M_1=0$ with trivial $\bZ[G]$-module structures cannot lift to any $\cC$-action. 
		Thus, $d_\cC\neq 1$ whenever $[\omega]$ is non-trivial. 
		We will see in \cref{eg:liftconstgrp2} that sometimes $d_\cC$ can be as big as $r$. 
	\end{enumerate}
\end{eg}

\section{Examples from the Temperley--Lieb--Jones category}\label{sec:TLJ}
In this section, we focus on the unitary fusion category of $\operatorname{SU}(2)$ at level $k\in\bZ_{\geq 0}$, which is denoted by $\cC(\mathfrak{sl}_2,k)$ or $\cC_k$ here. 
It moreover has the structure of a unitary modular tensor category and in particular a \emph{unitary braiding}, that is, a natural unitary isomorphism $\br(\pi,\varpi)\colon \pi\otimes\varpi\xto{\sim} \varpi\otimes\pi$, where $\pi,\varpi\in\cC_k$, such that the following diagrams are commutative for all $\tau,\pi,\varpi\in\cC_k$: 
\begin{align*}
	&
	\begin{aligned}
		\xymatrix@C=7em@R=2em{
			(\tau\otimes \pi)\otimes \varpi \ar[r]^-{ \br(\tau,\pi)\otimes\id_\varpi } \ar[d]_-{\ass(\tau,\pi,\varpi) } & 
			(\pi\otimes\tau)\otimes \varpi 
			\ar[d]^-{ \ass(\pi,\tau,\varpi) } 
			\\
			\tau\otimes (\pi\otimes\varpi) \ar[d]_-{ \br(\tau,\pi\otimes\varpi) } & 
			\pi\otimes (\tau\otimes\varpi) \ar[d]^-{ \id_\pi \otimes\br(\tau,\varpi) }
			\\
			(\pi\otimes\varpi) \otimes \tau \ar[r]^-{\ass(\pi,\varpi,\tau)}&
			\pi\otimes(\varpi\otimes\tau) 
		}
	\end{aligned}, 
\end{align*}
and 
\begin{align*}
	&
	\begin{aligned}
		\xymatrix@C=7em@R=2em{
			\tau\otimes (\pi\otimes \varpi) \ar[r]^-{ \id_\tau\otimes\br(\pi,\varpi) } \ar[d]_-{\ass(\tau,\pi,\varpi)^* } & 
			\tau\otimes (\varpi\otimes\pi) \ar[d]^-{ \ass(\tau,\varpi,\pi)^* } 
			\\
			(\tau\otimes \pi)\otimes\varpi \ar[d]_-{ \br(\tau\otimes\pi,\varpi) } & 
			(\tau\otimes\varpi)\otimes\pi \ar[d]^-{ \br(\tau,\varpi)\otimes\id_\pi }
			\\
			\varpi\otimes(\tau \otimes \pi) \ar[r]^-{\ass(\varpi,\tau,\pi)^*}&
			(\varpi\otimes\tau)\otimes \pi 
		}
	\end{aligned}. 
\end{align*}
See \cite{Frohlich-Kerler:book,Wenzl:Cstar-tensor,Xu1998:standard,Yamagami2004:categorical,Edie-Michell-Morrison2017:field} for details. 
We can enumerate $\Irr\cC_k=\{ \pi_i \mid 0\leq i\leq k \}$ and it holds $\pi_i\otimes\pi_j\cong \bigoplus_{h=|i-j|}^{\min \{ i+j, 2k-i-j \}}\pi_h$ for all $i,j\in\{0,1,\cdots,k\}$. 
In particular, 
$\pi_1\otimes \pi_i\cong \pi_i\otimes \pi_1\cong \pi_{i-1}\oplus \pi_{i+1}$, 
where $\pi_{-1}=\pi_{k+1}:=0$. 
We also have $\pi_0=\mathbbm{1}_{\cC_k}$ and $\overline{\pi_i}=\pi_i$. 
The categorical dimension of $\pi_i$ is $\dim_{\cC_k}\pi_i = (\sin \frac{\pi}{k+2})^{-1} (\sin \frac{(i+1)\pi}{k+2})$. 

We can twist the associator ${\ass}(\tau,\pi,\varpi)\colon (\tau\otimes\pi)\otimes\varpi\xto{\sim} \tau\otimes(\pi\otimes\varpi)$ of $\cC_k$ by $\ass^{\tw}(\pi_{i},\pi_{j},\pi_{h}) := (-1)^{ijh}{\ass}(\pi_{i},\pi_{j},\pi_{h})$ for $i,j,h = 0,1,\cdots,k$ to obtain another unitary fusion category denoted by $\cC_k^{\tw}$. See \cite{Frohlich-Kerler:book,Kazhdan-Wenzl:reconst,Edie-Michell-Morrison2017:field}. Also, we can twist the unitary braiding $\br(\pi,\varpi)\colon \pi\otimes\varpi\to
\varpi\otimes\pi$ on $\cC_{k}$ to obtain a unitary natural isomorphism $\br^{\tw}\colon\pi\otimes\varpi\to\varpi\otimes\pi$ such that $\br^{\tw}(\pi_i,\pi_j):=\sqrt{-1}^{i^2j^2}\br(\pi_i,\pi_j)$ for $i,j = 0,1,\cdots,k$, which is unique by considering irreducible decompositions. Then, we can check the commutativity of the following diagrams for all $i,j,h=0,1,\cdots,k$, 
\begin{align*}
	&
	\begin{aligned}
		\xymatrix@C=8em@R=2em{
			(\pi_i\otimes \pi_j)\otimes \pi_h \ar[r]^-{ \sqrt{-1}^{i^2j^2} \br(\pi_i,\pi_j)\otimes\id_{\pi_h} } \ar[d]_-{(-1)^{ijh}\ass(\pi_i,\pi_j,\pi_h) } & 
			(\pi_j\otimes\pi_i)\otimes \pi_h 
			\ar[d]^-{ (-1)^{ijh}\ass(\pi_j,\pi_i,\pi_h) } 
			\\
			\pi_i\otimes (\pi_j\otimes\pi_h) \ar[d]_-{ \br^{\tw}(\pi_i,\pi_j\otimes\pi_h) } & 
			\pi_j\otimes (\pi_i\otimes\pi_h) \ar[d]^-{ \id_{\pi_j} \otimes\sqrt{-1}^{i^2h^2}\br(\pi_i,\pi_h) }
			\\
			(\pi_j\otimes\pi_h) \otimes \pi_i \ar[r]^-{(-1)^{ijh}\ass(\pi_j,\pi_h,\pi_i)}&
			\pi_j\otimes(\pi_h\otimes\pi_i) 
		}
	\end{aligned}, 
\end{align*}
by using $\sqrt{-1}^{i^2(j+h)^2}\br(\pi_i,\pi_j\otimes\pi_h) = \br^{\tw}(\pi_i,\pi_j\otimes\pi_h)$, and similarly 
\begin{align*}
	&
	\begin{aligned}
		\xymatrix@C=8em@R=2em{
			\pi_i\otimes (\pi_j\otimes \pi_h) \ar[r]^-{ \id_{\pi_i}\otimes\sqrt{-1}^{j^2h^2}\br(\pi_j,\pi_h) } \ar[d]_-{(-1)^{ijh}\ass(\pi_i,\pi_j,\pi_h)^* } & 
			\pi_i\otimes (\pi_h\otimes\pi_j) \ar[d]^-{ (-1)^{ijh}\ass(\pi_i,\pi_h,\pi_j)^* } 
			\\
			(\pi_i\otimes \pi_j)\otimes\pi_h \ar[d]_-{ \sqrt{-1}^{(i+j)^2h^2}\br(\pi_i\otimes\pi_j,\pi_h) } & 
			(\pi_i\otimes\pi_h)\otimes\pi_j \ar[d]^-{ \sqrt{-1}^{i^2h^2}\br(\pi_i,\pi_h)\otimes\id_{\pi_j} }
			\\
			\pi_h\otimes(\pi_i \otimes \pi_j) \ar[r]^-{(-1)^{ijh}\ass(\pi_h,\pi_i,\pi_j)^*}&
			(\pi_h\otimes\pi_i)\otimes \pi_j 
		}
	\end{aligned}, 
\end{align*}
which shows that $\br^\tw(\pi,\varpi)\colon \pi\otimes\varpi\to\varpi\otimes\pi$ is a unitary braiding on $\cC_{k}^{\tw}$ by considering irreducible decompositions. 

We write $\cC_{k,0}$ for the unitary tensor full subcategory of $\cC_k$ such that $\Irr\cC_{k,0} = \{ \pi_0, \pi_2,\cdots,\pi_{2\lfloor\frac{k}{2}\rfloor} \}$. 
This is also a unitary tensor full subcategory of $\cC_k^{\tw}$. 
We can check that $[\pi_i]=U_i(2^{-1}[\pi_1])$ recursively on $i=0,1,\cdots,k$, where $U_{n}(X)$ is the Chebyshev polynomial of the second kind. Then note that $U_i(2^{-1}X)\in\bZ[X]$ and that $U_i(2^{-1}X)=\prod_{j=1}^{i}(X-2\cos\frac{j\pi}{i+1})\in\bZ[X]$. 
It follows that the eigenvalues of $[\pi_1]\in \bM_{\Irr\cC_k}=\bM_{\Irr\cC_k^{\tw}}$ must be $2\cos(\frac{j+1}{k+2}\pi)$ for $0\leq j\leq k$ with multiplicity one since $U_{k+1}(2^{-1}X)\in\bZ[X]$ is the minimal monic polynomial of $[\pi_1]$. Also, the eigenvalues of $[\pi_2]=[\pi_1]^2-1\in \bM_{\Irr\cC_{k,0}}$ must be $2\cos(\frac{2(j+1)}{k+2}\pi)+1$ for $0\leq j\leq \lfloor \frac{k}{2}\rfloor$ with multiplicity one.

\begin{eg}\label{eg:Temperley-Lieb-Jones}
	We calculate the upper bounds of $d_{\cC_{k,0}}$, $d_{\cC_k^{\tw}}$, and $d_{\cC_k}$ obtained from \cref{thm:localfusion}. 
	Note that the case $k=1$ reduces to the cases of the trivial group, $\bZ/2\bZ$, and its $3$-cocycle twist, respectively. 
	
	The eigenvalues of 
	$\sum_{i=0}^{\lfloor k/2 \rfloor}[\pi_{2i}][\overline{\pi_{2i}}] = \sum_{i=0}^{\lfloor k/2 \rfloor}U_{2i}(2^{-1}[\pi_1])^2$ 
	in $\bM_{\Irr\cC_{k,0}}$ are 
	\begin{align*}
		&
		\sum_{i=0}^{\lfloor k/2 \rfloor}U_{2i}\Bigl(\cos\frac{(j+1)\pi}{k+2}\Bigr)^2 
		= \sum_{i=0}^{\lfloor k/2 \rfloor}\biggl(\frac{\sin\frac{(j+1)(2i+1)\pi}{k+2}}{\sin\frac{(j+1)\pi}{k+2}}\biggr)^2 
		\\={}& 
		\sum_{i=0}^{\lfloor k/2 \rfloor}\frac{1-\cos\frac{2(j+1)(2i+1)\pi}{k+2}}{1-\cos\frac{2(j+1)\pi}{k+2}}
		= 
		\frac{\lfloor \frac{k}{2} \rfloor+1 + (\frac{k}{2} - \lfloor \frac{k}{2} \rfloor)}{1-\cos\frac{2(j+1)\pi}{k+2}}
		\\={}& 
		\frac{k+2}{(1-e^{\frac{2(j+1)\pi}{k+2}\sqrt{-1}})(1-e^{\frac{2(k-j+1)\pi}{k+2}\sqrt{-1}})} 
	\end{align*}
	for $0\leq j\leq \lfloor \frac{k}{2}\rfloor$. 
	Thus \[|\det Z_{\cC_{k,0}}|=(k+2)^{\lfloor k/2 \rfloor+1}(k+2)^{-1}2^{k-1 -2\lfloor k/2 \rfloor}=(k+2)^{\lfloor k/2 \rfloor}2^{k-1 -2\lfloor k/2 \rfloor}.\]
	
	Similarly, the eigenvalues of 
	$\sum_{i=0}^{k}[\pi_i][\overline{\pi_i}] = \sum_{i=0}^{k}U_i(2^{-1}[\pi_1])^2$ 
	in $\bM_{\Irr\cC_k}$ are 
	\begin{align*}
		&
		\sum_{i=0}^{k}U_i\Bigl(\cos\frac{(j+1)\pi}{k+2}\Bigr)^2 
		= \frac{2(k+2)}{(1-e^{\frac{2(j+1)\pi}{k+2}\sqrt{-1}})(1-e^{\frac{2(k-j+1)\pi}{k+2}\sqrt{-1}})} 
	\end{align*}
	for $0\leq j\leq k$. 
	Thus \[|\det Z_{\cC_k^{\tw}}|=|\det Z_{\cC_k}|=(k+2)^{k+1}2^{k+1}(k+2)^{-2}=2^{k+1}(k+2)^{k-1}.\]
	In summary, we see that $d_{\cC_{k,0}}\leq k+2$ and that $d_{\cC_{k}^{\tw}},d_{\cC_{k}}\leq 2(k+2)$ for $k\geq 2$. 
	Here, note that $\bZ[d^{-1}]=\bZ[d^{\prime -1}]$ for any $d,d'\in\bZ_{\geq 1}$ with the same set of prime divisors. 
\end{eg}

\begin{prop}\label{prop:Temperley-Lieb-Jones}
	For every $k\in\bZ_{\geq 1}$, there exists an outer $\cC_{2k,0}$-action on an arbitrary Kirchberg algebra. 
\end{prop}

Here, if a unitary fusion category $\cC$ admits a unitary fibre functor, then it yields a compact quantum group that is finite-dimensional, which is necessarily of Kac type by \cite[Th{\'e}or{\`e}me 4.10]{Baaj-Skandalis:mltu},
and thus $\dim_\cC\pi\in\bZ$ for every object $\pi\in \cC$. 
Therefore, we see that $\cC_{2k,0}$ does not have a unitary fibre functor for $k\geq 3$ as $\dim_{\cC_{2k,0}}\pi_2=2\cos\frac{2\pi}{2k +2} +1\not\in\bZ$. 

\begin{proof}
	Thanks to the $\cO_\infty$-absorption theorem~\cite{Kirchberg-Phillips:embedding} combined with \cref{lem:irrbimodtensor}, it suffices to construct a $\cC_{2k,0}$-action on $\cO_\infty$. 
	
	Let $\cC_{2k,1}$ be the semisimple full subcategory of $\cC_{2k}$ such that $\Irr\cC_{k,1} = \{ \pi_1, \pi_3,\cdots,\pi_{2k-1} \}$. 
	Then, the unitary right $\cC_{2k}$-module structure on $\cC_{2k}$ as in \cref{eg:LRtranslation} (3) restricts to unitary right $\cC_{2k,0}$-module structures on $\cC_{2k,0}$ and $\cC_{2k,1}$. 
	Also, we have the well-defined functor $\cC_{2k,1}\ni\pi\mapsto \pi_1\otimes\pi\in\cC_{2k,0}$, which is canonically unitarily $\cC_{2k,0}$-linear. 
	Then by \cref{rem:LRtranslation}, we obtain $\cC_{2k,0}$-actions on $C(\Irr\cC_{2k,0})$ and $C(\Irr\cC_{2k,1})$ and a proper $\cC_{2k,0}$-$(C(\Irr\cC_{2k,1}),C(\Irr\cC_{2k,0}))$-correspondence such that the induced map $\K_0(C(\Irr\cC_{2k,1}))\to \K_0(C(\Irr\cC_{2k,0}))$ 
	equals the left multiplication by $[\pi_1]$. 
	Then by the exact sequence 
	\begin{align*}
		0\to \bZ[\cC_{2k,1}] \xto{[\pi_1]\times(-)} \bZ[\cC_{2k,0}] &\to \bZ \to 0 , 
		\\
		[\pi_{2i}] &\mapsto (-1)^i 
	\end{align*}
	we obtain an outer $\cC_{2k,0}$-action on $\cO_{\infty}$ using \cref{prop:inductiveMC} and \cref{thm:tensorKirch}. 
\end{proof}

The full tensor subcategory generated by $\pi_0,\pi_k\in\cC_k$ is unitarily monoidally equivalent to $\Hilb_{\bZ/2\bZ}$ if $k$ is even but not if $k$ is odd. But, $\pi_0,\pi_k$ in $\cC_{k}^{\tw}$ generate the full tensor subcategory unitarily monoidally equivalent to $\Hilb_{\bZ/2\bZ}$ for all $k\in\bZ_{\geq 1}$. See \cite[Lemma 7, Remark 13]{Ostrik2003:module} for details, where we note that a unitary tensor category that is (just $\bC$-linearly) monoidally equivalent to $\Hilb_{\bZ/2\bZ}$ is unitarily monoidally equivalent to it because the inclusion $\bT\subset \bC^\times$ induces an isomorphism $\Ho^3(\bZ/n\bZ;\bT)\to\Ho^3(\bZ/n\bZ;\bC^{\times})$ for all $n\in\bZ_{\geq 1}$. 
Thus besides $C(\Irr\cC_{k}^{\tw})$, we have another source of $\cC_{k}^{\tw}$-actions coming from the Q-system $\varpi_0:=\pi_0\oplus\pi_k$ with the unit $\eta\colon\pi_0\to\varpi_0$ and the multiplication $\mu\colon\varpi_0\otimes\varpi_0\to\varpi_0$ corresponding to those of the group C*-algebra $\bC[\bZ/2\bZ]$. 
This gives a right $\cC_k^{\tw}$-module category of type T by \cite[Remark 14 (ii)]{Ostrik2003:module}. To see its unitary structure, we describe it as follows. 
Consider the idempotent-complete additive C*-category $\cD_{k}$ whose objects are of the form $\varpi_0\otimes\pi$ for $\pi\in \cC_k^{\tw}$ such that $\Hom_{\cD_{k}}(\varpi_0\otimes\pi,\varpi_0\otimes\varpi)$ for $\pi,\varpi\in\cC_{k}^{\tw}$ is defined by the set of morphisms $f\in\Hom_{\cC_{k}^{\tw}}(\varpi_0\otimes\pi,\varpi_0\otimes\varpi)$ such that the following diagram is commutative, 
\begin{align*}\begin{aligned}
	\xymatrix{
		\varpi_0\otimes(\varpi_0\otimes\pi) \ar[r]^{{\ass}^{\tw *}}\ar[d]_-{\id\otimes f}& 
		(\varpi_0\otimes\varpi_0)\otimes\pi \ar[r]^-{\mu\otimes\id}& 
		\varpi_0\otimes\pi \ar[d]^-{f}\\
		\varpi_0\otimes(\varpi_0\otimes\varpi) \ar[r]^-{{\ass}^{\tw *}}&
		(\varpi_0\otimes\varpi_0)\otimes\varpi \ar[r]^-{\mu\otimes\id}&
		\varpi_0\otimes\varpi 
	}
\end{aligned}.\end{align*}
We observe that $\cD_{k}$ is a semisimple C*-category with the unitary right $\cC_k^{\tw}$-module structure induced by the monoidal structure of $\cC_{k}^{\tw}$ such that $\Irr\cD_{k}=\{ \varpi_i \mid i=0,1,\cdots,\lfloor k/2\rfloor \}$, where $\varpi_i:=\varpi_0\otimes\pi_i$, which can be checked to be isomorphic to $\varpi_0\otimes\pi_{k-i}$ in $\cD_{k}$. 
Note that the forgetful functor $\mathop{\mathrm{Forg}}_{\varpi_0}\colon\cD_{k}\to \cC_{k}^{\tw}$ is unitarily $\cC_{k}^{\tw}$-linear. 
Then by \cref{eg:LRtranslation}, the unital C*-subalgebra 
\begin{align*}
	C(\Irr\cD_{k})=\bigoplus_{i=0}^{\lfloor k/2\rfloor}\bC (e_{\pi_i} + e_{\pi_{k-i}})\subset \bigoplus_{i=0}^{k}\bC e_{\pi_i} = C(\Irr\cC_{k}^{\tw}) 
\end{align*}
admits the $\cC_{k}^{\tw}$-action such that this inclusion becomes a $\cC_{k}^{\tw}$-$*$-homomorphism canonically. 
Also, the unitary functor $\mathop{\mathrm{Ind}}_{\varpi_0}\colon \cC_{k}^{\tw}\ni \pi\mapsto \varpi_0\otimes\pi\in\cD_{k}$ is unitarily $\cC_{k}^{\tw}$-linear, which induces a proper $\cC_{k}^{\tw}$-$(C(\Irr\cC_{k}^{\tw}),C(\Irr\cD_{k}))$-correspondence. 

\begin{thm}\label{thm:Temperley-Lieb-Jones}
	For $k\in\bZ_{\geq 1}$ such that $p:=k+2$ is an odd prime, $d_{\cC_{k,0}}=d_{\cC_{k}^{\tw}}=1$. 
\end{thm}

We will see that $d_{\cC_k}=2$ in \cref{eg:Temperley-Lieb-Jones2}. 
We separate several number-theoretic materials required in the proof of \cref{thm:Temperley-Lieb-Jones} into \cref{sec:numbertheory}. 

\begin{proof}
	From $\dim_\bQ \bQ(2\cos\frac{2\pi}{p})=\frac{p-1}{2}=\frac{\deg U_{k+1}}{2}$, 
	we see that $U_{k+1}(2^{-1}X)\in\bZ[X^2]$ is the monic minimal polynomial (with respect to the variable $X^2$) of $[\pi_1]^2$ and that $\bZ[\cC_{k,0}]\cong \bZ[X^2]/U_{k+2}(2^{-1}X) \bZ[X^2] \cong \bZ[2\cos\frac{2\pi}{p}]$, whose global dimension is $1$ (see \cref{notn:algnumber,eg:cyclotomic}). 
	Thus the case $\cC=\cC_{k,0}$ follows from \cref{thm:homologyKirch} with $d=1$. 
	From now on, we consider the case $\cC=\cC_{k}^{\tw}$. 
	
	As in the proof of \cref{thm:homologyKirch}, it suffices to show the case $M_1=0$. 
	We have $\bZ[\cC_{k}^{\tw}]\cong \bZ[2\cos\frac{2\pi}{p}][\bZ/2\bZ]$ by $\pi_k\not \in\Irr\cC^{\tw}_{k}\setminus\Irr\cC_{k,0}$, the commutativity of the fusion rule, and $\pi_k\otimes\pi_k\cong \pi_0$. 
	Thus we can apply \cref{thm:Z/2module} to $R:=\bZ[2\cos\frac{2\pi}{p}]$ thanks to \cref{eg:cyclotomic} and obtain a resolution of the right $\bZ[\cC_{k}^{\tw}]$-module $M:=M_0$ only using countably many direct sums of $\bZ[\cC_{k}^{\tw}]$ and $R_+\cong\bZ[\cD_{k}]$. 
	We can lift the $\bZ[\cC_{k}^{\tw}]$-linear maps $F_1,F_2,F_3$ in \cref{eq:thm:Z/2module0} to the sequence of proper $\cC_{k}^{\tw}$-C*-correspondences by a similar argument as in \cref{thm:homologyKirch} thanks to the following. 
	\begin{itemize}
		\item
		As we have used in \cref{thm:homologyKirch} for $d=1$, any $\bZ[\cC_{k}^{\tw}]$-linear map $\bZ[\cC_{k}^{\tw}]\to \bZ[\cC_{k}^{\tw}]$ is a $\bZ$-linear combination of the maps coming from $\pi_i\otimes(-)\colon \cC_{k}^{\tw}\to \cC_{k}^{\tw}$ over $i$. 
		\item
		For a $\bZ[\cC_{k}^{\tw}]$-linear map $f\colon \bZ[\cD_{k}]\to \bZ[\cC_{k}^{\tw}]$, there are $f_i\in \bZ$ such that $f([\varpi_j])=\sum_{i=0}^{\lfloor k/2\rfloor}f_i [\pi_i]([\pi_0]+[\pi_k]) [\pi_j]$ for all $j$, and thus $f$ is a $\bZ$-linear combination of the maps coming from the compositions $\cD_{k}\xto{\mathop{\mathrm{Forg}}_{\varpi_0}} \cC_{k}^{\tw}\xto{\pi_i\otimes(-)} \cC_{k}^{\tw}$ over $i$. 
		\item
		For a $\bZ[\cC_{k}^{\tw}]$-linear map $f\colon \bZ[\cC_{k}^{\tw}]\to \bZ[\cD_{k}]$, there are $f_i\in \bZ$ such that $f([\pi_j])=\sum_{i=0}^{\lfloor k/2\rfloor}f_i [\varpi_i] [\pi_j]$ for all $j$, and thus $f$ is a $\bZ$-linear combination of the maps coming from the compositions $\cC_{k}^{\tw} \xto{\pi_i\otimes(-)} \cC_{k}^{\tw} \xto{\mathop{\mathrm{Ind}}_{\varpi_0}} \cD_{k}$ over $i$. 
		\item
		For a $\bZ[\cC_{k}^{\tw}]$-linear map $f\colon \bZ[\cD_{k}]\to \bZ[\cD_{k}]$, there are $f_i\in \bZ$ such that $f([\varpi_j])=\sum_{i=0}^{\lfloor k/2\rfloor}f_i [\varpi_i] [\pi_j]$ for all $j$, and thus $f$ is a $\bZ$-linear combination of the maps coming from $(-)\otimes\pi_i\colon \cD_{k} \to \cD_{k}$ over $i$. 
		This endofunctor is unitarily $\cC_{k}^{\tw}$-linear 
		thanks to the unitary braiding $\br^{\tw}$. 
	\end{itemize}
	Now \cref{prop:inductiveMC} and \cref{thm:tensorKirch} complete the proof. 
\end{proof}

Unfortunately, such an argument is no longer applicable when $k+2$ is not a prime. It certainly holds that for some positive integer $d$, the fusion ring $\bZ[d^{-1}][\cC_{k,0}]$ is isomorphic to the direct product of $\bZ[d^{-1}][2\cos\frac{2\pi}{n}]$ for positive divisors $n$ of $k+2$. However, we cannot take $d=1$, for example when $k+2=9$ or $15$.

\begin{eg}[Fibonacci category]\label{eg:Fibonacci}
	Let $\cC$ be the Fibonacci category, or equivalently, $\cC=\cC_{k,0}$ with $k=3$, where we can apply \cref{thm:Temperley-Lieb-Jones}. 
	Then the Cuntz algebra $\cO_{n+1}$ admits a $\cC$-action if and only if there exists $m\in\bZ$ such that $m^2 - (m+1) \in n\bZ$. 
	This is impossible if $n$ is even. Next, we consider the case when $2$ is invertible modulo $n$. 
	It follows from \[(m+2^{-1})^2 -(m+2^{-1})-1 \equiv m^2-4^{-1}5 \mod n \]
	that $\cC$ can act on $\cO_{n+1}$ if and only if $5$ is a quadratic residue of $n$. 
	When $n$ is a prime, this is equivalent to $p\equiv 0,\pm1$ modulo $5$ thanks to quadratic reciprocity. 
	For general $n$, this is the case if and only if $n$ is either of the form $\prod_{i=1}^{h}p_i^{s_i}$ or $5\prod_{i=1}^{h}p_i^{s_i}$ for some $h\in\bZ_{\geq 0}$, $s_i\in\bZ_{\geq 0}$ and primes $p_i$ with $p_i\equiv\pm1$ modulo $5$, with the aid of Chinese remainder theorem and the fact on an odd prime $p$ that $a\in\bZ\setminus p\bZ$ is a quadratic residue modulo $p^r$ for all $r\in\bZ_{\geq 1}$ if it holds for $r=1$. 
	See \cite[Example 3.1]{Izumi:subalgebrasI}, \cite[Example 4.8]{Izumi:inclusions} for constructions of $\cC$-actions on $\cO_{n+1}$ for several $n$ based on a different approach. 
\end{eg}

Finally, note that if $\cC_k$ or $\cC_k^{\tw}$ has an outer action $(\alpha,\fu)$ on a unital Kirchberg algebra $A$, we have a unital inclusion $A\subset\cK(\alpha(\pi_1))$ with $A'\cap\cK(\alpha(\pi_1))\cong\bC$ with Watatani index $(\dim_{\cC_k}\pi_1)^2=(\dim_{\cC_k^{\tw}}\pi_1)^2=4\cos^2\frac{\pi}{k+2}$. 
Moreover, $A\cong\cK(\alpha(\pi_1))$ if $\pi_1$ fixes the class $[1_A]\in\K_0(A)$, which happens when $[1_A]=0\in\K_0(A)$, for example. 

\section{Actions of 3-cocycle twists of cyclic groups on Cuntz algebras}\label{sec:cyclicCuntz}
In this section, we consider actions of $3$-cocycle twists of discrete groups on a Kirchberg algebra $A$ that are given by $*$-automorphisms. 
When $A$ is stable, they are equivalent to considering $\Hilb_{\Gamma,\omega}$-actions on $A$ by Kasparov's stabilization (see \cite[Lemma 2.36 (1)]{Arano-Kitamura-Kubota:tensorKK}). 
When $A$ is unital, they correspond to $\Hilb_{\Gamma,\omega}$-actions on $A$ whose induced $\bZ[\Gamma]$-module structures on $\K_0(A)$ fix the class $[1_A]$ in the sense that $[1_A][g]=[1_A]\in\K_0(A)$ for all $g\in\Gamma$ (see \cite[Lemma 2.36 (2)]{Arano-Kitamura-Kubota:tensorKK} or more generally \cref{lem:corneranomalous}). 

\subsection{Conventions on anomalous actions and \texorpdfstring{$\K_0^{\#}$}{K0\#}}\label{ssec:termK0sharp}
We say that a \emph{right $(\Gamma,\omega)$-action} on a C*-algebra $A$ is a $\Hilb_{\Gamma,\omega}$-action $(\alpha,\fu)$ on $A$ such that $\alpha(g)$ is of the form $(A,\alpha_g)$ for some $\alpha_g\in\Aut(A)$ for all $g\in \Gamma$. 
Equivalently, a right $(\Gamma,\omega)$-action on $A$ is a family $(\alpha_g)_{g\in\Gamma}$ of $*$-automorphisms on $A$ and a family $(\fu_{g,h})_{g,h\in\Gamma}$ of unitaries in $\cU\cM(A)$ satisfying for all $f,g,h\in\Gamma$, 
\begin{align*}
	\Ad \fu_{g,h}\alpha_h\alpha_g &= \alpha_{gh} , 
	\qquad\text{and}\qquad
	\omega(f,g,h) \fu_{fg,h} \alpha_h(\fu_{f,g}) 
	= 
	\fu_{f,gh} \fu_{g,h}. 
\end{align*}
In this section, we shall use the multiplicative notation for $\bT$-valued $3$-cocycles and thus the cohomology group $\Ho^3(\Gamma;\bT)$ (unlike the additive notation of $\Ho^3(\Gamma;\bR/\bZ)$ in \cref{main:4}). 

\begin{rem}\label{rem:rightleftanomalous}
	A \emph{left} $(\Gamma,\omega)$-action on a C*-algebra $A$ consists of a family $(\alpha_g)_{g\in\Gamma}$ of $*$-automorphisms on $A$ and a family $(\fu_{g,h})_{g,h\in\Gamma}$ of unitaries in $\cU\cM(A)$ satisfying for all $f,g,h\in\Gamma$, 
	\begin{align*}
		\alpha_g\alpha_h &= \Ad \fu_{g,h}\alpha_{gh} , 
		\qquad\text{and}\qquad
		\alpha_f(\fu_{g,h}) \fu_{f,gh} 
		= 
		\omega(f,g,h) \fu_{f,g} \fu_{fg,h}. 
	\end{align*}
	In the literature, this is often called an anomalous action of $(\Gamma,\omega)$, and the induced group homomorphism $\Gamma\to \mathop{\mathrm{Out}}(A)$ is called a $\Gamma$-kernel or an outer action of $\Gamma$ (warn that this notion is not the same as our outerness). 
	We shall mainly treat right $(\Gamma,\omega)$-actions in this section. For a right $(\Gamma,\omega)$-action $(\alpha,\fu)$ on $A$, we let 
	\begin{align*}
		\omega'(f,g,h) := \omega(h^{-1},g^{-1},f^{-1}), 
		\qquad \alpha'_g := \alpha_{g^{-1}}, 
		\qquad \fu'_{g,h} := \fu_{h^{-1},g^{-1}}^*, 
	\end{align*}
	to have a well-defined $3$-cocycle $\omega'$ on $\Gamma$ and a left $(\Gamma,\omega')$-action $(\alpha',\fu')$ on $A$. 
	Conversely, a left $(\Gamma,\omega')$-action on $A$ gives a right $(\Gamma,\omega)$-action on $A$ in a similar way. 
	The assignment $\omega\mapsto\omega'$ induces an involutive group automorphism on $\Ho^3(\Gamma;\bT)$. 
\end{rem}

\begin{lem}\label{lem:corneranomalous}
	For a $\Hilb_{\Gamma,\omega}$-C*-algebra $(A,\alpha,\fu)$ and a projection $p\in A$ with $\cspan ApA=A$ such that there is a unital $*$-homomorphism from $\cO_{\infty}$ to $B:=pAp$, 
	we write $(\beta,\fv)$ for the $\Hilb_{\Gamma,\omega}$-action on $B:=pAp$ induced by the imprimitivity bimodule $Ap$. 
	If the $\Gamma$-action on $\K_0(A)$ fixes $[p]$, 
	then the $\Hilb_{\Gamma,\omega}$-action $(\beta,\fv)$ on $B$ is given by a right $(\Gamma,\omega)$-action. 
\end{lem}

\begin{proof}
	We take mutually orthogonal isometries $s_k\in B$ for $k\in\bZ_{\geq 1}$. 
	It suffices to show that for any full proper non-degenerate $(B,B)$-correspondence $(E,\phi)$ that induces an endomorphism on $\K_0(B)$ fixing $[1_B]$, we have $E\cong B$ as a Hilbert $B$-module. 
	
	By using Kasparov's stabilization with the aid of the unitality of $\cK(E)$, we see that there is a projection $q\in\bK \otimes B$ with $q(\bK \otimes B)q\cong \cK(E)$. 
	Since $E$ is full and $B$ is unital, there are $n\in\bZ_{\geq 1}$ and $\xi_1,\cdots,\xi_n\in E$ such that $x:=\sum_{k=1}^{n}\bra \xi_k,\xi_k\ket\geq 0$ is invertible. Then $\xi:=\sum_{k=1}^{n}\phi(s_k) \xi_k x^{-1/2} \in E$ satisfies $\bra\xi,\xi\ket=1_B$ and gives a right $B$-linear adjointable isometry $B\ni b\mapsto \xi b\in E$. 
	Thus, $q$ contains a projection $r$ with $r\sim 1_B$, where $\sim$ denotes Murray--von Neumann equivalence. 
	Since $[1_B]=[q]\in\K_0(B)$ by assumption, if $l$ is large enough we have 
	\begin{align*}
		q\sim (q-r)\oplus 1_B \sim q\oplus 1_B^{\oplus l}\oplus \Bigl(1_B-\sum_{k=1}^{l+1} s_ks_k^*\Bigr) \sim 1_B\oplus 1_B^{\oplus l}\oplus \Bigl(1_B-\sum_{k=1}^{l+1} s_ks_k^*\Bigr) \sim 1_B
	\end{align*}
	and thus $E\cong B$ as a Hilbert $B$-module. 
\end{proof}

\begin{eg}\label{eg:poly-Z}
	When $\Gamma$ is a poly-$\bZ$ group with Hirsch length $h=h(\Gamma)\in\bZ_{\geq 0}$, the right global dimension of $\bZ[\Gamma]$ is at most $h+1$. See \cite[Corollary 7.5.6]{McConnell-Robson:book}. 
	Thus, $d_\cC=1$ for any $[\omega]\in\Ho^3(\Gamma;\bT)$ by \cref{thm:homologyKirch}. 
	In the case of the trivial $\bZ[\Gamma]$-module $\bZ$, for all $[\omega]\in\Ho^3(\Gamma;\bT)$ there is a right $(\Gamma,\omega)$-action on $\cO_\infty$ by \cref{lem:corneranomalous}. 
\end{eg}

\begin{eg}\label{eg:liftconstgrp2}
	Let $\cC=\Hilb_{\bZ/m\bZ,\omega}$ for $m\in\bZ_{\geq 1}$ and a generator $[\omega]\in\Ho^3(\bZ/m\bZ;\bZ)\cong \bZ/m\bZ$. Then $d_\cC=r$ for the radical $r$ of $m$. 
	To see this, we fix $k\in\bZ_{\geq 1}\setminus r\bZ$ and take a prime divisor $p$ of $m$ that does not divide $k$. 
	Thanks to \cite[Theorem 3.4]{Izumi2023:Gkernels} combined with \cref{lem:corneranomalous}, there is no $\cC$-action on $\cO_{p+1}$ with trivial $\bZ[\bZ/m\bZ]$-module structures on $\K_i(\cO_{p+1})$. As $k$ acts invertibly on $\K_0(\cO_{p+1})\cong\bZ/p\bZ$, it follows that $d=k$ does not satisfy the property in \cref{def:liftconst}. 
	Thus, $d\in \bZ_{\geq 1}$ can satisfy the property in \cref{def:liftconst} only if $d\in r\bZ$. 
	Combined with \cref{eg:liftconstgrp} (1), we see $d_\cC=r$. 
\end{eg}

\begin{eg}\label{eg:anomalousPimsner}
	Let $(A,\alpha,\fu)$ be a right $(\Gamma,\omega)$-action and $(E,\phi,\bv)$ be a non-degenerate faithful $\Hilb_{\Gamma,\omega}$-$(A,A)$-correspondence. 
	For $g\in\Gamma$, by regarding $\bv_g\colon E\cong A\otimes_{\phi}E\to E\otimes_{\alpha_g}A$, 
	we consider the surjective isometric $\bC$-linear map $\mu_g\colon E\to E$ determined by $\mu_g(\xi)\otimes_{\alpha_g}1_A = \bv_g(\xi)$ for $\xi\in E$. 
	By definition, we have $\phi(a)\mu_g(\xi)=\mu_g(\phi\alpha_g(a)\xi)$, $\mu_g(\xi)a=\mu_g(\xi\alpha_g(a))$, and $\alpha_g(\bra \mu_g(\xi),\mu_g(\eta)\ket)=\bra\xi,\eta\ket$ for $a\in A$, $\xi,\eta\in E$, and $g\in\Gamma$. 
	It holds that 
	\begin{align*}
		&
		\mu_g\mu_h(\xi) \otimes_{\alpha_g}1_A \otimes_{\alpha_h}1_A
		= 
		(\bv_g\otimes_{\alpha_h}1_A)\bv_h(\xi)
		=
		(1_E\otimes\fu_{g,h}^*)\bv_{gh}(\phi(\fu_{g,h})\xi)
		\\={}&
		(1_E\otimes\fu_{g,h}^*) \bigl( \mu_{gh}(\phi(\fu_{g,h})\xi)\otimes_{\Ad\fu_{g,h}\alpha_h\alpha_g}1_A \bigr)
		\\={}&
		\bigl( \mu_{gh}(\phi(\fu_{g,h})\xi) \alpha_g^{-1}\alpha_h^{-1}(\fu_{g,h}^*) \bigr)
		\otimes_{\alpha_h\alpha_g}1_A 
		=
		\mu_{gh}(\phi(\fu_{g,h})\xi \fu_{g,h}^*) 
		\otimes_{\alpha_h\alpha_g}1_A 
	\end{align*}
	for any $g,h\in\Gamma$ and $\xi\in E$, 
	where we have used $\alpha_g^{-1}\alpha_h^{-1}(\fu_{g,h}^*) = \alpha_g^{-1}\alpha_h^{-1}(\fu_{g,h}^*\fu_{g,h}^*\fu_{g,h}) = \alpha_{gh}^{-1}(\fu_{g,h}^*)$. 
	
	Then if we define $\alpha^\cT_g\in\Aut(\cT_E)$ for $g\in\Gamma$ by $\alpha^\cT_g(a)=\alpha_g(a)$ and $\alpha^\cT_g(\tS_\xi)=\tS_{\mu^{-1}_g \xi}$ for $a\in A$, $\xi\in E$, we can check that $(\alpha^\cT,\fu)$ is a well-defined right $(\Gamma,\omega)$-action on $\cT_E$. The corresponding $\Hilb_{\Gamma,\omega}$-action on $\cT_E$ is isomorphic to the one given by \cref{prop:tensorKirch}. 
	Thus, as in the proof of \cref{thm:tensorKirch}, for any right $(\Gamma,\omega)$-action $(\alpha,\fu)$ on a nuclear separable C*-algebra $A$, there is a Kirchberg algebra $B$ with a right $(\Gamma,\omega)$-action $(\beta,\fv)$ that is injective (and thus outer) as a $\Hilb_{\Gamma,\omega}$-action and a non-degenerate injective $\Hilb_{\Gamma,\omega}$-$*$-homomorphism $(\iota,\bu)\colon A\to B$ giving a $\KK^\cC$-equivalence such that $\iota\alpha_g=\beta_g\iota$ and $\fv_{g,h}=\iota(\fu_{g,h})$ for all $g,h\in\Gamma$. 
	Note that $B$ must be unital if so is $A$. 
\end{eg}

Next, we recall some definitions and facts we need from \cite{Izumi2023:Gkernels,GironPacheco:thesis}. 
For a unital Kirchberg algebra $A$ with $\K_1(A)=0$, we define the group $\K^\#_0(A)$ as the quotient 
\begin{align*}
	\K^\#_0(A) 
	:= 
	{}&
	\frac{
		\bigl\{ u\in \cU(C[0,1]\otimes A) 
		\,\big|\, 
		u(0)=1, 
		u(1) \in \bT 
		\bigr\} 
	}{
		\bigl\{ u\in \cU(C[0,1]\otimes A) 
		\,\big|\, 
		u(0)=u(1)=1 
		\bigr\}_0
	}, 
\end{align*}
where $_0$ at the bottom means the connected component containing the unit. 
Then we have an isomorphism 
\begin{align*}
	(\K_0(A)\oplus\bR)/\bZ([1_A],-1) \ni ([p], s) \mapsto{}& [t\mapsto e^{ 2\pi\sqrt{-1} t (s + p) }] \in \K^\#_0(A) , 
\end{align*}
where $p\in A$ is a projection. 
In particular, $\bT\ni e^{ 2\pi\sqrt{-1} s } \mapsto [t\mapsto e^{ 2\pi\sqrt{-1} t ns }] \in \K_0^{\#}(\cO_{n+1})$, where $n\in\bZ_{\geq 1}$ and $s\in\bR$, is a well-defined group isomorphism. 

For a right $(\Gamma,\omega)$-action $(\alpha,\fu)$ on $A$, we take a norm-continuous path $\wt{\fu}_{g,h}\colon [0,1]\to \cU(A)$ with $\wt{\fu}_{g,h}(0)=1$ and $\wt{\fu}_{g,h}(1)=\fu_{g,h}$ for $g,h\in\Gamma$ and define the $3$-cochain $\wt{\omega}=\wt{\omega}_{\alpha,\fu}\colon \Gamma^{3}\to \K_0^\#(A)$ by 
\begin{align*}
	&
	\wt{\omega}(f,g,h) 
	:= 
	\alpha_h(\wt{\fu}_{f,g}^*) \wt{\fu}_{fg,h}^* \wt{\fu}_{f,gh} \wt{\fu}_{g,h} . 
\end{align*}
By putting $\wt{\fu}_{g,h}^{\prime}:=\wt{\fu}_{h^{-1},g^{-1}}^*$ and using the left $(\Gamma,\omega')$-action $(\alpha',\fu')$ as in \cref{rem:rightleftanomalous}, we see that 
\begin{align*}
	\wt{\omega}'(f,g,h)
	:=
	\wt{\omega}(h^{-1},g^{-1},f^{-1})
	=
	\alpha'_{f}(\wt{\fu}_{g,h}^{\prime}) \wt{\fu}_{f,gh}^{\prime} \wt{\fu}_{fg,h}^{\prime *} \wt{\fu}_{f,g}^{\prime *} , 
\end{align*}
and thus our $\wt{\omega}$ is compatible with \cite[Definition 3.3]{Izumi2023:Gkernels} up to left and right. 
In view of the bijective correspondence $\wt{\omega}\mapsto\wt{\omega}'$, it holds that $\wt{\omega}$ is a well-defined $3$-cocycle, 
whose class $[\wt{\omega}_{\alpha,\fu}]\in \Ho^3(\Gamma;\K_0^\#(A))$ does not depend on the choice of $\wt{\fu}$. 
Let us write $\mathrm{ev}_{1*}$ for the map $\Ho^3(\Gamma;\K_0^\#(A)) \to \Ho^3(\Gamma;\bT)$ induced by $\mathrm{ev}_1\colon \K_0^\#(A)\ni [u]\mapsto u(1)\in\bT$. 
Since $\mathrm{ev}_{1*}([\wt{\omega}])=[\omega]$ by construction, we see that if $\omega\in\mathrm{Z}^3(\Gamma;\bT)$ admits a right $(\Gamma,\omega)$-action, then $[\omega]$ must be contained in the image of $\mathrm{ev}_{1*}$, as observed in \cite{Izumi2023:Gkernels}. 

\begin{eg}\label{eg:Temperley-Lieb-Jones2}
	For an odd prime $p$ and $k:=p-2$, we show $d_{\cC_{k}}=2$ as promised after \cref{thm:Temperley-Lieb-Jones}. 
	We see $d_{\cC_k}\leq 2$ by \cref{thm:homologyKirch} because the right global dimension of $\bZ[2^{-1}][\cC_k]\cong\bZ[2\cos\frac{2\pi}{p},2^{-1}][\bZ/2\bZ]$ is at most $1$ by \cref{thm:localfusion} for $R:=\bZ[2\cos\frac{2\pi}{p},2^{-1}]$ and $\cC:=\Hilb_{\bZ/2\bZ}$. 
	
	To see $d_{\cC_k}\neq 1$, let $M:=\bZ[2\cos\frac{2\pi}{p}]$ equipped with the $\bZ[2\cos\frac{2\pi}{p}][\bZ/2\bZ]$-module structure such that $\bZ/2\bZ$ acts trivially. 
	It suffices to deduce a contradiction by supposing that there is a Kirchberg algebra $A$ with a $\cC_k$-action such that $\K_0(A)\cong M/2M$ and $\K_1(A)\cong 0$. 
	For the non-trivial class $[\omega]\in\Ho^3(\bZ/2\bZ;\bT)\cong\bZ/2\bZ$, it holds that $\Hilb_{\bZ/2\bZ,\omega}\subset \cC_{k}$ by \cite[Lemma 7]{Ostrik2003:module}. 
	Restricting the $\cC_k$-action on $A$ to $\Hilb_{\bZ/2\bZ,\omega}$ and applying \cref{lem:corneranomalous} to some projection $q\in A$ with $0\neq[q]\in\K_0(A)$ such that $N\oplus (\bZ/2\bZ)[q]= \K_0(A)$ for some subgroup $N\subset \K_0(A)$ (we can choose such $q$ and $N$ since $\bZ[2\cos\frac{2\pi}{p}]$ is a finitely generated free abelian group), we obtain a right $(\bZ/2\bZ,\omega)$-action on the unital Kirchberg algebra $B=qAq$. 
	We see the commutativity of the following diagram 
	\begin{align*}
		\xymatrix@C=2em{
			\K_0^\#(B)\ar[r]^-{\sim}\ar[dr]_-{[u]\mapsto u(1)\quad}& (\K_0(B)\oplus \bR)/\bZ([q],-1) 
			\ar[r]^-{\sim} & N\oplus \bR/2\bZ \ar[d]^-{(x,r)\mapsto r}
			\\
			& \bT & \bR/2\bZ \ar[l]_-{r\mapsto e^{2\pi \sqrt{-1} r}}
		}. 
	\end{align*}
	The induced map $\Ho^3(\bZ/2\bZ;\K_0^\#(B))\to \Ho^3(\bZ/2\bZ;\bT)\cong \bZ/2\bZ$ sends $[\wt{\omega}]$ to $[\omega]\neq 0$, while the range of this map must be $0$ as in the proof of \cite[Theorem 3.4]{Izumi2023:Gkernels}, yielding a contradiction as desired. 
	Since any odd natural number acts on $\K_0(B)$ invertibly, this argument also shows that $d\in\bZ_{\geq 1}$ satisfies the condition in \cref{def:liftconst} if and only if $d$ is even. 
\end{eg}

\subsection{Construction of \texorpdfstring{$(\bZ/m\bZ,\omega)$}{(Z/mZ,ω)}-actions}\label{ssec:cyclicCuntz}
In this subsection, we consider the problem on the realizability of right $(\bZ/m\bZ,\omega)$-actions on Cuntz algebras $\cO_{n+1}$. 
We fix $m,n\geq 1$ and put $\zeta_m:=e^{\frac{2\pi}{m}\sqrt{-1}}$. 
For an integer $k\in \bZ$, we write $\bra k\ket_m:=k-\lfloor\frac{k}{m}\rfloor m \in \{ 0,1,\cdots,m-1 \}$ for the remainder of the division by $m$. 
Also, we define the $3$-cocycle $\omega_m^k\in\mathrm{Z}^3(\bZ/m\bZ;\bT)$ by 
\begin{align*}
	\omega_m^k(i,j,h) := \exp\Bigl( 2\pi\sqrt{-1} \Bigl\lfloor\frac{i+j}{m}\Bigr\rfloor \frac{hk}{m} \Bigr) , 
\end{align*}
for $i,j,h=0,1,\cdots,m-1$, and we have a group isomorphism (see \cite[Example 2.6.4]{Etingof-Gelaki-Nikshych-Ostrik:book}, for example) 
\begin{align*}
	\{ z\in\bC \mid z^m=1 \} \ni \zeta_m^k  &\mapsto [\omega_m^k] \in \Ho^3(\bZ/m\bZ;\bT) . 
\end{align*}
Note that for $k\in\bZ/m\bZ$, a pair $(\alpha,u)$ of a $*$-automorphism $\alpha\in\Aut(A)$ and a unitary $u\in\cU\cM(A)$ such that $\alpha^m=\Ad u^*$ and $\zeta_m^{k}\alpha(u)= u$ gives a right $(\bZ/m\bZ,\omega_m^{k})$-action on $A$ by 
$(\alpha^i \mid i=0,1,\cdots,m-1)$ and $(u^{\lfloor\frac{i+j}{m}\rfloor} \mid i,j=0,1,\cdots,m-1)$. 

The $(\bZ/m\bZ,\omega)$-actions in this subsection come from those on the following locally compact space. See \cite{Jones2021:remarks} for several examples and constraints on actions of $3$-cocycle twists of groups on compact spaces. 
\begin{eg}\label{eg:cyclicspace}
	For $n\geq 1$, consider the locally compact space $X_n:= (0,1] \times\bT/{\sim}$, where we define $(r,x)\sim (s,y)$ if and only if $s=r=1$ and $x^n=y^n$. 
	Then $\K_0(C_0(X_n))=\bZ/n\bZ$ and $\K_1(C_0(X_n))=0$ by the six term exact sequence associated with the mapping cone sequence for $C_0(X_n)=\Cone(\phi)$, where $\phi\colon C(\bT)\ni f(x)\mapsto f(x^n)\in C(\bT)$. 
	Thus $C_0(X_n)$ is $\KK$-equivalent to $\cO_{n+1}$. 
	
	It is not hard to see that we have the right $(\bZ/m\bZ,\omega_m^{nk})$-action on $C(\bT)$ defined by $\alpha(f)\colon f(x)\mapsto f(\zeta_m x)$ for $f\in C(\bT)$ and $u(x):=x^{-nk}$. 
	This action induces a right $(\bZ/m\bZ,\omega_m^{nk})$-action on $C_0(X_n)$ by the tensor product with $C[0,1]$ and the restriction to $C_0(X_n)\subset C_0([0,1]\times\bT)$. 
\end{eg}

When $m,n$ are coprime, the construction by Connes~\cite{Connes1977:periodic} gives a $(\bZ/m\bZ,\omega)$-action on $\bM_{m^\infty}$ and thus on $\cO_{n+1}\cong \cO_{n+1}\otimes\bM_{m^\infty}$ as observed in \cite[Theorem C]{Evington-GironPacheco2023:anomalous} and the remark before \cite[Problem 3.5]{Izumi2023:Gkernels}. However, this strategy is not applicable when $m,n$ are not coprime. 
By the following result, we can still construct $(\bZ/m\bZ,\omega)$-actions on $\cO_{n+1}$ in this case. 

\begin{thm}\label{thm:cyclicCuntz}
	For $m,n\in\bZ_{\geq 1}$, 
	every class in $\Ho^3(\bZ/m\bZ;\K_0^\#(\cO_{n+1}))$ can be realized as $[\wt{\omega}]=[\wt{\omega}_{\gamma,\fw}]$ for some $\omega\in \mathrm{Z}^3(\bZ/m\bZ;\bT)$ with a right $(\bZ/m\bZ,\omega)$-action $(\gamma,\fw)$ on $\cO_{n+1}$. 
	In particular, for $k\in\bZ/m\bZ$, a right $(\bZ/m\bZ,\omega_m^{k})$-action on $\cO_{n+1}$ exists if (and only if by \cite[Theorem 3.4]{Izumi2023:Gkernels}) $k=nl$ for some $l\in \bZ / m\bZ$. 
\end{thm}

Of course, we can arrange these $(\bZ/m\bZ,\omega)$-actions to be outer by \cref{eg:anomalousPimsner}. 
Note that the former assertion says that the map $\wt{\mathrm{ob}}$ in \cite[Problem 3.5]{Izumi2023:Gkernels} is surjective by considering $\wt{\omega}'$. 
Thus the latter assertion will follow from the argument of \cite[Theorem 3.4]{Izumi2023:Gkernels} once we can show the former one. 

We briefly explain the strategy of constructing right $(\bZ/m\bZ,\omega)$-actions on $\cO_{n+1}$. 
Naively, we want to apply \cref{thm:tensorKirch} to \cref{eg:cyclicspace}, but then it would not be easy to discuss $\wt{\omega}$ since $u\not\in C_0(X_n)$ by the non-unitality of $C_0(X_n)$. Instead, before applying \cref{thm:tensorKirch}, we modify $C_0(X_n)$ to be a unital C*-algebra that is $\KK$-equivalent, denoted by $A_n$ in the proof below, and take a good full corner $pA_n p$ such that its unit generates $\K_0(pA_np)$ just as $[1_{\cO_{n+1}}]$ generates $\K_0(\cO_{n+1})$. 

\begin{proof}
	As we have argued above, it suffices to construct a right $(\bZ/m\bZ,\omega_{m}^{nk})$-action $(\gamma,\fw)$ on $\cO_{n+1}$ with $[\wt{\omega}_{\gamma,\fw}]\equiv k$ in $\Ho^3(\bZ/m\bZ;\K_0^\#(\cO_{n+1}))\cong \bZ/m\bZ$ for all $k\in\bZ$. 
	%Let us briefly review the argument of \cite[Theorem 3.4]{Izumi2023:Gkernels}. 
	Indeed, since the image of the map $\mathrm{ev}_{1*}\colon \Ho^3(\bZ/m\bZ;\K_0^\#(\cO_{n+1}))\to \Ho^3(\bZ/m\bZ;\bT)\cong \bZ/m\bZ$ is $n(\bZ/m\bZ)$ (see \cite{Izumi2023:Gkernels}), 
	for any $l\in n(\bZ/m\bZ)$, there is an element $\upsilon\in\mathrm{ev}_{1*}^{-1}(\{ [\omega_m^{l}] \})\subset \Ho^3(\bZ/m\bZ;\K_0^\#(\cO_{n+1}))$, which must be $[\wt{\omega}_{\gamma,\fw}]$ for some $k\in\bZ$ and some $(\bZ/m\bZ,\omega_m^{nk})$-action $(\gamma,\fw)$, by the claimed construction. Since $[\omega_m^{l}]=\mathrm{ev}_{1*}(\upsilon)=\mathrm{ev}_{1*}([\wt{\omega}_{\gamma,\fw}])=[\omega_m^{nk}]$, we get $l\equiv nk$ modulo $m$ and thus $\omega_m^{l}=\omega_m^{nk}$, which shows the latter assertion of \cref{thm:cyclicCuntz}. 
	
	We fix a unital embedding $\cO_2\subset\cO_{\infty}^{\st}$ and define the unital C*-algebras in the UCT class, 
	\begin{align*}
		A_{n,0} := 
		\left\{ f\in C[0,1]\otimes C(\bT)\otimes \cO_{\infty}^{\st} \,\bigg|\, 
		\begin{array}{l}
			f(1,x)=f(1, e^{\frac{2\pi}{n}\sqrt{-1}} x), 
			\\
			f(0,x)\in \cO_2 , 
		\end{array}
		\forall x\in\bT
		\right\} 
	\end{align*}
	and $A_{n}:=A_{n,0}\otimes\cO_{\infty}$. 
	Since the canonical inclusion 
	$C_0(X_n)\otimes\cO_{\infty}^{\st}\subset A_{n,0}$ is a $\KK$-equivalence by $A_{n,0}/(C_0(X_n)\otimes\cO_{\infty}^{\st})\cong C(\bT)\otimes\cO_2$, we see the $\KK$-equivalence of $A_n$ and $\cO_{n+1}$. 
	The right $(\bZ/m\bZ,\omega_m^{nk})$-action on $C(\bT)$ constructed in \cref{eg:cyclicspace} induces a right $(\bZ/m\bZ,\omega_m^{nk})$-action on $C[0,1]\otimes C(\bT)\otimes \cO_{\infty}^{\st}\otimes\cO_{\infty}$ and restricts to $A_n$. 
	This action is given by $(\alpha^i,u^{\lfloor\frac{i+j}{m} \rfloor})$, where $\alpha\colon f(r,x)\mapsto f(r,\zeta_mx)$ and $u:=[(r,x)\mapsto x^{-nk}]\in\cU\cZ (A_n)$. 
	We take a projection $p_0\in A_{n,0}$ with $[p_0]=1\in\K_0(A_{n,0})\cong \bZ/n\bZ$ such that $1\otimes s_1s_1s_1^*s_1^*\leq p_0\leq 1\otimes s_1s_1^*$ for orthogonal isometries $s_1,s_2\in\cO_2\subset\cO_{\infty}^{\st}$ with $s_1s_1^*+s_2s_2^*=1$ 
	(for example, take a projection $p_0'\in A_{n,0}$ with $[p_0']=1$ and set $p_0:=1\otimes(s_1s_1s_1^*s_1^*) + (1\otimes s_1s_2) p_0' (1\otimes s_2^* s_1^*)$). 
	We let $p:=p_0\otimes 1_{\cO_{\infty}}\in A_n$. 
	
	We consider the $\bR$-action $\alpha_t$ on $A_n$ such that 
	$\alpha_t(f)(r,x) = f(r,e^{2\pi\sqrt{-1} \frac{t}{m}} x)$ 
	for $f\in A_n$, $r\in[0,1]$, $x\in\bT$, and $t\in\bR$. 
	We let $\wt{\alpha}:=(\alpha_t)_{t\in[0,1]}\in\Aut(C[0,1]\otimes A_n)$, i.e., $\wt{\alpha}(a)(t) = \alpha_t(a(t))$ for a continuous map $a\colon [0,1]\to A_n$ and $t\in[0,1]$. 
	Since $[1_{C[0,1]}\otimes p]=[\wt{\alpha}(1_{C[0,1]}\otimes p)]\in\K_0(C[0,1]\otimes A_n)$ by $\wt{\alpha}\simeq\id_{A_n}$, there is $\wt{w}\in \cU(C[0,1]\otimes A_n)$ with $\wt{w}\wt{\alpha}(1\otimes p)\wt{w}^*=1\otimes p$ because \cref{lem:corneranomalous} shows that $( 1\otimes p, \wt{\alpha}(1\otimes p) )$ and $( 1\otimes (1-p), \wt{\alpha}(1\otimes (1-p)) )$ are pairs of Murray--von Neumann equivalent projections each one of which generates $C[0,1]\otimes A_n$ as a closed ideal 
	because $(1\otimes s_1^*s_1^*\otimes1)\alpha_t(p)(1\otimes s_1s_1\otimes 1)=1=(1\otimes s_2^*\otimes 1)\alpha_t(1-p)(1\otimes s_2\otimes1)$ for all $t\in[0,1]$. 
	We write $w_t\in\cU(A_n)$ for the unitary $\wt{w}$ evaluated at each $t\in[0,1]$. 
	We may assume that $w_0=1$ by considering $w_0^*w_t$. 
	We note $\alpha=\alpha_1$ and $\alpha_m=\id$, put $w:=w_1$, 
	and for general $t\geq 0$, define 
	$w_t:=w\alpha(w)\alpha^2(w)\cdots\alpha^{\lfloor t\rfloor-1}(w)\alpha^{\lfloor t\rfloor}(w_{t-\lfloor t\rfloor})$. 
	By definition, $w_t\alpha_t(p)w_t^*=p$ and $w_{h+t}=w_h\alpha_h(w_t)$ for $t\geq 0$ and $h\in\bZ_{\geq 0}$. 
	
	Then we have a right $(\bZ/m\bZ,\omega_m^{nk})$-action $(\beta,\fv)$ on $pA_n p$ defined by for $i,j,=0,1,\cdots,m-1$ and $a\in A_n$, 
	\begin{align*}
		\beta_i(pap) &:=  pw_i\alpha_i(a)w_i^*p = p(\Ad w\alpha)^i(a)p, 
		\\
		\fv_{i,j} &:= pw_{\bra i+j\ket_m} u^{\lfloor\frac{i+j}{m}\rfloor}w_{i+j}^* p = p u^{\lfloor\frac{i+j}{m}\rfloor} w_{m}^{-\lfloor\frac{i+j}{m}\rfloor} . 
	\end{align*}
	We take paths $\wt{v}_{+}\in\cU(C[0,1]\otimes pA_np)$, $\wt{v}_{-}\in\cU(C[0,1]\otimes (1-p)A_n(1-p))$, and $\wt{v}:=\wt{v}_{+}\oplus\wt{v}_{-}\in \cU(C[0,1]\otimes A_n)$ with $\wt{v}(0)=1$ and $\wt{v}(1)=uw_m^*$. 
	This is possible because $pA_n p$ and $(1-p)A_n(1-p)$ are $\cO_{\infty}$-stable by the choice $p=p_0\otimes 1_{\cO_{\infty}}$ and thus $\K_1$-injective thanks to \cite[Theorem 3]{Jiang1997:nonstable}, \cite[Theorem 4.8]{Carrion-Gabe-Schafhauser-Tikuisis-White2023:classifyingI}. 
	%(If we chose $\wt{w}$ such that $\wt{w}(1\otimes q)=1\otimes q$ for $q:=1\otimes(s_1^3s_1^{*3}+s_2^2s_2^{*2})\otimes 1\in A_n$, we can also check the existence of $\wt{v}$ more directly by a similar argument as in \cref{lem:corneranomalous} with the aid of a unitary path $\bigl(\begin{smallmatrix}u&0\\0&u\end{smallmatrix}\bigr)\simeq \bigl(\begin{smallmatrix}u^2&0\\0&1\end{smallmatrix}\bigr)$.)
	Then we have the well-defined norm-continuous map 
	\begin{align*}
		\wt{V}\colon
		[0,\infty)\times [0,1] &\to \cU(pA_n p),
		\\
		(s,t) &\mapsto e^{-2\pi\sqrt{-1} \frac{nkt}{m} s} w_{s+tm}\alpha_s(w_{tm}^*\wt{v}(t)^*)w_s^* \wt{v}(t) p 
	\end{align*}
	where the output is a well-defined unitary because $w_{s+tm}\alpha_s(w_{tm}^*\wt{v}(t)^*)w_s^* \wt{v}(t)$ commutes with $p$. 
	We check for $s\geq 0$, $h\in\bZ_{\geq 0}$, and $t\in[0,1]$, 
	\begin{align}\label{eq:thm:cyclicCuntz}\begin{aligned}
		\wt{V}(s,0) 
		&= 
		w_{s}\alpha_s(1)w_s^* 1 p 
		= p , 
		\\
		\wt{V}(s,1) 
		&=
		e^{-2\pi\sqrt{-1}\frac{nk}{m} s} w_{s+m}\alpha_s(w_{m}^*w_mu^*)w_s^* uw_m^* p 
		\\&= 
		e^{-2\pi\sqrt{-1}\frac{nk}{m} s} w_{s+m}w_s^*w_m^* \alpha_s(u^*)u p 
		=
		p , 
		\\
		\wt{V}(h,t) 
		&= 
		e^{-2\pi\sqrt{-1} \frac{nkt}{m} h} w_{h+tm}\alpha_h(w_{tm}^*\wt{v}(t)^*)w_h^* \wt{v}(t) p 
		\\&= e^{-2\pi\sqrt{-1} \frac{nkt}{m} h} w_{h}\alpha_h(\wt{v}(t)^*)w_h^* \wt{v}(t) p ,
	\end{aligned}\end{align}
	and in particular $\wt{V}(0,t) = p$.
	Now we apply \cref{thm:tensorKirch} to get a right $(\bZ/m\bZ,\omega_m^{nk})$-action on $\cO_{n+1}$ denoted by $(\gamma,\fw)$ and an equivariant unital injective $*$-homomorphism $\iota\colon pA_np\to \cO_{n+1}$ as $p$ generates the group $\K_0(pA_np)\cong \bZ/n\bZ$. 
	By \cref{eg:anomalousPimsner}, we have $\gamma_i\iota=\iota\beta_i$ and $\fw_{i,j}=\iota(\fv_{i,j})$. 
	Then using the path $\wt{\fw}_{i,j}:=\iota(\wt{v}^{\lfloor\frac{i+j}{m}\rfloor}p)$, we have for $i,j,h=0,1,\cdots,m-1$, 
	\begin{align*}
		\wt{\omega}(i,j,h) 
		:={}& \iota\bigl(
		\beta_{h}(p \wt{v}^{-\lfloor\frac{i+j}{m}\rfloor} ) 
		\wt{v}^{-\lfloor\frac{\bra i+j\ket_m+h}{m}\rfloor} 
		\wt{v}^{\lfloor\frac{i+\bra j+h\ket_m}{m}\rfloor} 
		\wt{v}^{\lfloor\frac{j+h}{m}\rfloor} p
		\bigr) 
		\\={}&
		\iota(\beta_h(p\wt{v}^{-\lfloor\frac{i+j}{m}\rfloor}) \wt{v}^{\lfloor\frac{i+j}{m}\rfloor}p)
		=
		\iota\bigl( w_h \alpha_{h}(\wt{v}^{-\lfloor\frac{i+j}{m}\rfloor}) w_h^* \wt{v}^{\lfloor\frac{i+j}{m}\rfloor}p \bigr). 
	\end{align*}
	Thanks to \cref{eq:thm:cyclicCuntz}, we see $[\wt{\omega}_{\gamma,\fw}(i,j,h)]=[e^{2\pi\sqrt{-1}\frac{nkt}{m} \lfloor\frac{i+j}{m}\rfloor h}] \in \K_0^\#(\cO_{n+1})$ via $\iota(\wt{V}(s,t)^{\lfloor\frac{i+j}{m}\rfloor})$ for $s\in[0,h]$. 
	Therefore $(\gamma,\fw)$ gives the cohomological class $k\in \bZ/m\bZ\cong \Ho^3(\bZ/m\bZ;\K_0^\#(\cO_{n+1}))$. 
\end{proof}

\begin{cor}\label{cor:cyclicCuntz}
	For $m,n\in\bZ_{\geq 1}$ and $k\in \bZ/m\bZ$, the following are equivalent. 
	\begin{itemize}
		\item
		There is a $\Hilb_{\bZ/m\bZ,\omega_m^k}$-action on $\cO_{n+1}$ such that the induced $\bZ/m\bZ$-action on $\K_0(\cO_{n+1})$ is trivial. 
		\item
		$k\in n(\bZ/m\bZ)$. 
	\end{itemize}
\end{cor}

\begin{proof}
	This is a direct consequence of \cref{thm:cyclicCuntz} and \cref{lem:corneranomalous}. 
\end{proof}

\subsection{Bimodule case}\label{ssec:Cuntzbimod}

\begin{lem}\label{lem:cyclicprod}
	For $l,m,n\in\bZ_{\geq 1}$ and $k\in\bZ$, we have the following. 
	\begin{enumerate}
		\item
		There is a fully faithful unitary tensor functor $\Hilb_{\bZ/m\bZ,\omega_m^k}\subset\Hilb_{\bZ/mn\bZ,\omega_{mn}^k}$ 
		sending $\bC_{\bra i\ket_m}\mapsto \bC_{\bra ni\ket_{mn}}$ for $i\in\bZ$. 
		\item
		If $m$ and $n$ are coprime, there is a unitary monoidal equivalence $\Hilb_{\bZ/m\bZ,\omega_m^k}\boxtimes\Hilb_{\bZ/n\bZ,\omega_n^k}\simeq \Hilb_{\bZ/mn\bZ,\omega_{mn}^k}$ sending $\bC_{\bra i\ket_m}\boxtimes\bC_{\bra j\ket_n}\mapsto \bC_{\bra ni+mj\ket_{mn}}$ for $i,j \in\bZ$. 
	\end{enumerate}
\end{lem}

\begin{proof}
	First, note that for discrete groups $\Gamma,\Lambda$, a group homomorphism $\phi\colon \Gamma\to\Lambda$, and a $3$-cocycle $\omega\in \mathrm{Z}^3(\Lambda;\bT)$, we have a unitary tensor functor $\Hilb_{\Gamma,\omega\circ\phi^{\times3}}\to \Hilb_{\Lambda,\omega}$ sending $\bC_{g}\mapsto \bC_{\phi(g)}$. 
	Thus we see (1) by comparing 
	for $i,j,h=0,1,\cdots m-1$, 
	\begin{align*}
		&
		\exp\Bigl( 2\pi\sqrt{-1} \frac{k}{m} \Bigl\lfloor\frac{i+j}{m}\Bigr\rfloor h \Bigr) 
		=
		\exp\Bigl( 2\pi\sqrt{-1} \frac{k}{mn} \Bigl\lfloor\frac{ni+nj}{mn}\Bigr\rfloor nh \Bigr) . 
	\end{align*}
	
	If $m$ and $n$ are coprime, for uniquely determined $l\in\bZ/mn\bZ$ there is a unitary monoidal equivalence $\Hilb_{\bZ/m\bZ,\omega_m^k}\boxtimes\Hilb_{\bZ/n\bZ,\omega_n^k}\simeq \Hilb_{\bZ/mn\bZ,\omega_{mn}^l}$ sending $\bC_{\bra i\ket_m}\boxtimes\bC_{\bra j\ket_n}\mapsto \bC_{\bra ni+mj\ket_{mn}}$. 
	Since its restrictions to $\Hilb_{\bZ/m\bZ,\omega_m^k}\boxtimes\mathbbm{1}$ yields a fully faithful unitary functor into $\Hilb_{\bZ/mn\bZ,\omega_{mn}^k}$ by (1), it must hold that $k\equiv l \mod m$. Similarly we can see $k\equiv l \mod n$ and thus modulo $mn$, which implies (2). 
\end{proof}

\begin{eg}\label{eg:cyclicCuntzbimod}
	Let $l,n\in\bZ_{\geq 1}$ and $k\in\bZ$. 
	The $*$-automorphism $\alpha\colon C(\bZ/n\bZ) \ni f \mapsto [h\mapsto f(1+h)] \in C(\bZ/n\bZ)$ with $C(\bZ/n\bZ)\ni u\colon h\mapsto e^{-2\pi\sqrt{-1}\frac{k}{n}h}$ gives a right $(\bZ/ln\bZ,\omega_{ln}^{lk})$-action on $C(\bZ/n\bZ)$. 
	
	Moreover, $\lambda:=\alpha$ with the constants 
	\begin{align*}
		\bv_j := e^{-2\pi\sqrt{-1}\frac{k}{ln^2}j}1_{C(\bZ/n\bZ)} \in \cU(C(\bZ/n\bZ))
	\end{align*}
	for $j=0,1,\cdots,ln-1$ is a well-defined $(\bZ/ln\bZ,\omega_{ln}^{lk})$-$*$-endomorphism on $(C(\bZ/n\bZ),\alpha^i,u^{\lfloor \frac{i+j}{ln} \rfloor})$ as we check for $i,j=0,1,\cdots,ln-1$, 
	\begin{align*}
		&
		u^{\lfloor\frac{i+j}{ln}\rfloor} \alpha_j(\bv_{i})\bv_{j} 
		=
		u^{\lfloor\frac{i+j}{ln}\rfloor} \bv_{\bra i+j\ket_{ln}} e^{ -2\pi\sqrt{-1} \frac{k}{ln^2} ln\lfloor\frac{i+j}{ln}\rfloor }
		=
		\bv_{\bra i+j\ket_{ln}} \lambda(u^{\lfloor\frac{i+j}{ln}\rfloor}) .
	\end{align*}
\end{eg}

\begin{thm}\label{thm:cyclicCuntzbimod}
	Let $m,n\in \bZ_{\geq 1}$ and $k\in \bZ$. 
	We express $m=\prod_{i=1}^{h}p_i^{r_i}$ and $n=\prod_{i=1}^{h}p_i^{s_i}$ using mutually different prime numbers $p_1,\cdots,p_h$, and indices $r_1,\cdots,r_h,s_1,\cdots,s_h\in\bZ_{\geq 0}$ with $h\in\bZ_{\geq 0}$. 
	Then a $\Hilb_{\bZ/m\bZ,\omega_m^{k}}$-action on $\cO_{n+1}$ exists if and only if for every $i=1,\cdots,h$, at least one of the following holds. 
	\begin{itemize}
		\item
		$k \in p_i^{\min \{r_i, s_i\}}\bZ$. 
		\item
		$k \in p_i^{r_i-s_i+1}\bZ$ and $p_i\neq 2$. 
		\item
		$k \in p_i^{r_i-s_i+2}\bZ$ and $p_i=2$. 
	\end{itemize}
\end{thm}

Here, the existence of a $\Hilb_{\bZ/m\bZ,\omega_m^{k}}$-action on $\cO_{n+1}$ is equivalent to the existence of a $(\bZ/m\bZ,\omega_m^{k})$-action on $\cO_{n+1}^{\st}$ by \cref{lem:corneranomalous}. 
Again, note that we can arrange these $(\bZ/m\bZ,\omega)$-actions to be outer by \cref{eg:anomalousPimsner}. 

\begin{proof}
	Suppose that $\Hilb_{\bZ/m\bZ,\omega_m^{k}}$-action on $\cO_{n+1}$ exists. 
	Then for all $i=1,\cdots,h$, the full unitary tensor subcategory  $\Hilb_{\bZ/p_i^{r_i}\bZ,\omega_{p_i^{r_i}}^{k}}\subset \Hilb_{\bZ/m\bZ,\omega_m^{k}}$ in view of \cref{lem:cyclicprod} (1) can act on $\cO_{n+1}\otimes\bM_{\prod_{j\neq i}p_j^{\infty}}\cong \cO_{p_i^{s_i}+1}$. 
	Conversely, if a $\Hilb_{\bZ/p_i^{r_i}\bZ,\omega_{p_i^{r_i}}^{k}}$-action on $\cO_{p_i^{s_i}+1}$ exists for all $i=1,\cdots,h$, 
	then 
	$\Hilb_{\bZ/m\bZ,\omega_{m}^{k}}$ admits an action on $\cO_{p_i^{s_i}+1}\otimes\bigotimes_{j\neq i}\bM_{p_j^{\infty}}\cong \cO_{p_i^{s_i}+1}$ with the aid of \cref{lem:cyclicprod} (2) 
	because $\Hilb_{\bZ/p_j^{r_j}\bZ,\omega_{p_j^{r_j}}^{k}}$ for $j\neq i$ admits an action on $\bM_{p_j^{\infty}}$ thanks to \cite[Theorem C]{Evington-GironPacheco2023:anomalous}. 
	Hence $\Hilb_{\bZ/m\bZ,\omega_{m}^{k}}$ can act on $\bigoplus_{i=0}^{h}\cO_{p_i^{s_i}+1}$ and thus on $\cO_{n+1}$ by \cref{thm:tensorKirch} with the aid of the Kirchberg--Phillips theorem. 
	Thus it suffices to show the case when $m=p^r$ and $n=p^s$ for $r,s\in\bZ_{\geq 0}$ and a prime $p$. 
	We put $\sigma:=1$ if $p\neq 2$ and $\sigma:=2$ if $p=2$.

	First, suppose that there is a $\Hilb_{\bZ/p^{r}\bZ,\omega_{p^{r}}^k}$-action $(\alpha,\fu)$ on $\cO_{p^{s}+1}$. 
	It suffices to show a contradiction by assuming $k \not\in p^{\min\{ r,s,r-s+\sigma \}}\bZ$. 
	Then $r,s,r-s+\sigma\geq 1$, and $\alpha$ induces a group homomorphism 
	\begin{align*}
		f_0\colon \bZ/p^{r}\bZ\to \Aut(\K_0(\cO_{p^{s}+1}))
		\cong (\bZ/p^{s}\bZ)^{\times}. 
	\end{align*}
	If $p^s$ is $2$ or odd, there is a primitive root $g\in(\bZ/p^s\bZ)^{\times}$. 
	If $p^s\in4\bZ$, there is $g\in(\bZ/p^s\bZ)^{\times}$ such that $(\bZ/p^s\bZ)^{\times}=\bra g\ket_{\text{group}}\times\{\pm 1\}$. 
	For such $g$, we may assume that $\alpha$ induces a group homomorphism 
	\begin{align*}
		f\colon \bZ/p^{r}\bZ\to \bra g\ket_{\text{group}} \cong \bZ/p^{\max\{s-\sigma,0\}}(p-1)\bZ . 
	\end{align*}
	Indeed, since otherwise $p^s\in4\bZ$, we may tensor a $\bZ/2^r\bZ$-action on $\cO_\infty^{\st}$ such that $1\in\bZ/2^r\bZ$ acts on $\bZ\cong\K_0(\cO_\infty^{\st})$ by $-1$, which exists thanks to \cite{Katsura2008:construction}, to obtain an action as claimed. 
	Then $\Ker f=\bZ/2^r\bZ$ if $p=2$ and $s=1$, and $\Ker f\supset p^{s-\sigma}\bZ/p^{r}\bZ$ otherwise. Thus we obtain a $\Hilb_{\bZ/p^{\min\{ r , r-s+\sigma \}}\bZ,\omega_{p^{\min\{ r , r-s+\sigma \}}}^k}$-action on $\cO_{p^s+1}$ fixing the $\K_0$-group via \cref{lem:cyclicprod} (1), which forces $k\in p^{s}\bZ + p^{\min\{ r , r-s+\sigma \}}\bZ$ by \cite[Theorem 3.4]{Izumi2023:Gkernels} (see also \cref{cor:cyclicCuntz}). 
	This contradicts the assumption on $k$ as desired. 
	
	Conversely, we are going to show that $\Hilb_{\bZ/p^{r}\bZ,\omega_{p^{r}}^k}$ can act on $\cO_{p^{s}+1}$ if $k \in p^{\min\{ r,s,r-s+\sigma \}}\bZ$. 
	It suffices to show the case of $r\geq s-\sigma$ because if so, in particular, there is a $\Hilb_{\bZ/p^{s-\sigma}\bZ,\omega_{p^{s-\sigma}}^k}$-action on $\cO_{p^{s}+1}$ for all $k\in\bZ$, which restricts to a $\Hilb_{\bZ/p^r\bZ,\omega_{p^r}^k}$-action on $\cO_{p^s+1}$ for any $r\leq s-\sigma$ by \cref{lem:cyclicprod} (1). 
	By \cref{cor:cyclicCuntz}, $\Hilb_{\bZ/p^{r}\bZ,\omega_{p^{r}}^k}$ can act on $\cO_{p^{s}+1}$ if $k\in p^{\min \{r,s\}}\bZ$. 
	From now on, we consider the case when $k\in p^{ r-s+\sigma }\bZ \setminus p^{\min\{r,s\}}\bZ$ and $r\geq s-\sigma$. 
	Then $s\geq 3$ if $p=2$, and $s\geq 2$ otherwise. 
	We can take $g\in \bZ$ with $g\leq -2$ such that $d:=(1-g^{s-\sigma})/p^s$ is an integer coprime to $p$. 
	Indeed, for $g_1\in\bZ$ with $g_1\leq -p-2$ giving an element in $(\bZ/p^s\bZ)^\times$ of order $p^{s-\sigma}$, at least one of $g_1$ or $g_1\pm p$ satisfies the properties of $g$ by the following calculations. 
	When $p$ is odd, then 
	\begin{align*}
		(g_1+ p)^{p^{s-1}} 
		={}& 
		g_1^{p^{s-1}}+ p^{s}g_1^{p^{s-1}-1} 
		+ \sum_{j=2}^{p^{s-1}} \frac{p^{s-1}}{j}\binom{p^{s-1}-1}{j-1}p^j g_1^{p^{s-1}-j} 
		\\\in{}& 
		g_1^{p^{s-1}}+ p^{s}g_1^{p^{s-1}-1} + p^{s+1}\bZ , 
	\end{align*}
	and when $p=2$ and $s= 3$, then by choosing $\tau\in\{\pm1\}$ with $1+\tau g_1\in 2\bZ \setminus 4\bZ$, 
	\begin{align*}
		&
		(g_1+\tau 2)^{2^{s-2}} 
		= 
		g_1^{2}+\tau 4g_1 
		+ 4 
		=
		g_1^{2^{s-2}}+ 8\frac{\tau g_1 +1}{2} , 
	\end{align*}
	whereas when $p=2$ and $s\geq 4$, then by choosing $\tau\in\{\pm 1\}$ with $2^{s-2}-1+\tau g_1\in 4\bZ$, 
	\begin{align*}
		(g_1+\tau 2)^{2^{s-2}} 
		={}&
		g_1^{2^{s-2}}
		+ \tau 2^{s-1}g_1^{2^{s-2}-1} 
		+ \frac{2^{s-2}(2^{s-2}-1)}{2} 2^2 g_1^{2^{s-2}-2} 
		\\&+ 
		\frac{2^{s-2}}{4} \binom{2^{s-2}-1}{3} 2^4 g_1^{2^{s-2}-4} 
		+ \sum_{\substack{3\leq j\leq p^{s-1} \\ j\neq 4}} \frac{2^{s-2}}{j}\binom{2^{s-2}-1}{j-1} \tau^j2^j g_1^{2^{s-2}-j} 
		\\\in{}& 
		g_1^{2^{s-2}}+ 2^{s} 3^{-1} (2^{s-2}-1)(2^{s-3}-1)(2^{s-2}-3) g_1^{2^{s-2}-4} + 2^{s+1}\bZ . 
	\end{align*}
	
	Here, the order of $g$ in $(\bZ/p^s\bZ)^{\times}$ must be $p^{s-\sigma}$ since otherwise for some $l\in\bZ$ and $j\in\bZ_{\geq 1}$ it holds that $g^{p^{s-\sigma}} = (1+p^sl)^{p}=1+p^{s+1}\bZ$, which contradicts $d=(1-g^{p^{s-\sigma}})p^{-s}\in\bZ\setminus p\bZ$. 
	Then we have the well-defined surjective ring homomorphism 
	\begin{align*}
		\bZ[d^{-1}][\bZ/p^{s-\sigma}\bZ] = \bZ[d^{-1}][Y]/(Y^{p^{s-\sigma}}-1) \ni Y^j \mapsto g^j \in \bZ/p^s\bZ . 
	\end{align*}
	Its kernel is 
	\begin{align*}
		&
		%\biggl\{ \sum_{j=0}^{p^{s-\sigma}-1} a_jY^j \;\bigg|\;
		%a_j\in\bZ[d^{-1}], \sum_{j=0}^{p^{s-\sigma}-1} a_jg^j\in p^s\bZ[d^{-1}] %\biggr\} 
		%\\={}&
		(Y-g)\bZ[d^{-1}][\bZ/p^{s-\sigma}\bZ] + p^s\bZ[d^{-1}][\bZ/p^{s-\sigma}\bZ] 
		=%\\={}&
		(Y-g)\bZ[d^{-1}][\bZ/p^{s-\sigma}\bZ] , 
	\end{align*}
	because 
	\begin{align*}
		p^s &= d^{-1}(Y^{p^{s-\sigma}}-g^{p^{s-\sigma}}) 
		%= (Y-g) \frac{\sum_{j=1}^{p^{s-\sigma}} g^{p^{s-\sigma}-j}Y^{j-1}}{d} 
		%\\&
		\in (Y-g)\bZ[d^{-1}][\bZ/p^{s-\sigma}\bZ]. 
	\end{align*}
	The $\bZ[d^{-1}][\bZ/p^{s-\sigma}\bZ]$-linear map 
	\begin{align*}
		\bZ[d^{-1}][\bZ/p^{s-\sigma}\bZ]\ni f \mapsto (Y-g)f\in (Y-g)\bZ[d^{-1}][\bZ/p^{s-\sigma}\bZ]
	\end{align*}
	is injective because any element in its kernel must be of the form 
	$\sum_{j=0}^{p^{s-\sigma}-1} a_jY^j$ for $a_j\in\bZ[d^{-1}]$ such that $a_{j}=ga_{j+1}$ for all $j$ and thus $a_j=g^{p^{s-\sigma}}a_j$, which implies $a_j=0$ by $|g|\geq 2$. 
	Hence, we obtain the short exact sequence of right $\bZ[d^{-1}][\bZ/p^{s-\sigma}\bZ]$-modules 
	\begin{align}\label{eq:thm:cyclicCuntzbimod}
		0\to \bZ[d^{-1}][\bZ/p^{s-\sigma}\bZ] \xto{Y-g} \bZ[d^{-1}][\bZ/p^{s-\sigma}\bZ] 
		\to \bZ/p^s\bZ\to 0. 
	\end{align}
	Via the ring homomorphism $\bZ[\bZ/p^{r}\bZ]\ni \bZ[Y]/(Y^{p^{r}}-1)\bZ[Y]\ni Y \mapsto Y\in \bZ[\bZ/p^{s-\sigma}\bZ]$, which is surjective by the assumption $r\geq s-\sigma$, we may regard \cref{eq:thm:cyclicCuntzbimod} as an exact sequence of $\bZ[d^{-1}][\bZ/p^{r}\bZ]$-modules. 
	
	Thanks to $k\in p^{r-s+\sigma}\bZ$, it follows from \cref{eg:cyclicCuntzbimod} tensored with $\bM_{d^\infty}$ that there is a right $(\bZ/p^r\bZ,\omega_{p^r}^{k})$-action on $A:=\bM_{d^\infty}\otimes C(\bZ/p^{s-\sigma}\bZ)$ and a $(\bZ/p^r\bZ,\omega_{p^r}^{k})$-$*$-endomorphism $(\id_{\bM_{d^\infty}}\otimes\lambda,1\otimes\bv)$ on $A$ such that its action on $\K_0(A)\cong \bZ[d^{-1}][\bZ/p^{s-\sigma}\bZ]$ equals the multiplication by $Y$. 
	Since the map on $\K_0$-groups induced by the proper $(\bZ/p^r\bZ,\omega_{p^r}^{k})$-$(A,A)$-correspondence $(\id_{\bM_{d^\infty}}\otimes\lambda)\oplus \id_{A}^{\oplus -g}\colon A\to \bM_{1-g}\otimes A$ fits into \cref{eq:thm:cyclicCuntzbimod}, we obtain a $\Hilb_{\bZ/p^r\bZ,\omega_{p^r}^{k}}$-action on $\cO_{p^{s}+1}$ by \cref{prop:inductiveMC} and \cref{thm:tensorKirch}. 
\end{proof}

\appendix
\section{Involutory operators over a ring of integers}\label{sec:numbertheory}
We separate some number-theoretic ingredients from the proof of \cref{thm:Temperley-Lieb-Jones}. We refer to \cite[Chapter I]{Neukirch:book} for some fundamental facts about number theory. 

\begin{rem}\label{notn:algnumber}
	In this section, we keep in mind the following situation. 
	\begin{enumerate}
		\item
		Let $K$ be a finite Galois extension of $\bQ$ and $R=\cO_K$ be the ring of integers of $K$, i.e., the set of $a\in K$ such that there is a monic polynomial $f\in\bZ[X]$ with $f(a)=0$. 
		Then we let $G:=\mathop{\mathrm{Gal}}(K/\bQ)$ denote its Galois group, which is the same as the ring automorphism groups $\Aut(K)=\Aut(R)$ of $K$ and $R$. 
		Note that $R$ has global dimension $1$ as it is a Dedekind domain. %\cite[Theorem I.3.1]{Neukirch:book}
		In particular, any $R$-submodule of a projective $R$-module is projective by Schanuel's lemma. 
		\item
		The ideal $2R\subset R$ uniquely decomposes $2R=\prod_{i=1}^{n}\fp_i^{m_i}$ into mutually different prime ideals $\fp_1,\cdots,\fp_n\subset R$ with some $m_i\in \bZ_{\geq 1}$. 
		We assume that $K/\bQ$ is unramified at $2$, i.e., $m_i=1$ for all $i$. Then, $R/2R \cong \prod_{i=1}^{n}R/\fp_i$ is a finite product of finite fields $R/\fp_i$. 
		We write $e_i\in R/2R$ for the idempotent element corresponding to the unit of $R/\fp_i$. 
		\item
		The action of $G$ on the set $\{\fp_1,\cdots,\fp_n\}$ defined by $\fp_i\mapsto g(\fp_i)$ for each $g\in G$ is transitive (\cite[Proposition I.9.1]{Neukirch:book}). 
		We have an induced $G$-action on $R/2R$ by automorphisms. It follows that the action of $G$ on the set $\{ e_1,\cdots,e_n \}$ defined by $e_i\mapsto g(e_i)$ for each $g\in G$ is again transitive. 
		\item
		We use the variable $Y$ to express $R[\bZ/2\bZ]:=R[Y]/(Y^2-1)R[Y]$. 
		Then the $G$-action on $R$ extends to that on $R[\bZ/2\bZ]$ by $g(Y):=Y$ for $g\in G$. 
	\end{enumerate}
\end{rem}

\begin{eg}\label{eg:cyclotomic}
	Let $p\in\bZ_{\geq 1}$ be an odd prime and put $\zeta_p:=e^{\frac{2\pi}{p}\sqrt{-1}}$. 
	\begin{enumerate}
		\item
		Note that $\bZ[2\cos\frac{2\pi}{p}]$ is the ring of integers of the field $\bQ(2\cos\frac{2\pi}{p})$. Indeed, this follows from the fact (\cite[Proposition I.10.2]{Neukirch:book}) that $\bZ[\zeta_p]$ is the the ring of integers of $\bQ(\zeta_p)$ combined with 
		\[\bR\cap\bZ[\zeta_p] = \bZ+\sum_{k=1}^{p-2} \bZ 2\cos\frac{2k\pi}{p} = \bZ\Bigl[2\cos\frac{2\pi}{p}\Bigr] . \]
		%Here, the former equality is verified using $[\bQ(2\cos\frac{2\pi}{p}):\bQ]=[\bQ(\zeta_p):\bQ(2\cos\frac{2\pi}{p})]^{-1}[\bQ(\zeta_p):\bQ]=\frac{p-1}{2}$ and the last equality holds since $2\cos\frac{2k\pi}{p} -(2\cos\frac{2\pi}{p})^k \in \sum_{j=0}^{k-1}\bZ(2\cos\frac{2\pi}{p})^j$ inductively for $k\geq 1$ by the relation $2\cos((k+1)\theta) = 4\cos(\theta)\cos(k\theta) - 2\cos((k-1)\theta)$. 
		\item
		It is well-known that $\bQ(\zeta_p)/\bQ$ is a finite Galois extension with an abelian Galois group, which is unramified at $2$ by \cite[Proposition I.10.3, Corollary I.10.4]{Neukirch:book}. 
		Thus $K:=\bQ(\zeta_p)$ satisfies the setting of \cref{notn:algnumber} and therefore so does the subfield $K:=\bQ(2\cos\frac{2\pi}{p})\subset\bQ(\zeta_p)$. 
	\end{enumerate}
\end{eg}

The goal of this section is to show the following resolution. 
We write $R_+$ and $R_-$ for the right $R[\bZ/2\bZ]$-modules defined by $R_{\pm}:=R$ on which $Y\in R[\bZ/2\bZ]$ acts by the scalar $\pm 1\in R$, respectively. 

\begin{thm}\label{thm:Z/2module}
	In the setting of \cref{notn:algnumber}, any countable right $R[\bZ/2\bZ]$-module $M$ fits into some $R[\bZ/2\bZ]$-linear exact sequence of the form 
	\begin{align}\label{eq:thm:Z/2module0}\begin{aligned}
			0&\to R_{+}^{\oplus\infty}\xto{F_3}  R_{+}^{\oplus\infty} \oplus R[\bZ/2\bZ]^{\oplus\infty} \xto{F_2} R_{+}^{\oplus\infty} \oplus R[\bZ/2\bZ]^{\oplus\infty} 
			\\&\xto{F_1} R[\bZ/2\bZ]^{\oplus\infty}\xto{F_0} M\to 0. 
	\end{aligned}\end{align}
\end{thm}

If $2R\subset R$ is a prime ideal, this is a consequence of the structural results on torsion-free $R$-modules similar to \cite[Theorem 2.8]{Benson-Kumjian-Phillips2003:symmetries} and \cite[Theorem 3.4]{Butler-Campbell-Kovacs2004:infinite}. 
In that case, we can actually get a resolution a bit shorter than \cref{eq:thm:Z/2module0}. 
However, often $2\bZ[2\cos\frac{2\pi}{p}]\subset \bZ[2\cos\frac{2\pi}{p}]$ is not a prime ideal, which makes some of the arguments above invalid. Fortunately, thanks to the presence of the Galois group, we can still obtain a statement (\cref{lem:BKPthm2.8}) that is enough for our purpose. 

We begin with recalling the following lemma, which is a special case of \cite[Theorem 6.1]{Benson-Goodearl2000:periodic}. 

\begin{lem}[{Benson--Goodearl~\cite{Benson-Goodearl2000:periodic}}]
	\label{lem:BKPlem2.2}
	Let $R$ be an integral domain and $M$ be a right $R[\bZ/2\bZ]$-module that is projective as an $R$-module. If $M/2M$ is projective as an $(R/2R)[\bZ/2\bZ]$-module, then $M$ is projective as an $R[\bZ/2\bZ]$-module. 
\end{lem}

Indeed, this is a consequence of the fact that the $R[\bZ/2\bZ]$-module $M$ is projective if and only if there is an $R$-linear endomorphism $\phi$ on $M$ such that $x=\phi(x)+\phi(xY)Y$ for all $x\in M$ (see \cite[Theorem 1]{Higman1954:modules}, for example) with the aid of the $R[\bZ/2\bZ]$-linear isomorphism $M\ni x\mapsto 2x\in 2M$. 
Using this lemma with the aid of \cref{notn:algnumber} (2), we can show the following lemma in the totally same way as \cite[Lemma 2.5]{Benson-Kumjian-Phillips2003:symmetries}. 
When a right $R[\bZ/2\bZ]$-module $M$ is torsion-free as an $R$-module, $M\otimes_{R} K$ is a well-defined right $K[\bZ/2\bZ]$-module containing $M=M\otimes_{R} 1$ as an $R[\bZ/2\bZ]$-submodule. 
\begin{lem}[cf.~{Benson--Kumjian--Phillips~\cite{Benson-Kumjian-Phillips2003:symmetries}}]\label{lem:BKPlem2.5}
	In the setting of \cref{notn:algnumber}, let $M$ be a right $R[\bZ/2\bZ]$-module that is projective as an $R$-module. 
	Suppose that $M\frac{1+Y}{2}\cap M=M(1+Y)$ and $M\frac{1-Y}{2}\cap M=M(1-Y)$ inside $M\otimes_RK$. 
	Then $M$ is projective as an $R[\bZ/2\bZ]$-module. 
\end{lem}

%We give its proof for convenience of the reader. 
%\begin{proof}
%	By \cref{lem:BKPlem2.2}, it suffices to show that $M/2M$ is a projective $(R/2R)[\bZ/2\bZ]$-module. We set $N:=(M(1+Y)+M(1-Y))/2M$ %=(M/2M)(1+Y)$
%	and observe 
%	\begin{align}\label{eq:lem:BKPlem2.5}
%		\Ker(1+Y \colon M/2M\to M/2M) = (M/2M)(1+Y) = N. 
%	\end{align}
%	Indeed, for any $m\in M$ with $m(1+Y)\in 2M$, by the assumption we have 
%	\begin{align*}
%		m &= m\frac{1+Y}{2} + m\frac{1-Y}{2}\in \Bigl(M\frac{1+Y}{2}\cap M\Bigr)+\Bigl(M\frac{1-Y}{2}\cap M\Bigr) \\&= M(1+Y) + M(1-Y) = M(1+Y) + 2M, 
%	\end{align*}
%	which shows \cref{eq:lem:BKPlem2.5} combined with $(1+Y)^2=2(1+Y)$. 
%	
%	Since $N$ is a direct summand of $M/2M$ as an $R/2R$-module by \cref{notn:algnumber} (2), there is a projective $R/2R$-submodule $L\subset M/2M$ with $M/2M=N\oplus L$. Note that $N=(M/2M)(1+Y)=(N+L)(1+Y)=L(1+Y)$ by \cref{eq:lem:BKPlem2.5}. For any $x\in L$ with $xY\in L$, we have $x(1+Y)\in N\cap L=0$ and thus $x\in N\cap L=0$ by \cref{eq:lem:BKPlem2.5}. Hence $L\cap LY=0$. Now we see $M/2M=L\oplus LY$, and thus $M/2M\cong L\otimes_{R/2R} (R/2R)[\bZ/2\bZ]$, which is a projective $(R/2R)[\bZ/2\bZ]$-module. 
%\end{proof}

The following lemma is the replacement of \cite[Lemma 2.7]{Benson-Kumjian-Phillips2003:symmetries} or \cite[Lemma 3.2]{Butler-Campbell-Kovacs2004:infinite} in  our setting. 
For a unital commutative ring $S$ acted on by a finite group $\Gamma$ as ring automorphisms and a right $S$-module $N$, we define a right $S$-module $N[g]$ for each $g\in\Gamma$ by $N[g]:=N$ as an abelian group, where we formally indicate an element in $N[g]$ by $x[g]$ for $x\in N$, and $(x[g]) a:=(x g(a)) [g]$ for $x\in N$ and $a\in S$. 
Moreover, we let $N[\Gamma]:=\bigoplus_{g\in \Gamma}N[g]$. 

\begin{lem}\label{lem:BCKlem3.2G}
	In the setting of \cref{notn:algnumber}, let $M$ be a free right $R$-module and $N\subset M$ be an $R$-submodule such that $2 M\subset N\subset M$. 
	Then there are projective $R$-submodules $L_0,L_1\subset M[G]$ such that $M[G]=L_0\oplus L_1$ and $N[G]=2L_0\oplus L_1$. 
\end{lem}

%The point in the proof is that we can amplify any $R/2R$-module $L$, which is automatically projective because $2$ is unramified, to get a free $R/2R$ module $L[G]$ thanks to the transitivity of $G$-action on $\{ \fp_1,\cdots,\fp_n \}$. 
Here, the amplifications $M[G]$ and $N[G]$ are essential. For example, let $M=R$ and $N=\fp_1$ and suppose $n\geq 2$. 
Note that $R/N= R/\fp_1$ and $N/2R\cong \bigoplus_{i\neq 1}R/\fp_i$, which are projective as modules over $R/2R\cong \bigoplus_{i}R/\fp_i$ but not free. 
Although $N$ satisfies $2R\subset N\subset R$, there are no projective $R$-submodules $I_0,I_1\subset R$ with $R=I_0\oplus I_1$ and $N=2I_0\oplus I_1$. Indeed, if there are such $I_0$ and $I_1$, then $I_0/2I_0\cong R/N\neq 0\neq N/2R \cong I_1/2I_1$, which implies that there are non-zero elements $x_0\in I_0$ and $x_1\in I_1$, but we get a contradiction as $\dim_\bQ \bQ\otimes_\bZ R \geq \dim_\bQ \bQ\otimes_\bZ (x_0R\oplus x_1R)= 2\dim_\bQ \bQ\otimes_\bZ R$. 
To avoid such a situation, we need to average the ranks of quotient $R/\fp_i$-modules, which is possible thanks to the transitivity of the $G$-action on $\{\fp_1,\cdots,\fp_n\}$. 

\begin{proof}
	We write $q\colon M\to M/N$ and $\wt{q}\colon M[G]\to M[G]/N[G]\cong (M/N)[G]$ for the quotient maps. 
	Since $2M[g]\subset N[g]$ for $g\in G$, we see that $M[g]/N[g]$ is an $R/2R$-module with the canonical $R/2R$-linear isomorphism $M[g]/N[g]\cong (M/N)[g]$. In particular, $(M[g]/N[g])e_i$ is an $R/\fp_i$-vector space for each $i=1,\cdots,n$. 
	
	Let $A_0\subset M$ be an $R$-linear free basis. 
	For each $i=1,\cdots,n$, there is a subset $B_{i,0}\subset A_0$ such that $B_{i,0}\ni y\mapsto q(y)e_i \in (M/N)e_i$ is injective and $q(B_{i,0})e_i$ forms an $R/\fp_i$-linear basis of $(M/N)e_i$. 
	
	Inductively on $j=1,\cdots,n$, suppose we have an $R$-linear basis $A_{j-1}$ of $M$ and subsets $B_{i,j-1}\subset A_{j-1}$ such that $( q(y)e_i \mid y\in B_{i,{j-1}} )$ is an $R/\fp_i$-linear basis of $(M/N)e_i$ for $i=1,\cdots,n$ and $q(A_{j-1}\setminus B_{i,{j-1}})e_i=0$ for $i=1,\cdots,j-1$. 
	Then for each $x\in A_{j-1}\setminus B_{j,j-1}$, there is $y_j(x)\in \bigoplus_{y\in B_{j,j-1}} y\bigl(\prod_{i\neq j}\fp_i\bigr)$ such that 
	\begin{align*}
		q(x-y_j(x))e_i=\left\{ \begin{array}{ll}q(x)e_i &(\text{if }i\neq j) \\ 0&(\text{if }i=j)
		\end{array}\right. 
	\end{align*}
	by using $d_j\in R$ with $d_j\mod 2R = e_j$. 
	We define $x_j(x):=x-y_j(x)\in M$ for $x\in A_j\setminus B_{j,j-1}$ and $x_j(x):=x$ for $x\in B_{j,j-1}$. 
	Then $A_j:=x_j(A_{j-1})$ is an $R$-linear basis of $M$. When we let $B_{i,j}:=x_j(B_{i,j-1})\subset A_j$, then $( q(y)e_i \mid y\in B_{i,j} )$ is an $R/\fp_i$-linear basis for all $i=1,\cdots,n$, and 
	$q(A_j\setminus B_{i,j})e_i=0$ for all $i=1,\cdots,j$. 
	
	Then $A:=A_n$ is an $R$-linear free basis of $M$. 
	When we let $B:=\bigcup_{i=1}^{n}B_{i,n}\subset A$, then we have that $B\ni y\mapsto q(y)\in (M/N)\setminus\{0\}$ is injective and that $M/N=\bigoplus_{y\in B} q(y)(R/2R)$. 
	Here, we warn that $q(y)(R/2R)$ might be smaller than $R/2R$. 
	
	When we put $I_y:=\{ i=1,\cdots, n \mid q(y)e_i\neq 0 \}$ and $n_y:=\# I_y \in \{ 1,\cdots, n \}$ for each $y\in B$, it holds that $\bigoplus_{g\in G}\wt{q}(y[g])(R/2R)$ is a free $R/2R$-module, whose rank is 
	\begin{align*}
		\#\{ g\in G \mid q(y)g(e_i)\neq 0 \} = \frac{\# G}{n} n_y, 
	\end{align*}
	independently of $i=1,\cdots,n$, thanks to \cref{notn:algnumber} (3).  
	For $y\in B$, when we fix a subset $G_y\subset G$ with $\# G_y =\frac{\# G}{n} n_y$, then for all $g\in G_y$, we can take $\wt{y}_g\in M[G]$ with $\wt{y}_g-y[g]\in \sum_{h\in G\setminus G_y} y[h] \bigl(\prod_{i\in I_y}g^{-1}(\fp_i)\bigr)$ such that $(\wt{q}(\wt{y}_g) \mid g\in G_y)$ is an $R/2R$-linear free basis of $\bigoplus_{g\in G}\wt{q}(y[g])(R/2R)$ by using $d_i\in R$ with $d_i\mod 2R=e_i$. 
	We let $\wt{B}:=\{ \wt{y}_g \mid y\in B, g\in G_y \}$ and 
	\begin{align*}
		\wt{A}:= \Bigl(\Bigl(\bigcup_{g\in G}A[g]\Bigr) \setminus \{ y[g] \mid y\in B, g\in G_y \}\Bigr) \cup \wt{B} . 
	\end{align*}
	Then $\wt{A}$ is an $R$-linear free basis of $M[G]$, and $(\wt{q}(w) \mid w\in\wt{B})$ is an $R/2R$-linear free basis of $(M/N)[G]$. 
	
	Finally, for $z\in \wt{A}\setminus\wt{B}$, there is $w(z)\in \bigoplus_{w\in \wt{B}} wR$ such that $\wt{q}(z-w(z))=0$. Thus if we let $L_1:=\bigoplus_{z\in \wt{A}\setminus\wt{B}}(z-w(z))R$ and $L_0:=\bigoplus_{w\in\wt{B}}wR$, 
	we see that $L_0\oplus L_1=M[G]$ and $2L_0\oplus L_1=\Ker\wt{q}=N[G]$ as desired. 
\end{proof}

Now we can show a variant of \cite[Theorem 2.8]{Benson-Kumjian-Phillips2003:symmetries}. 
For a right $R[\bZ/2\bZ]$-module $M$ and $g\in G$, we define the $R[\bZ/2\bZ]$-module $M[g]$ by $M[g]$ as an $R$-module and $x[g]Y:=xg(Y)[g]=xY[g]$ for $x\in X$, taking \cref{notn:algnumber} (4) into account. 
Note that $R_{\pm}[g]\cong R_{\pm}$ and $R[\bZ/2\bZ][g]\cong R[\bZ/2\bZ]$ as $R[\bZ/2\bZ]$-modules. 

\begin{thm}\label{lem:BKPthm2.8}
	In the setting of \cref{notn:algnumber}, let $M$ be a countable right $R[\bZ/2\bZ]$-module that is projective as an $R$-module. 
	Then, we have an $R[\bZ/2\bZ]$-linear isomorphism 
	\begin{align*}
		M[G]\oplus R[\bZ/2\bZ]^{\oplus\infty} \cong P_+\oplus P_-\oplus R[\bZ/2\bZ]^{\oplus\infty} , 
	\end{align*}
	for some countable right $R[\bZ/2\bZ]$-modules $P_+$ and $P_-$ that are projective as $R$-modules such that $Y\in R[\bZ/2\bZ]$ acts on $P_+$ and $P_-$ by the scalars $+1$ and $-1\in R$, respectively. 
\end{thm}

\begin{proof}
	We put $L:=M\oplus R[\bZ/2\bZ]^{\oplus\infty}$ and $\wt{L}:=L[G]\cong M[G]\oplus R[\bZ/2\bZ]^{\oplus\infty}$. Then since $R[\bZ/2\bZ]\frac{1\pm Y}{2}\cong R$ as $R$-modules, we see that $L\frac{1\pm Y}{2}\subset L\otimes_R K$ is free as an $R$-module by \cref{rem:absorption} and \cref{notn:algnumber} (1). 
	For $\tau\in\{\pm 1\}$, we apply \cref{lem:BCKlem3.2G} to the inclusions $L(1+\tau Y)\subset L\frac{1+\tau Y}{2}\cap L \subset L\frac{1+\tau Y}{2}$ and obtain $R$-submodules $L_{\tau,0},L_{\tau,1}\subset \bigoplus_{g\in G} L\frac{1+\tau Y}{2}[g]=\wt{L}\frac{1+\tau Y}{2}$ such that $\wt{L} \frac{1+\tau Y}{2} = L_{\tau,0}\oplus L_{\tau,1}$ and that 
	\begin{align*}
		\wt{L}\frac{1+\tau Y}{2}\cap\wt{L}
		=
		\bigoplus_{g\in G} \Bigl(L\frac{1+\tau Y}{2}\cap L\Bigr)[g] 
		= 
		2L_{\tau,0}\oplus L_{\tau,1} . 
	\end{align*}
	Since $Y$ acts on $\wt{L}\frac{1+\tau Y}{2}$ by the scalar $\tau\in R$, the $R$-submodules $L_{\tau,0}$ and $L_{\tau,1}$ are $R[\bZ/2\bZ]$-submodules. 
	Since $0=\wt{L}\frac{1+Y}{2}\cap \wt{L}\frac{1-Y}{2}$ by the idempotence of $\frac{1+\tau Y}{2}$, when we put $P_+:=L_{+1,1}$, $P_-:=L_{-1,1}$, and $P_0:=\wt{L}\cap (L_{+1,0}\oplus L_{-1,0})$, they satisfy $\wt{L}=P_+\oplus P_-\oplus P_0$ as an $R[\bZ/2\bZ]$-module and are projective as $R$-modules since $\wt{L}$ is a free $R$-module. 
	
	To see that $P_0$ is a projective $R[\bZ/2\bZ]$-module, it suffices to show $P_0\frac{1+\tau Y}{2}\cap P_0=P_0(1+\tau Y)$ for $\tau\in\{\pm1\}$ by \cref{lem:BKPlem2.5}. 
	For any $w\in P_0\frac{1+\tau Y}{2}\cap P_0\subset \wt{L}\frac{1+\tau Y}{2}\cap \wt{L}=2L_{\tau,0}\oplus L_{\tau,1}$, there are $x\in L_{\tau,0}$ and $y\in L_{\tau,1}$ such that $w=2x+y$. Since $P_0\ni w-2x=y \in L_{\tau,1}\subset P_+\oplus P_-$, we see $y=0$. 
	There are $a\in \wt{L}$, $z\in L_{-\tau,0}$, and $b\in L_{-\tau,1}$ such that $x=a\frac{1+\tau Y}{2}$ and $a\frac{1-\tau Y}{2}=z+b$. 
	Then $L_{+1,0}\oplus L_{-1,0}\ni x+z = a-b \in \wt{L}$, which implies $x+z\in P_0$. 
	Thus $w=2x=(x+z)(1+\tau Y)\in P_0(1+\tau Y)$. 
	This shows $P_0\frac{1+\tau Y}{2}\cap P_0 \subset P_0(1+\tau Y)$, while the converse inclusion is obvious. 
	
	Finally, by adding $R[\bZ/2\bZ]^{\oplus\infty}$ again if necessary, we obtain the desired result by \cref{rem:absorption} for $P_0$. 
\end{proof}

\begin{proof}[Proof of \cref{thm:Z/2module}]
	There is an $R[\bZ/2\bZ]$-linear surjection $f_0\colon R[\bZ/2\bZ]^{\oplus\infty}\to M$, and $\Ker f_0$ is a countable projective $R$-module by \cref{notn:algnumber} (1). 
	Then, by \cref{lem:BKPthm2.8} and \cref{rem:absorption}, we have $R_+^{\oplus\infty}\oplus R_-^{\oplus\infty}\oplus (\Ker f_0)[G] \oplus R[\bZ/2\bZ]^{\oplus\infty} \cong R_+^{\oplus\infty}\oplus R_-^{\oplus\infty}\oplus R[\bZ/2\bZ]^{\oplus\infty}$. 
	By adding $(R_+\oplus R_-\oplus (\Ker f_0)[G]\oplus R[\bZ/2\bZ])^{\oplus\infty}$, we obtain an exact sequence of $R[\bZ/2\bZ]$-modules of the form 
	\begin{align*}\begin{aligned}
			0&\to R_+^{\oplus\infty}\oplus R_-^{\oplus\infty}\oplus R[\bZ/2\bZ]^{\oplus\infty} 
			\xto{f_2} 
			R_+^{\oplus\infty}\oplus R_-^{\oplus\infty}\oplus R[\bZ/2\bZ]^{\oplus\infty} 
			\xto{f_1} \Ker f_0 \to 0 . 
	\end{aligned}\end{align*}
	Note that we have the $R[\bZ/2\bZ]$-linear exact sequence 
	\begin{align*}
		0\to R_+\xto{\iota} R[\bZ/2\bZ]\xto{p} R_-\to 0 , 
	\end{align*}
	where $\iota(x):=x(1+Y)$ and $p(y+zY)=y-z$ for $x,y,z\in R$. 
	We can take an $R[\bZ/2\bZ]$-linear endomorphism $h$ on $R_+^{\oplus\infty}\oplus R[\bZ/2\bZ]^{\oplus\infty}\oplus R[\bZ/2\bZ]^{\oplus\infty}$ such that 
	\begin{align}\label{eq:thm:Z/2module1}
		(\id\oplus p^{\oplus\infty}\oplus\id)h=f_2(\id\oplus p^{\oplus\infty}\oplus\id)
	\end{align}
	by using the free basis of $R[\bZ/2\bZ]^{\oplus\infty}$ and letting $h|_{R_+^{\oplus\infty}}:=f_2|_{R_+^{\oplus\infty}}$, which makes sense thanks to the non-existence of a non-zero $R[\bZ/2\bZ]$-linear map $R_+\to R_-$. 
	Then we can lift the sequence above to obtain an $R[\bZ/2\bZ]$-linear exact sequence 
	\begin{align}\label{eq:thm:Z/2module2}\begin{aligned}
			0 &\to L\xto{j} R_+^{\oplus\infty}\oplus R[\bZ/2\bZ]^{\oplus\infty}\oplus R[\bZ/2\bZ]^{\oplus\infty} \oplus R_+^{\oplus\infty}
			\\&
			\xto{h\oplus(0,\iota^{\oplus\infty},0)} 
			R_+^{\oplus\infty} \oplus R[\bZ/2\bZ]^{\oplus\infty} \oplus R[\bZ/2\bZ]^{\oplus\infty} 
			\\&
			\xto{f_1 \circ(\id\oplus p^{\oplus\infty}\oplus\id)} R[\bZ/2\bZ]^{\oplus\infty} 
			\xto{f_0} M \to 0 , 
	\end{aligned}\end{align}
	for some $R[\bZ/2\bZ]$-module $L$. 
	
	Suppose that there is $x\in L$ such that $xY\neq x$. Then $y:=xY-x$ satisfies $0\neq y= -yY$, and thus $j(y)$ is contained in the component $(R[\bZ/2\bZ]^{\oplus\infty})^{\oplus2}$, which implies $j(y)\in\Ker h$. 
	It follows from \cref{eq:thm:Z/2module1} and the injectivity of $f_2$ that $j(y)\in\Ker(p^{\oplus\infty}\oplus\id)=\iota(R_+)^{\oplus\infty}\oplus0$, which yields a contradiction as $j(y)Y=j(y)\neq -j(y)=j(yY)$. Thus $Y$ acts on $L$ by the scalar $1$, and since $L$ is an $R$-submodule of a free $R$-module and hence projective by \cref{notn:algnumber} (1), we obtain the desired exact sequence by adding $R_+^{\oplus\infty}\xto{\id}R_+^{\oplus\infty}$ to \cref{eq:thm:Z/2module2} at $j$ with the aid of \cref{rem:absorption}. 
\end{proof}

\providecommand{\bysame}{\leavevmode\hbox to3em{\hrulefill}\thinspace}
\providecommand{\MR}{\relax\ifhmode\unskip\space\fi MR }
\providecommand{\MRhref}[2]{%
	\href{http://www.ams.org/mathscinet-getitem?mr=#1}{#2}
}
\providecommand{\href}[2]{#2}

\end{document}